\documentclass[10pt]{amsart}

\usepackage{amssymb, amsmath, amsthm, amsfonts}
\usepackage{mathrsfs,comment}
\usepackage{graphicx}
\usepackage{placeins}
\usepackage{todonotes}
\usepackage{rotating}
\usepackage{tikz}
\usepackage{float}
\usepackage{hvfloat}
\usepackage{caption}
\usepackage{pdflscape}
\usepackage{url}
\usepackage[all,arc,2cell]{xy}
\UseAllTwocells
\usepackage{enumerate}
\usepackage{color}

\usepackage{sseq}

\usepackage{hyperref}
  \definecolor{dark-red}{rgb}{0.6,0.15,0.15}
   \definecolor{dark-blue}{rgb}{0.15,0.15,0.6}
   \definecolor{medium-blue}{rgb}{0,0,0.5}

\hypersetup{
    colorlinks, 
    linkcolor=dark-red,
    citecolor=dark-blue, urlcolor=medium-blue
}

\usepackage[nameinlink,capitalise,noabbrev]{cleveref}

\numberwithin{equation}{subsection}

\newtheorem{thm}{Theorem}[subsection]
\newtheorem{cor}{Corollary}[subsection]

\newtheorem{prop}{Proposition}[subsection]
\newtheorem{lem}{Lemma}[subsection]

\theoremstyle{definition}
\newtheorem{defn}{Definition}[subsection]

\newtheorem{notation}{Notation}[subsection]

\newtheorem{exmp}{Example}[subsection]

\newtheorem{rem}{Remark}[subsection]
\newtheorem{warn}{Warning}[subsection]

\newtheorem{thmintro}{Theorem}[subsection]
\newtheorem{conjintro}{Conjecture}[subsection]

\makeatletter
\let\c@equation=\c@thm
\let\c@lem=\c@thm
\let\c@cor=\c@thm
\let\c@conj=\c@thm
\let\c@prop=\c@thm
\let\c@lem=\c@thm
\let\c@defn=\c@thm
\let\c@notation=\c@thm
\let\c@note=\c@thm
\let\c@exmp=\c@thm
\let\c@ex=\c@thm
\let\c@exmps=\c@thm
\let\c@rem=\c@thm
\let\c@warn=\c@thm
\let\c@claim=\c@thm
\let\c@convention=\c@thm
\let\c@conventions=\c@thm
\let\c@quest=\c@thm
\let\c@thmintro=\c@thm
\let\c@conjintro=\c@thm
\let\c@thmbig=\c@thm
\let\c@conbig=\c@thm
\let\c@facts=\c@thm
\makeatother

\def\quickop#1{\expandafter\newcommand\csname #1\endcsname{\operatorname{#1}}}
\quickop{Hom} \quickop{End} \quickop{Aut} \quickop{Tel} \quickop{Mic} \quickop{map}
\quickop{Ext} \quickop{Tor} \quickop{Cotor} \quickop{Id} \quickop{Coker} \quickop{Ker}
\quickop{Lim} \quickop{Colim} \quickop{Holim} \quickop{Hocolim}
\quickop{id} \quickop{tel} \quickop{mic} \quickop{coker}
\quickop{colim} \quickop{holim} \quickop{hocolim} \quickop{im}

\numberwithin{equation}{subsection}
\newcommand{\co}{\mskip0.5mu\colon\thinspace}   

\newcommand{\F}{\mathbb{F}}
\newcommand{\R}{\mathbb{R}}
\newcommand{\bS}{\mathbb{S}}
\newcommand{\Sn}{\mathbb{S}}
\newcommand{\G}{\mathbb{G}}
\newcommand{\W}{\mathbb{W}}
\newcommand{\Q}{\mathbb{Q}}
\newcommand{\smsh}{\wedge}
\newcommand{\ra}{\rightarrow}
\newcommand{\xra}{\xrightarrow}

\newcommand{\ZZ}{\mathbb{Z}}
\newcommand{\FF}{\mathbb{F}}

\DeclareFontFamily{OMS}{rsfs}{\skewchar\font'60}
\DeclareFontShape{OMS}{rsfs}{m}{n}{<-5>rsfs5 <5-7>rsfs7 <7->rsfs10 }{}
\DeclareSymbolFont{rsfs}{OMS}{rsfs}{m}{n}
\DeclareSymbolFontAlphabet{\scr}{rsfs}

\newcommand{\sE}{\scr{E}}

\def\makeop#1{\expandafter\def\csname #1\endcsname{\mathop{\mathrm{#1}}\nolimits}}

\makeop{Gal}
\makeop{id}
\makeop{Mod}
\makeop{Tot}
\makeop{gr}
\makeop{Out}
\makeop{Hom}
\makeop{Ext}
\makeop{End}
\makeop{Aut}
\makeop{Tor}
\makeop{ev}
\def\FF{\mathbb{F}}
\def\QQ{\mathbb{Q}}
\def\GG{\mathbb{G}}
\def\SS{\mathbb{S}}
\def\WW{{{\mathbb{W}}}}
\def\ZZ{{{\mathbb{Z}}}}
\def\Z{{{\mathbb{Z}}}}
\def\Ext{\mathrm{Ext}}
\def\longr{{{\longrightarrow\ }}}

\def\myzeta{{{\zeta}}}
\def\myoneb{{{\chi}}}
\def\mytwob{{\widetilde{\chi}}}
\def\myonebK{{{\chi}}}
\def\myzetaK{{{\zeta}}}
\def\inproofchisquare{{{\gamma}}}
\def\inproofz{{{\varepsilon}}}
\def\inproofw{{{\omega}}}

\definecolor{darkspringgreen}{rgb}{0.09, 0.45, 0.27}

\newcommand{\assa}{\Delta_0} 
\newcommand{\assb}{b_0}  
\newcommand{\assc}{\overline{b}_0}  
\newcommand{\assd}{\overline{\Delta}_0}  

\newcommand{\la}{\Delta}  
\newcommand{\lb}{b} 
\newcommand{\lc}{\overline{b}} 
\newcommand{\ld}{\overline{\Delta}} 

\newcommand{\bock}[1]{\delta^{(#1)}}
\newcommand{\bockn}{\delta}

\newcommand{\ravclass}{\chi}

\DeclareRobustCommand\circled[1]{\tikz[baseline=(char.base)]{
   \node[shape=circle,draw,inner sep=0pt] (char) {#1};}}

\def\masseyp{{\langle \mytwob,2,\eta \rangle}}

\title[Chromatic splitting for the $K(2)$-local sphere at $p=2$]{Chromatic splitting for the 
$K(2)$-local sphere at $p=2$}
\date{\today}

\author[A. Beaudry]{Agn\`es Beaudry}
\address{Department of Mathematics\\ University of Colorado Boulder \\ 
\newline Campus Box 395 \\ Boulder \\ Colorado \\ 80309}

\author[P. G. Goerss]{Paul G. Goerss}
\address{Department of Mathematics\\ Northwestern University \\ \newline 2033 Sheridan Road 
\\ Evanston \\ Illinois \\ 60208 }

\author[H.-W. Henn]{Hans-Werner Henn}
\address{Institut de Recherche Math\'ematique Avanc\'ee, C.N.R.S. et 
Universit\'e de Strasbourg \\ \newline 7, rue Ren\'e Descartes \\ 67084 Strasbourg Cedex \\ France}

\thanks{This material is based upon work supported by ANR-16-CE40-0003 ChroK and by 
the National Science Foundation under grants No. DMS-1308916 and DMS-1612020/1725563}

\begin{document}
\maketitle

\begin{abstract}
We calculate the homotopy type of $L_1L_{K(2)}S^0$ and $L_{K(1)}L_{K(2)}S^0$ at the prime 2, 
where $L_{K(n)}$ is localization with respect to Morava $K$-theory and $L_1$ localization 
with respect to $2$-local $K$ theory. In  $L_1L_{K(2)}S^0$ we find all the summands predicted by the 
Chromatic Splitting Conjecture,
but we find some extra summands as well. An essential ingredient in our approach is the analysis of the
continuous group cohomology $H^\ast_c(\GG_2,E_0)$ where $\GG_2$ is the Morava stabilizer group and
$E_0 = \WW[[u_1]]$ is the ring of functions on the height $2$ Lubin-Tate space. We show that the inclusion
of the constants $\WW \to E_0$ induces an isomorphism on group cohomology, a radical simplification.
\end{abstract}
\setcounter{tocdepth}{1}
\tableofcontents


\section{Introduction}

\setcounter{subsection}{1}

The problem of understanding the homotopy groups of spheres has always been central to algebraic 
topology. A period of calculation beginning with Serre's computation of the cohomology of Eilenberg-
MacLane spaces  and Toda's work with the EHP sequence culminated, in the late 1970s, with the work of
Miller, Ravenel, and Wilson on periodic  phenomena in the homotopy groups of spheres and Ravenel's 
nilpotence conjectures. The solutions to most of these conjectures 
in the middle 1980s established the primacy of the chromatic point of view, which uses the algebraic 
geometry of smooth $1$-parameter formal groups to organize the search for large scale phenomena in 
stable homotopy theory. 

This has been remarkably successful. Much of what we know about stable homotopy theory can be 
motivated and conjectured by analyzing the moduli stack of formal groups and its quasi-coherent sheaves. 
See, for example, the table in Section 2 of \cite{HopkinsGross}. 
In particular, this stack has a stratification by height and working along this stratification highlights two 
distinct lines of research. First, we'd like to discover all that can be learned by working at a single height; 
or, put another way, we make calculations in $K(n)$-local homotopy theory. Second, we need to assemble 
the information from different heights. This is the chromatic assembly problem.
  
In this paper, we give an analysis of the Chromatic Splitting Conjecture of Hopkins at $p=2$ and $n=2$. 
At first glance, this is a chromatic assembly question, but as we proceed here we need 
extensive information from height 2 calculations. Thus, the questions of single-height calculations and
chromatic assembly remain closely related. 

The Chromatic Splitting Conjecture predicts a splitting of $L_1L_{K(2)}S^0$. 
We do get a splitting, and it contains the expected summands,  but it contains other summands as well. 
That there has to be more was already proved in \cite{Paper5}, which in turn builds on the papers
\cite{Paper1} and \cite{Paper2}, all by the first author. There are hints at extra summands in
the work of Shimomura and Wang \cite{MR1702307, MR1935487} as well. This is discussed in \Cref{rem:shimwang}.

Fix a prime $p$ and let $K(n)$ be the $n$th Morava $K$-theory spectrum; by convention,
$K(0)=H\QQ$, the rational  homology spectrum. Let $L_n$ be localization with respect to the homology
theory represented by $K(0) \vee K(1) \vee \cdots \vee K(n)$. This is the same as localization
with respect to the Johnson-Wilson theory $E(n)$. Then for all spectra $X$ there is a
natural localization map $L_nX \to L_{n-1}X \simeq L_{n-1}L_nX$. The Hopkins-Ravenel Chromatic Convergence Theorem of \S 8.6 of
\cite{RavNil} then says that if $X$ is a finite CW spectrum the induced map
\[
X \longr \holim L_nX
\]
is localization at the homology theory $H_\ast(-,\ZZ_{(p)})$. 

Next, let $L_{K(n)}X$ be localization with respect to $K(n)$. 
Then the $K(n)$-localization of $L_nX$ is a map $L_nX \to L_{K(n)}X$  and, for any spectrum $X$, the natural map $L_nX \to L_{n-1}X$ can be 
recovered from the {\it chromatic  fracture square}
\begin{equation}\label{eq:intro-chrom-sq}
\xymatrix{
L_nX \ar[r]\ar[d] & L_{K(n)}X\ar[d]\\
L_{n-1}X \ar[r] & L_{n-1}L_{K(n)}X.
}
\end{equation}  
That this square is a homotopy pull-back can be found in Theorem 6.19 of \cite{HovStrick}, 
but is implicit in \cite{Rav84} as well.

The chromatic fracture square and chromatic convergence together imply that we can recover the 
homotopy type of a finite CW spectrum from $L_{K(n)}X$ for all $n$ and all $p$, provided we can complete 
the assembly process of \eqref{eq:intro-chrom-sq}.

Calculations of $L_{K(n)}X$ usually come down to analysis of the $K(n)$-local Adams-Novikov
Spectral Sequence. The cohomology theory $K(n)^\ast$ is complex orientable and the associated formal 
group $\Gamma_n$ is of height $n$. Let $\GG_n$ be the automorphisms of the pair 
$(\FF_{p^n},\Gamma_n)$ and let $E_n$ be the Morava $E$-theory associated to the pair 
$(\FF_{p^n},\Gamma_n)$. By the Goerss-Hopkins-Miller Theorem the profinite group $\GG_n$ acts on $E_n$, 
and hence on $(E_n)_\ast X := \pi_\ast L_{K(n)}(E_n \wedge X)$. If $X$ is a finite CW spectrum, we then
have  a spectral sequence 
\begin{equation}\label{eq:ANSS-intro}
H_c^s(\GG_n,(E_n)_tX) \Longrightarrow \pi_{t-s}L_{K(n)}X
\end{equation}
where cohomology is continuous group cohomology. Much of $K(n)$-local  homotopy theory 
comes down to the analysis of the group $\GG_n$ and the action of $\GG_n$ on $(E_n)_\ast$. 
We give some more details in \Cref{sec:prelims} and add references to the large literature 
on the subject there. 

Even assuming we can master the $K(n)$-local calculations, there remains the assembly question.
Let $M_nX$ be the fiber of $L_nX \to L_{n-1}X$; then the key result needed to establish the
chromatic fracture square of \eqref{eq:intro-chrom-sq} is that the induced map $M_nX \to M_nL_{K(n)}X$ is an
equivalence. Crucial to assembly is the other fiber; that is, the fiber of 
$L_{n-1}X \to L_{n-1}L_{K(n)}X$. This is the main subject of this paper. In fact,
we investigate the homotopy type of the map $L_{n-1}X \to L_{n-1}L_{K(n)}X$. 
The Chromatic Splitting Conjecture, due to Hopkins, is a very specific conjecture about this map.
We will get into this below; before that, however, 
we will state our main results. 

Let $p=2$ and write $E=E_2$ for our choice of Morava $E$-theory at height $2$. Let $\WW$ be the Witt
vectors on $\FF_4$. Then there is an isomorphism $E_\ast \cong \WW[[u_1]][u^{\pm 1}]$
where the power series ring is  in degree $0$ and $u$ has degree $-2$.
Our first result is a $K(2)$-local result. See \Cref{thm:consts21} and \Cref{thm:const} 

\begin{thmintro}\label{thm:D} The inclusion of constants into the power series ring induce
isomorphisms in group cohomology
\[
H_c^*(\GG_2,\FF_4) \to H_c^*(\GG_2, E_0/2)
\]
and
\[
H_c^*(\GG_2,\WW) \to H_c^*(\GG_2, E_0).
\]
\end{thmintro} 

This is a remarkable simplification. We conjecture that the analogous result is true at all heights and all
primes. 
\begin{conjintro}[Chromatic Vanishing Conjecture]\label{conj:chrom-vanish}
For any prime $p$ and height $n$, the inclusion of constants into the power series ring induces 
isomorphisms in group cohomology
\[
H_c^*(\GG_n,\FF_{p^n}) \to H_c^*(\GG_n, \pi_0E_n/p)
\]
and
\[
H_c^*(\GG_n,\WW) \to H_c^*(\GG_n, \pi_0E_n)
\]
\end{conjintro}
\noindent
We call this the Chromatic Vanishing Conjecture as it implies and is implied by the vanishing of the cohomology groups $H_c^*(\GG_n, (\pi_0E_n/p)/\F_{p^n})$  and $H_c^*(\GG_n,(\pi_0E_n)/\W)$ in all degrees.

This conjecture is true wherever it has been checked; that is, for $n \leq 2$ and all primes. If $n=1$ this is a 
tautology. For $n=2$ and $p> 3$, it can be deduced from \cite{shimyab} 
(see also Corollaire 4.5 of \cite{Lader}). This basic case was also proved
later in \cite{Kohlhaase} using different techniques. 
For $n=2$ and $p=3$ it can be deduced from \cite{HKM}; 
\cite{GoerssSplit}. The primes $2$ and $3$ are harder, as the
group $\GG_2$ contains $p$-torsion subgroups. 

The next step is to calculate differentials in the Adams-Novikov spectral sequence 
\eqref{eq:ANSS-intro} for $X=S^0$. By an old result of Lazard, applied by Morava in our case \cite{lazard_analyticgroups,morava}, we know that for all $n$ 
and $p$ the cohomology ring $H_c^*(\G_n, \WW)\otimes \QQ$ is an exterior algebra on $n$ generators 
of degree $2i-1$ for $1\leq i\leq n$; then 
\Cref{thm:D} implies $\pi_i L_{K(2)}S^0$ has a torsion-free generator when $i=0$, $-1$, $-3$, or $-4$. To
get further, we need to get some control on the torsion. 

For all $n$ and $p$, the group $\GG_n$ comes equipped with a  determinant map
\[
\det \colon \GG_n \to \ZZ_p^\times
\]
to the units in the $p$-adic integers. If $p=2$, there is an isomorphism 
$\ZZ_2^\times \cong \ZZ_2 \times C_2$,
where $C_2 = \{ \pm 1\}$ is the cyclic group of order $2$. We thus get a map
\[
E(\zeta_2) \otimes \FF_2[\myoneb] \cong H_c^\ast (\ZZ_2 \times C_2,\FF_2) \to 
H_c^\ast(\GG_2,\FF_2).
\]
Here $E(-)$  denotes the exterior algebra  
over $\FF_2$. These cohomology classes will be  discussed in \Cref{rem:hone}.  
We will show in \Cref{prop:b0cubedbig} (but see also Theorem 6.3.24 of
\cite{ravgreen}), that this map induces an injection
\[
\xymatrix{
E(\zeta_2) \otimes \FF_2[\myoneb]/(\myoneb^3) \ar[r]^-{\subseteq} &H_c^\ast(\GG_2,\FF_2).
}
\] 
Here is one place when the prime $2$ is fundamentally different. 
At odd primes, $\ZZ_p^\times \cong \ZZ_p \times C_{p-1}$, where $C_{p-1}$ is a cyclic group
of order $p-1$. Thus, there is an isomorphism $H_c^\ast(\ZZ_p^\times,\FF_p) \cong E(\zeta_2)$. The class
$\myoneb$ only appears at $p=2$.

The class $\zeta_2$ is the reduction of a class of infinite order in $H_c^1(\GG_2,\ZZ_2)$; that this class is
a permanent cycle in the Adams-Novikov Spectral Sequence is well understood. In a standard abuse
of notation we also write $\zeta_2 \in \pi_{-1}L_{K(2)}S^0$ for a particular homotopy class
detected by the cohomology class $\zeta_2$. See \cite{DH} and \Cref{prop:zeta-permanent} for details.
We will show that the Bockstein on the class  $\myoneb$ is also a permanent cycle 
in $H_c^2(\GG_2,\ZZ_2)$ and detects a class $x \in \pi_{-2}L_{K(2)}S^0$ of order $2$. 
We will then show that the Bockstein on
$\zeta_2\myoneb$ detects the class $\zeta_2x\in \pi_{-3}L_{K(2)}S^0$, also of order $2$.
 
Let $V(0)$ be the mod $2$ Moore spectrum and let $\iota:S^0 \to L_{K(2)}S^0$ be the unit.
The classes $\iota$, $\zeta_2$, $x$ and $\zeta_2x$ together with choices for the torsion free 
generators of $\pi_iL_{K(2)}S^0$ for $i=-3$ and $i=-4$ can be used to define a map
\begin{equation}\label{eq:the-map-f-intro}
f: S^0 \vee S^{-1} \vee S^{-3} \vee  S^{-4} \vee \Sigma^{-2}V(0)\vee \Sigma^{-3}V(0) \longr L_{K(2)}S^0
\end{equation}

Our main result then describes chromatic splitting at $n=p=2$. Let $S^n_2$ denote the $2$-completed
$n$-sphere.

\begin{thmintro}\label{thm:main} The map $f$ defines a weak equivalence
\[
L_1(S_2^0 \vee S_2^{-1}) \vee L_0(S_2^{-3} \vee S_2^{-4}) \vee L_1(\Sigma^{-2}V(0)\vee \Sigma^{-3}V(0))
\simeq L_1L_{K(2)}S^0.
\] 
In particular, $L_1S_2^0 \ra L_1L_{K(2)}S^0$ is a split inclusion.
\end{thmintro}
Proving that the restriction of $f$ to $S^{-3} \vee  S^{-4}$ factors through $L_0(S_2^{-3} \vee  S_2^{-4})$ is part of the work necessary to prove the result. The spherical summands are predicted by the Chromatic Splitting Conjecture, 
but the Moore spectrum summands are the new phenomenon. 

\begin{rem}\label{rem:shimwang}
Shimomura  \cite{MR1702307} computes the Adams-Novikov $E_2$-term for $L_2V(0)$ and, with Wang \cite{MR1935487}, for $L_2S^0$ using the chromatic spectral sequence. For this reason, we cannot give a precise dictionary between our results. However, in Theorem 2.7 of \cite{MR1702307}, the class $\zeta_1$\footnote{In our paper, $\zeta_1$ has a different meaning. We do not refer to Shimomura and Wang's computations again so there will be no source for confusion.} corresponds to the class we call $\chi$ later in this paper and gives rise to the Moore spectrum summands that are not predicted by the Strong Chromatic Splitting Conjecture. However, Shimomura and Wang do not compute the differentials in theses spectral sequences and their work settles neither the Strong nor Weak Chromatic Splitting Conjecture.
\end{rem}

The strategy for proving \Cref{thm:main} is to use the chromatic fracture square 
\eqref{eq:intro-chrom-sq} to deduce the homotopy type $L_1L_{K(2)}S^0$. More precisely, 
we prove the following rational result. See \Cref{prop:ratcohE} and \Cref{thm:mainrat}.

\begin{thmintro}\label{thm:B} Let $S^n_2$ be the $2$-complete $n$-sphere. Then the map $f$
of \eqref{eq:the-map-f-intro} induces a weak equivalence
\[
L_0(S^0_2 \vee S^{-1}_2 \vee S^{-3}_2 \vee S^{-4}_2) \simeq L_{0}L_{K(2)}S^0.
\]
\end{thmintro}

Furthermore, we prove the following $K(1)$-local result. See \Cref{thm:the-big-one}.

\begin{thmintro}\label{thm:C} The restriction of the map $f$ given in \eqref{eq:the-map-f-intro} to the wedge summand
$S^0\vee S^{-1} \vee \Sigma^{-2}V(0) \vee \Sigma^{-3}V(0)$ induces a weak equivalence
\[
L_{K(1)}\big(S^0 \vee S^{-1}  \vee \Sigma^{-2}V(0) \vee \Sigma^{-3}V(0)\big)
\simeq L_{K(1)}L_{K(2)}S^0.
\]
\end{thmintro}

Much of the work in this paper goes into this last theorem. \Cref{thm:main} then follows exactly as in
the prime $3$ case; see \cite{GoerssSplit} and \Cref{thm:mainrat}. 

We conclude this introduction by revisiting the Chromatic Splitting Conjecture. This is due to Hopkins
and can be found in the literature in \cite{cschov}. As we mentioned above, there is a result of 
Morava and Lazard
that the cohomology ring $H_c^*(\G_n, \WW)\otimes \QQ$ is an exterior algebra on $n$ generators 
of degree $2i-1$ for $1\leq i\leq n$. Part of the Chromatic Splitting Conjecture is that, for some choice of 
generators $x_i \in H_c^{2i-1}(\GG_n,\WW)$, the exterior algebra $E_{\ZZ_p}(x_1, \ldots, x_n)$ over $\ZZ_p$
maps non-trivially to permanent cycles in $H_c^*(\GG_n, (E_n)_*)$. This would give a map out of a wedge of 
spheres to $L_{K(n)}S^0$ indexed on the monomial basis of the exterior algebra. The following
would completely describe the gluing data \eqref{eq:intro-chrom-sq}.

\begin{conjintro}[Strong Chromatic Splitting Conjecture]\label{conj:strong} 
This map out of the wedge of spheres induces an equivalence
\[
L_{n-1}S_p^0 \vee \bigvee_{1\leq i_1 <\ldots < i_j\leq n} L_{n-i_j}S_p^{-2\left(\sum i_k \right)+j}
\simeq L_{n-1}L_{K(n)}S^0\ .
\]
\end{conjintro}

The conjecture is known to hold for $n=1$ if $p\geq 2$ and for $n=2$ if $p\geq 3$
(\cite{behse2}, \cite{GoerssSplit}, \cite{shimyab}). However, the results of \cite{Paper5} 
already implies that \Cref{conj:strong} 
does not hold when $n=p=2$ and \Cref{thm:main} makes this completely precise.
We note that although \Cref{conj:strong} does not hold at $p=2$, there is no evidence that it should fail at odd primes, 
or at least when $p$ is large with respect to $n$.
When $p=2$, our computations at $n=2$ beckon a reformulation of the Strong Chromatic Splitting Conjecture. We refer the reader to \cite[Section 5.6]{BarthelBeaudry} for a discussion on this topic.

We can also write down a weaker form of the conjecture, which holds for all $p$ and $n\leq 2$; 
that is, for all cases where we've been able to check. 

\begin{conjintro}[Weak Chromatic Splitting Conjecture]\label{conj:weak}
If $X$ is the $p$-completion of a finite spectrum, the map $L_{n-1}X \to L_{n-1}L_{K(n)}X$ is a split inclusion.
\end{conjintro}

This second conjecture would imply 
that there are maps $L_{K(n)}S^0 \ra L_{K(n-1)}S^0$ such that $S^0 \simeq \holim_n L_{K(n)}S^0$; that is,
$S^0$ can be recovered from its Morava $K$-theory localizations. 

\subsection*{Organization of the paper}

In \Cref{sec:prelims}, we review some of the background from chromatic homotopy theory, including some
of the more technical techniques at the prime $p=2$. In \Cref{sec:recollections}, 
we recall some classical results  from homotopy theory; most of this can be summarized in the remark that
extra care is needed  because the order of the identity of the mod $2$ Moore spectrum is not equal to two.
Then in \Cref{sec:k1-local-calcs} we give a detailed review of  $K(1)$-local computations at the prime $p=2$
which are used in later sections. In \Cref{sec:height2}, we begin  our analysis of the case $n=p=2$. We describe
the cohomology  of various subgroups $G$ of $ \GG_2$. This is preliminary to \Cref{sec:c-v-section},
where we give the proof of \Cref{thm:D} in \Cref{thm:const}. \Cref{sec:d3cycle} is dedicated 
to one of the key technical results of the paper: the class $\myoneb$ is a $d_3$-cycle in the  
$K(2)$-local Adams-Novikov Spectral Sequence for the Moore spectrum.  \Cref{sec:k1k2s} 
contains the proof of \Cref{thm:C}. This theorem is shown by studying the $v_1$-localized  Adams-Novikov
Spectral Sequences  computing $L_{K(1)}L_{K(2)}V(0)$ and especially  $L_{K(1)}L_{K(2)}Y$ where
$Y  = V(0) \smsh C(\eta)$ for $C(\eta)$ the cofiber of $\eta \colon S^1 \to S^0$. The spectrum $Y$ was used in Mahowald's proof of the Telescope Conjecture 
at $p=2$ and $n=1$. See \cite {MahImJ}. We also deduce  \Cref{thm:main} and \Cref{thm:B} 
at the end of \Cref{sec:k1k2s}.

We have been slightly disingenuous in our presentation above, as we actually prove \Cref{thm:C}
before constructing the map $f$ of \eqref{eq:the-map-f-intro}; that is, we only construct the
$K(1)$-localization of the map $f$ in \Cref{sec:k1k2s}. In \Cref{sec:K2locy}
we complete the difficult task of constructing $f$.

\subsection*{Acknowledgements} This project had its genesis in conversations  with Mark Mahowald,
who was trying to come to terms with calculations of Shimomura and Wang \cite{MR1935487}. Specifically,
Mark thought that those authors had identified $v_1$-torsion-free summands in the $E_2$-term of the
Adams-Novikov Spectral Sequence for $\pi_\ast L_{K(2)}V(0)$ which could not be explained by the 
Chromatic Splitting Conjecture. In some sense this entire paper, as well as \cite{Paper5} and 
\cite{Paper2}, are an attempt to ratify and explain this insight. 

Careful readers of \Cref{sec:d3cycle} below will recognize that the techniques and ideas there are 
completely different from the rest of the paper. This lateral move arises from an insight of 
Mike Hopkins: namely, that the isomorphism of \Cref{thm:D} could be extended to a homomorphism
of homotopy fixed point spectral sequences. See the beginning of \Cref{sec:d3cycle} for more details.
We don't completely prove that, but we do get enough information from 
this idea to prove our key \Cref{prop:b0d3cycle}. We extend heartfelt thanks to Hopkins for sharing this idea.  

We also thank Mark Behrens, Irina Bobkova, Lars Hesselholt, Peter May and
Zhouli Xu for useful conversations along the way, as well as two anonymous referees for their
helpful comments.

Finally, this work has taken place over a number years and at a number of places. We thank 
the Hausdorff Institute of Mathematics, the Universit\'e de Strasbourg, and the University of Colorado Boulder 
all for providing such wonderful places to work. 


\section{Preliminaries}\label{sec:prelims}

We begin by introducing the $K(n)$-local category, Morava $E$-theory, the Morava stabilizer group,
and general convergence results for the $K(n)$-local Adams-Novikov Spectral Sequence.
We then get specific at $n=2$ and $p=2$, discussing the role of formal groups from supersingular
elliptic curves. We close the section with some background 
on algebraic and topological duality resolutions.

\subsection{The $K(n)$-local category}

Fix a prime $p$. Let $\Gamma_n$ be a formal group of height $n$ over the finite field $\F_p$ 
of $p$ elements and let $\End(\Gamma_n/\FF_p)$ be the endomorphism ring of $\Gamma_n$ over
$\FF_p$. This ring is in fact the maximal order in a central division algebra over $\Q_p$ of Hasse invariant $1/n$ (see for example Appendix A.2 of \cite{ravgreen}).
 The unique map of rings $\ZZ_p \to \End(\Gamma_n/\FF_p)$ is an inclusion into the center. Because
$\Gamma_n$ is defined over $\FF_p$ the Frobenius map $\xi(x) = x^p$ also defines an endomorphism of $\Gamma_n$. 
We will assume the endomorphism 
$\xi^n(x) = x^{p^n}$ satisfies an equation
\begin{equation}\label{eq:frob-assumption}
\xi^n (x)= [ap](x) \in \End_{\FF_{p^n}}(\Gamma_n/\FF_p)
\end{equation}
where $a \in \ZZ_p^\times$ is a unit and $[\ell](x)$ denotes the $\ell$-series for $\ell \in \Z_p$. The Honda formal group of height $n$ satisfies these criteria: 
this has a formal group law which is  $p$-typical and with $p$-series $[p](x)=x^{p^n} = \xi^n(x)$. 
However, if $n=2$ and $p=2$, then the formal group of a 
supersingular elliptic curve defined over $\F_2$ will also do, and this will be our preferred choice if $p=2$;
see \Cref{sec:ellback}.

Let $i \colon \F_p \hookrightarrow k$ be  any extension of $\F_p$ and let 
$\Aut(\Gamma_n/k)$ be the group 
of automorphisms of $i^\ast \Gamma_n$ over $k$. Our
assumption \eqref{eq:frob-assumption} implies that for any extension 
$\F_{p^n} \subseteq k$ there is an isomorphism
\[
\Aut(\Gamma_n/\F_{p^n}) \mathop{\longrightarrow}^{\cong} \Aut(\Gamma_n/k).
\]
To shorten notation we define
\begin{equation}\label{morava-stab}
\Sn_n = \Aut(\Gamma_n/\F_{p^n}).
\end{equation}
If we choose a coordinate for $\Gamma_n$, then every element of $\Sn_n$ can be expressed as a power series
$\phi(x) \in x\F_{p^n}[[x]]$ invertible under composition. The map $\phi(x) \mapsto \phi'(0)$ defines a surjective map
\[
\Sn_n \longrightarrow\F_{p^n}^\times.
\]
We define $S_n$ to be the kernel of this map; this is the $p$-Sylow subgroup of the profinite
group $\Sn_n$.\footnote{Here, there is a small clash in the notation as $S_p$ sometimes also denotes the $p$-completed sphere spectrum. However, both the notation $S_n$ for the $p$-Sylow subgroup of the Morava Stabilizer group and $S_p$ for the $p$-completed sphere are well established and we have decided not to change either. We do not think this will cause any confusion as what we mean is clear from context.}
The Teichm\"uller lift $\F_{p^n}^{\times} \to \WW^{\times} \subseteq \mathbb{S}_n$ defines a section of this map, giving an isomorphism $S_n \rtimes \F_{p^n}^\times \cong \Sn_n$.

Define the  extended Morava stabilizer group $\G_n$ as the automorphism group of the pair 
$(\F_{p^n},\Gamma_n)$. Elements of $\G_n$ are pairs $(f,\phi)$ where $f \in \Aut(\F_{p^n}/\F_p)$ 
and $\phi\colon \Gamma_n \to f^*\Gamma_n$ is an isomorphism of formal groups.
Since $\Gamma_n$ is defined over $\F_p$, there is an isomorphism
\begin{equation}\label{big-morava-stab}
\G_n \cong \Aut(\Gamma_n/\F_{p^n}) \rtimes \Gal(\F_{p^n}/\F_p) = \Sn_n \rtimes \Gal(\F_{p^n}/\F_p).
\end{equation}

We next define Morava $K$-theory; there are many variants, all of which have the same Bousfield
class and define the same localization. To be specific, let
$K(n) = K(\F_{p^n},\Gamma_n)$ be the $2$-periodic ring spectrum with homotopy groups
\[
K(n)_\ast = \F_{p^n}[u^{\pm 1}]
\]
and with associated formal group $\Gamma_n$. Here the class $u$ is in degree $-2$. This slightly
unclassical choice of $K(n)$ has the property that it receives a map $E_n \to K(n)$ from Morava
$E$-theory defined below in (\ref{lubin-tate}). 

We will spend a great deal of time working in the $K(n)$-local category and, when doing so, all our
spectra will implicitly be localized. In particular, we emphasize that we will
write $X \wedge Y$ for $L_{K(n)}(X \wedge Y)$, as this is the smash product
internal to the  $K(n)$-local category. 

We now define the Lubin-Tate spectrum $E=E_n=E(\F_{p^n},\Gamma_n)$. This is a complex
oriented, Landweber exact, $2$-periodic, $E_\infty$-ring spectrum with
\begin{equation}\label{lubin-tate}
E_\ast = (E_n)_\ast \cong \W[[u_1,\cdots,u_{n-1}]][u^{\pm1 }]
\end{equation}
with $u_i$ in degree $0$ and $u$ in degree $-2$. Here $\W = W(\F_{p^n})$ is the ring of Witt
vectors on $\F_{p^n}$. Note that
$E_0$ is a complete local ring with residue field $\F_{p^n}$;
the formal group over $E_0$ is a choice of universal deformation
of the formal group $\Gamma_n$ over $\F_{p^n}$. (We will be specific about this choice
at $n=p=2$ below in \Cref{sec:ellback}.) 
By the Goerss-Hopkins-Miller theorem, the group $\G_n =    \Aut(\Gamma_n) \rtimes \Gal(\F_{p^n}/\F_p)$
acts on $E$ by maps of $E_{\infty}$-ring spectra \cite{rezk_hm, gh_modulien}.

\begin{defn}
 If $X$ is a spectrum we let 
\[
E_\ast X= \pi_\ast L_{K(n)}(E \wedge X).
\]
\end{defn}

\begin{warn}
Despite the notation, the functor $E_\ast(-)$ is not a homology theory, as it does not take wedges 
to sums in general. In \cite{HovStrick}, this is denoted by $E_\ast^{\vee}(-)$. Since it is our most important algebraic invariant intrinsic 
to the $K(n)$-local category, we use the simpler notation $E_\ast(-)$. Note that $ E \wedge X \simeq L_{K(n)}(E \wedge X)$ whenever $X$ is a finite spectrum. However, $E_*E$ which we have defined to be $\pi_*L_{K(n)}(E \wedge E)$ is not isomorphic to $\pi_*(E\wedge E)$.
\end{warn}

The $E_\ast$-module $E_\ast X$ is equipped with the $\mathfrak{m}$-adic topology 
where $\mathfrak{m}$ is the maximal ideal in $E_0$.
This topology
is always complete, but need not be separated in general. However, all the
$E_\ast$-modules we consider in this paper will be complete and separated. For instance, sufficient conditions for the topology on $E_*X$ to be complete and separated are
\begin{enumerate}[(a)]
\item if $E_\ast X$ is finitely generated as an $E_*$-module, or 
\item if $K(n)_*X$ is concentrated in even degrees (Proposition 2.2 of \cite{ghmr}).
\end{enumerate}
For all the spectra $X$ we consider in this paper, one of these two condition holds.

The action of $\GG_n$ on $E_n$ determines a continuous action of $\GG_n$ on $(E_n)_\ast X$.
If $n=1$, we can choose $E_\ast = K_\ast$, $p$-complete $K$-theory, and $\GG_1 = \ZZ_p^\times$, 
the units in the $p$-adics. The action is then through the Adams operations and in that case
we might write $\psi^\ell$ for the action of $\ell \in \ZZ_p^\times$; for example, as in \eqref{eq:adamsop}.
Not withstanding this, as general rule we will simply write $g_\ast$ for the action of $g \in \GG_n$
on $E_\ast X$. 

This  action is twisted in the sense that
if $g \in \GG_n$, $a \in E_\ast$ and $x \in E_\ast X$, then $g_\ast(ax) = g_\ast(a)g_\ast(x)$.
We will call such modules either \emph{Morava modules} or \emph{twisted $E_*$-$\G_n$-modules}. 
When we consider closed subgroups $H$ of $\GG_n$ and $E_*$-modules with an action of $H$ 
satisfying the analogous formula for $h\in H$ then we call such modules 
\emph{twisted $E_*$-$H$-modules}. See Section 1.3 \cite{BobkovaGoerss} for a more lengthy discussion on Morava modules.

The $E_\ast$-algebra $E_\ast E$ has a $\GG_n$-action on both the left 
and right factor. The action of the left
factor defines the Morava module structure. Using the action on the right factor we get a composite map
\[
\xymatrix{
E_\ast E \times \GG_n \ar[r] & E_\ast E \ar[r]^-m & E_\ast
}
\]
where $m$ is induced by the multiplication $E \wedge E \to E$. The adjoint to this map is an isomorphism 
\begin{equation}\label{eq:EhomofE}
E_\ast E \cong \map_{cts}(\G_n,E_\ast)
\end{equation}
of Morava modules. Here $\map_{cts}(-,-)$ denotes the set of continuous maps. On the right hand side of this
equation, $E_\ast$ acts on the target and the $\G_n$-action is the diagonal action
given by $(g_\ast \phi)(x) = g_\ast \phi(g_\ast^{-1}x)$.

Caution is needed here. The isomorphism \eqref{eq:EhomofE} need not hold for the Lubin-Tate homology
theory $E(k,\Gamma)$ for an arbitrary height $n$ formal group $\Gamma$ over a field $k \subseteq \overline{\FF}_p$. 
In the literature \eqref{eq:EhomofE}  is proved for the 
Honda formal group; see, for example, Theorem 12 of \cite{StrickGrossHop, DH, hovey_ops}. In examining the proof there
we see that what is needed is our assumption from \eqref{eq:frob-assumption}. The details needed to then produce the
isomorphism of \eqref{eq:EhomofE},
and more information as well, can be found in \S 5 of \cite{HWHennCentAt2}.

Now suppose $X$ is a finite $p$-local spectrum. From \eqref{eq:EhomofE} it follows (again see \cite{StrickGrossHop})
that the $K(n)$-local $E$-based Adams Spectral Sequence for $X$ has the form
\begin{equation}\label{ANSS-gen}
H_c^s(\GG_n,E_tX) \Longrightarrow \pi_{t-s}L_{K(n)}X.
\end{equation}
Group cohomology here is {\it continuous} group cohomology. There is extensive discussion
of this spectral sequence in \cite{DH}. 

Complex orientations define maps of ring spectra
\[
\xymatrix{
BP & \ar[l] MU \ar[r] & E\ .
}
\]
If we localize at a prime and if $X$ is a finite $p$-local spectrum, we get a diagram of spectral sequences 
where the upward arrows are isomorphisms
\begin{equation}\label{eq:compare-ANSS}
\xymatrix@R=15pt{
\Ext^{\ast,\ast}_{BP_\ast BP}(BP_\ast,BP_\ast X) \ar@{=>}[r] & \pi_\ast X\\
\Ext^{\ast,\ast}_{MU_\ast MU}(MU_\ast,MU_\ast X) \ar[u]^\cong\ar[d] \ar@{=>}[r] 
& \pi_\ast X \ar[u]_\cong\ar[d]\\
H_c^\ast(\GG_n,E_\ast X) \ar@{=>}[r] & \pi_\ast L_{K(n)}X\ .
}
\end{equation}

For example, let $F \subseteq \G_n$ be a closed subgroup. In \cite{DH} Devinatz and 
Hopkins defined and studied a homotopy fixed point spectrum $E^{hF}$ with the property that 
the isomorphism of
\eqref{eq:EhomofE} descends to an isomorphism of Morava modules
\begin{equation}\label{eq:EhomofF}
E_\ast E^{hF} \cong \map_{cts}(\G_n/F,E_\ast)
\end{equation}
and a spectral sequence for $X$ a finite CW spectrum
\begin{equation}\label{ANSS}
H_c^s(F,E_tX) \Longrightarrow \pi_{t-s}(E^{hF}\wedge X)\ .
\end{equation}
If $F = \GG_n$ itself this spectral sequence is the $K(n)$-local
Adams-Novikov Spectral Sequence for $L_{K(n)}S^0$ and, if $F$ is finite, this spectral sequence
is the homotopy fixed point spectral sequence for the action of $F$ on $E$.
The Devinatz-Hopkins
construction is natural in $F$ and, in particular, if $F_1 \subseteq F_2$ are two closed subgroups, we have
a commutative diagram of spectral sequences
\begin{equation}\label{eq:more-compare-ANSS}
\xymatrix{
H_c^s(F_2,E_t X)\ar[d] \ar@{=>}[r] & \pi_{t-s}(E^{hF_2} \wedge X)\ar[d] \\
H_c^s(F_1,E_t X)\ar@{=>}[r] & \pi_{t-s}(E^{hF_1} \wedge X)
}
\end{equation}
where the vertical arrows are the natural maps induced by the inclusion of $F_1$ into $F_2$.

Again, some caution is needed here, as the Devinatz-Hopkins paper works only with
the Lubin-Tate theory defined for the Honda formal group. However, a close reading of \cite{DH}
shows that they need not have been so specific. They use only the  Goerss-Hopkins-Miller theorem \cite{gh_modulien}, 
which states that the space of $E_\infty$-self maps of $E(k,\Gamma_n)$ is homotopically discrete, 
as well as the isomorphism
of \eqref{eq:EhomofE}. The Goerss-Hopkins-Miller theorem holds for any Lubin-Tate spectrum
$E(k,\Gamma)$ where $k \subseteq \overline{\FF}_p$ and $\Gamma$ is of finite height. In addition,
\eqref{eq:EhomofE} holds for the formal group of \Cref{sec:ellback}. This caution will come up
again below, but we won't repeat this comment.

\subsection{Subgroups of $\G_n$}\label{subgroupsection} 
Various subgroups of $\G_n$ will play a role in this paper; we discuss some here. Some finite
subgroups will appear later, in \Cref{sec:ellback}, after we have introduced our choice of formal group. 

The right action of $\SS_n = \Aut(\Gamma_n/\FF_{p^n})$ on $\End(\Gamma_n/\FF_{p^n})$ 
defines a determinant map $\det\colon \SS_n \to \Z_p^\times$ which extends 
to a determinant map
\begin{equation}\label{det-defined}
\xymatrix@C=35pt{
\G_n \cong  \Sn_n \rtimes \Gal(\F_{p^n}/\F_p)
\ar[r]^-{\det \times 1} 
&  \Z_{p}^\times \times \Gal(\F_{p^n}/\F_p)
\ar[r]^-{p_1} & \Z_p^\times
}
\end{equation}
where $p_1$ is the projection onto the first factor of the product.
Define the {\it reduced determinant} $N$ to be the composition
\begin{equation}\label{norm-defined}
\xymatrix{
\G_n \ar@/_1pc/[rr]_N \ar[r]^-{\det} & \Z_p^\times \ar[r] & \Z_p^\times/C \cong \Z_p
}
\end{equation}
where $C \subseteq \Z_p^\times$ is the maximal finite subgroup. 
For example, $C = \{\pm 1\}$ if $p=2$.
There are isomorphisms $\Z_p^\times/C \cong \Z_p$. We choose one for now, although we will
be much more specific in \Cref{rem:many-alphas} below.
Write $\G_n^1$ for the kernel of $N$, $\Sn_n^1 = \Sn_n \cap \G_n^1$, and
$S_n^1 = S_n \cap \G_n^1$. The map $N\colon S_n \to \Z_p$ is split surjective and we have
semi-direct product decompositions for all of $\G_n$, $\Sn_n$, and $S_n$; for example,
there is an isomorphism
\[
\Sn_n^1 \rtimes \Z_p \cong \Sn_n.
\]
If $n$ is prime to $p$, we can choose a central splitting and the semi-direct product is actually
a product, but that is not the case of interest here.

The surjective homomorphism $N\colon \GG_n \to \ZZ_p$ defines a non-zero cohomology class 
$\zeta_n \in H_c^1(\GG_n,\ZZ_p)$. Also write  
\begin{equation}\label{def-zeta-n}
\zeta_n \in H_c^1(\GG_n,E_0)
\end{equation}
for the image of this class under the map induced by the unique continuous homomorphism of rings
$\ZZ_p \to E_0$. 

The class $\zeta_n$ detects a homotopy class in $\pi_{-1}L_{K(n)}S^0$. Let $\pi \in \GG_n$
be any element so that $N(\pi) \in \ZZ_p$ is a topological generator. Then we have a fibration sequence
\begin{equation}\label{eq:an-easy-reduction}
\xymatrix{
L_{K(n)}S^0 \simeq E^{h\GG_n} \ar[r] & E^{h\GG_n^1} \ar[r]^-{\pi -1} & E^{h\GG_n^1} \ .
}
\end{equation}
The following result can be found in Proposition 8.2 of \cite{DH}.

\begin{prop}\label{prop:zeta-permanent} The class $\zeta_n$ is a non-zero permanent cycle in the
Adams-Novikov Spectral Sequence
\[
H_c^s(\GG_n,E_t) \Longrightarrow \pi_{t-s}L_{K(n)}S^0
\]
detecting the image of the unit in $\pi_0E^{h\GG_n^1}$ under the boundary map
\[
\pi_0E^{h\GG_n^1} \longr \pi_{-1}L_{K(n)}S^0.
\]
\end{prop}
\noindent
In an abuse of notation we will call this homotopy class $\zeta_n$ as well.

For many purposes the action of the Galois group
$\Gal = \Gal(\FF_{p^n}/\FF_p) \subseteq \GG_n$ is harmless. 
This can be made precise in the 
following two results, which are from Section 1 of \cite{BobkovaGoerss}.

\begin{lem}\label{coh-decomp-galois} Let $F \subseteq \GG_n$ be a closed subgroup and let
$F_0 = F \cap \SS_n$. Suppose the canonical map
\[
F/F_0 \longr \GG_n/\SS_n \cong \Gal
\]
is an isomorphism. Then for any twisted $\GG_n$-module $M$ we have isomorphisms
\begin{align*}
H_c^\ast(F,M)  &\cong H_c^\ast(F_0,M)^{\Gal}\\
H_c^\ast(F_0,M) &\cong \WW \otimes_{\ZZ_p} H_c^\ast(F,M)\ .
\end{align*}
\end{lem}

This has the following topological analog. If $X$ is a space, let $X_+$ denote $X$ with a disjoint basepoint.

\begin{lem}\label{more-split} Let $F \subseteq \GG_n$ be a closed subgroup and let $F_0 = F \cap \SS_n$.
Suppose the canonical map
\[
F/F_0 \longr \GG_n/\SS_n \cong \Gal
\]
is an isomorphism. Then there is a $\Gal$-equivariant equivalence
\[
\Gal_+ \wedge E^{hF} \to E^{hF_0}.
\]
\end{lem}

If $X$ is a finite spectrum $X$ and $F$ is a  closed subgroup of $\GG_n$, then the natural map
$E^{hF}\smsh X\to (E \smsh X)^{hF}$ is an equivalence. Combining \Cref{coh-decomp-galois} and
\Cref{more-split} thus yields an isomorphism of spectral sequences
\[
\xymatrix{
\WW \otimes_{\ZZ_p} H_c^\ast(F,E_\ast X) \ar@{=>}[r] \ar[d]_\cong &
\WW \otimes_{\ZZ_p} \pi_\ast (E\smsh X)^{hF}\ar[d]^\cong \\
H_c^\ast(F_0,E_\ast X) \ar@{=>}[r]  & \pi_\ast (E \smsh X)^{hF_0}
}
\]
where the differentials on the top line are the $\WW$-linear differentials extended
from the spectral sequence for $F$.

\subsection{Strong vanishing lines}\label{sec:van-line}

The $K(n)$-local Adams-Novikov Spectral Sequence satisfies a very strong horizontal vanishing line 
condition; see \Cref{def:horiz-van}, \Cref{lem:two-cond-van-line}, and \Cref{thm:en-vanishing} below.
We will need this throughout \Cref{sec:k1k2s}. 
These results are in the literature, but a bit hard to pull together in complete detail, 
so we provide a summary here. The main  reference is \S 5 \cite{DH}, but there are also very similar ideas
in the proof of Corollary 15 of \cite{StrickGrossHop}.

Let us write
\[
\xymatrix@C=10pt@R=12.5pt{
X \ar[dr] && \ar@{-->}[ll]Z\\
&Y\ar[ur]
}
\]
for a cofiber sequence (i.e., a triangle) $X \to Y \to Z \to \Sigma X$. Suppose we have a
diagram, where each of the triangles is a cofiber sequence. 

\begin{equation}\label{eq:exact-couple-diag}
\xymatrix@C=10pt@R=12.5pt{
X \ar[dr]&& \ar@{-->}[ll] X_1\ar[dr] && \ar@{-->}[ll] X_2\ar[dr] 
&&  \ar@{-->}[ll] X_3\ar[dr]&&\ar@{-->}[ll] \cdots \\
&F_0\ar[ur]&&F_1\ar[ur]&&F_2\ar[ur]&&F_3\ar[ur]
}
\end{equation}
From \eqref{eq:exact-couple-diag} we obtain a spectral sequence (with $s \geq 0$)
\begin{equation}\label{eq:exact-couple-diag-ss}
E_1^{s,t} = \pi_tF_s \Longrightarrow \pi_{t-s}X.
\end{equation}

Recall that the spectral sequence \eqref{eq:exact-couple-diag-ss} is strongly convergent if 
\begin{equation}\label{def:strong-converge}
\mathrm{lim}_s\ \pi_{t+s}X_s = 0 = \mathrm{lim}^1_s\ \pi_{t+s}X_s
\end{equation}
for all $t$. 

\begin{defn}\label{def:horiz-van} The spectral sequence of \eqref{eq:exact-couple-diag-ss} 
has a {\it strong horizontal vanishing line} at $N$ if for all $s$ and $t$ the map 
$\pi_tX_s \to \pi_{t-N}X_{s-N}$ is zero. 
\end{defn} 

The following result  justifies this nomenclature. The proof is a diagram chase. 

\begin{lem}\label{lem:two-cond-van-line} Suppose the spectral sequence of \eqref{eq:exact-couple-diag-ss}
has a strong horizontal vanishing line at $N$. Then
\begin{enumerate}

\item  the spectral sequence is strongly convergent; 

\item for all $s$ and $t$, $E_N^{s,t} \cong E_\infty ^{s,t}$; and,

\item for all $s \geq N$, $E_\infty^{s,t} = 0$.
\end{enumerate}
\end{lem}

Recall that we write $Y \wedge Z$ for $L_{K(n)}(Y \wedge Z)$ when we are in the $K(n)$-local category
and we write $E=E_n$ for our chosen Morava $E$-theory at $n$.
Then the $K(n)$-local Adams-Novikov Spectral Sequence is obtained from the standard
exact couple diagram constructed from the $E$-based cobar complex in the
 $K(n)$-local category:
\begin{equation}\label{eq:exact-couple-diag-e}
\xymatrix@C=15pt@R=12.5pt{
L_{K(n)}S^0 \ar[dr]&& \ar@{-->}[ll] \overline{E} \ar[dr] && 
\ar@{-->}[ll] \overline{E} \wedge \overline{E}\ar[dr] &&\ar@{-->}
[ll] \cdots \\
&E\ar[ur]&&E \wedge \overline{E}\ar[ur]&&E \wedge \overline{E}\wedge \overline{E}\ar[ur]
}
\end{equation}
Given a spectrum $X$, the $K(n)$-local Adams-Novikov Spectral Sequence
has $E_2$-term isomorphic to
\[
E_2^{s,t} \cong \pi^s\pi_tL_{K(n)}(E^{\wedge \bullet} \wedge X).
\]
If $Z$ is a finite spectrum and $G \subseteq \GG_n$ is closed, then we can set $X = E^{hG} \wedge Z$. 
Then by Proposition 6.6 of \cite{DH} the $E_2$-term
becomes
\[
E_2^{s,t} \cong H_c^s(G,E_tZ).
\]

We write $i_1 = i\colon \overline{E} \to \Sigma L_{K(n)}S^0$ for the boundary map in the first triangle. 
Then the later boundary maps can be written as
\[
i_s = i\wedge \overline{E}^{\wedge (s-1)} \colon \overline{E}^{\wedge s} 
\to \Sigma \overline{E}^{\wedge (s-1)}. 
\]
Furthermore we have identifications
\[
i_1 \circ \cdots \circ i_s = i^{\wedge s} \colon \overline{E}^{\wedge s} \to \Sigma^sL_{K(n)}S^0
\]
and, more generally,
\begin{equation}\label{eq:decompose-i}
i_{r+1} \circ \cdots \circ i_s =
\overline{E}^{\wedge r}\wedge i^{\wedge (s-r)} \colon \overline{E}^{\wedge s} 
\to \Sigma^{s-r} \overline{E}^{\wedge r}.
\end{equation}

The next result is proved using Lemma 5.11, Lemma 5.12, and the argument given for the proof of Theorem 5.3 in \cite{DH}. 
See also the proof of Corollary 15 of \cite{StrickGrossHop}. 

\begin{thm}\label{lem:basic-van-line} There is an integer $N=N_{S^0}$ so that 
\[
i^N: \overline{E}^{\wedge N} \longr  \Sigma^N L_{K(n)}S^0
\]
is null-homotopic.
\end{thm}

\begin{proof} Here is a sketch. Consider the class of $K(n)$-local 
spectra $Z$ for which there exists an integer $N_Z$ depending only on $Z$
so that 
\[
i^{\wedge s} \wedge Z: \overline{E}^{\wedge s} \wedge Z \to \Sigma^s Z
\]
is null-homotopic for $s \geq N_Z$. We first observe this is a thick subcategory of the homotopy category of $K(n)$-local spectra. To see this we note that if we have a cofiber sequence $Z_1 \to Z_2 \to Z_3$
and the result holds for $Z_1$ and $Z_3$   with integers $N_1$ and $N_3$ respectively, then the result holds for $Z_2$ 
with integer $N_1 + N_3$, using \eqref{eq:decompose-i}. Additionally, if $Z_0$ is a retract
of $Z$ and the result holds for $Z$ with integer $N_Z$, it also holds for $Z_0$ with integer $N_Z$.

Using this observation, we need only find a $K(n)$-local finite spectrum for which the result holds and which generates
the same thick subcategory as the sphere. By a result of Hopkins and Ravenel (see \S 8.3 of \cite{RavNil}) 
there is a torsion-free finite CW spectrum $X$ and an 
integer $M$ so that
\[
H^s_c(\GG_n,E_\ast X) = 0, \quad s > M.
\]
An application of Lemma 5.11 of \cite{DH} finishes the argument. 
\end{proof}

This result immediately has the following consequences. 

\begin{thm}\label{thm:en-vanishing} For all spectra $X$, 
the $K(n)$-local Adams-Novikov Spectral Sequence has a strong horizontal vanishing line 
at an integer $N_X \leq N_{S^0}$.
\end{thm}

\begin{cor}\label{thm:en-vanishing-ex} For all $K(n)$-local finite spectra $Z$ 
and all closed subgroups $G\subseteq \GG_n$, the $K(n)$-local Adams-Novikov Spectral Sequence
\[
H_c^s(G,E_tZ) \Longrightarrow \pi_{t-s}(E^{hG} \wedge Z)
\]
has a strong horizontal vanishing line at an integer $N \leq N_{S^0}$.
\end{cor}

\begin{cor}\label{thm:en-vanishing-ex-1} Let $Z$ be a finite complex with a self map $f:\Sigma^k Z \to Z$.
Then for all closed subgroups $G \subseteq \GG_n$, the localized $K(n)$-local Adams-Novikov 
Spectral Sequence
\[
f_\ast^{-1}H_c^s(G,E_tZ) \Longrightarrow f_\ast^{-1}\pi_{t-s}(E^{hG} \wedge Z)
\]
has a strong horizontal vanishing line at an integer $N \leq N_{S^0}$ and converges to the $f$-inverted homotopy groups.
\end{cor}

At various points in this paper we use the following version of the Geometric Boundary 
Theorem to name elements in an Adams-Novikov Spectral Sequence. 

\begin{thm}[{\bf Geometric Boundary Theorem}]\label{rem:whats-up-GBT}
Let  $X \longr Y \longr Z$ be a cofiber sequence of spectra. Let
\begin{enumerate}[(1)]
\item $R$ be a connected homotopy commutative ring spectrum, or
\item $R = E_n$ for some $n$.
\end{enumerate}
In the latter case, let  $R_\ast X = \pi_\ast L_{K(n)}(E_n \wedge X)$.
Suppose that $R_\ast(-)$ applied to the cofiber sequence above gives a short exact sequence
\[
0 \longr R_\ast X \longr R_\ast Y \longr R_\ast Z \longr 0\,
\]
and thus a long exact sequence on the $E_2$-term of the (1) $R$-based, or (2) $E_n$-based $K(n)$-local
Adams-Novikov Spectral Sequence. 

If $x \in \pi_{t-s} Z$ is 
detected by a class $a \in E_2^{s,t}Z$,
then the image of 
$x$ in $\pi_{t-s-1}X$ is detected by the image of $a$ under the connecting homomorphism
\[
\delta\colon E_2^{s,t}Z \to E_2^{s+1,t}X.
\]
\end{thm}
\begin{proof} For the case (1) see Theorem 2.3.4 of \cite{ravgreen}. For the case (2) see Proposition A.10 of \cite{DH}.
\end{proof}

\subsection{Elliptic curves and subgroups of $\GG_2$ at $p=2$}\label{sec:ellback}

Here we spell out what we need from the theory of elliptic curves at $p=2$; this will give us
a preferred formal group and a preferred universal deformation. Choose $\Gamma_2$ 
to be the formal group obtained from the elliptic curve $C_0$ over $\F_2$ defined by the 
Weierstrass equation
\begin{equation}\label{ss-curve}
y^2+y=x^3.
\end{equation}
This is a standard representative for the unique isomorphism class of supersingular curves
over $\overline{\F}_2$; see \cite{Silverman}, Appendix A. As a result $\Gamma_2$ has height
$2$, as the notation indicates. In addition, if $\xi$ is the Frobenius, we have $\xi^2 = [-2]$
in the endomorphism ring of $\Gamma_2$ over $\FF_2$. Thus this formal group satisfied our
assumption from \eqref{eq:frob-assumption}. See Lemma 3.1.1 of \cite{Paper2}. 

Following Strickland (see \cite{strickl2018level} and \cite[Section 2]{Paper2}), 
let $C$ be the elliptic curve over $\W[[u_1]]$ defined by the Weierstrass equation
\begin{equation}\label{ss-curve-def}
y^2+3u_1xy + (u_1^3-1)y=x^3.
\end{equation}

\begin{rem}\label{rem:v1-plus-ss}  
The curve $C$ reduces to $C_0$ modulo $\mathfrak{m} = (2,u_1)$. 
In the equation \eqref{ss-curve-def} note that the coefficient of $xy$ is congruent
to $u_1$ modulo $2$; hence $v_1 = u^{-1}u_1$ for the formal group of $C$ over $\FF_4[[u_1]]$.
See Proposition 6.1.1 of \cite{Paper2}. From this it follows that the formal group $C$ over $\WW[[u_1]]$
is a choice of the universal deformation of $\Gamma_2$. 
\end{rem}

In \cite{Paper2}, $C_0$ was called $\mathcal{C}$ and $C$ was called $\mathcal{C}_U$. The current 
notation is closer to what we find in the number theory literature. 

Again turning to \cite{Silverman}, Appendix A we have
\begin{equation}\label{ss-curve1} 
\Aut(C_0/\F_4) \cong Q_8 \rtimes \F_4^\times
\end{equation}
where the action of $\F_4^\times \cong C_3$ on $Q_8$ is determined by a cyclic permutation of generators
$i$, $j$ and their product $k=ij$. We will denote this group by $G_{24}$.
For a choice of generator $\omega \in C_3$, we can choose them to satisfy
\begin{align}\label{eq:ijkreln} \omega i \omega^2 =j, \ \ \ \omega j \omega^2 = k, \ \ \ \omega k \omega^2 = i.\end{align}

 Define 
\begin{equation}\label{ss-curve2}
G_{48} = \Aut(\F_4,C_0) \cong \Aut(C_0/\F_4) \rtimes \Gal(\F_4/\F_2).
\end{equation} 
Since any automorphism of the pair $(\F_4,C_0)$ induces an automorphism of  the
pair $(\F_4,\Gamma_2)$ we get a map $G_{48} \to \G_2$. This map is an injection 
and we identify $G_{48}$ with its image. 

\begin{rem}\label{rem:levelstructures} Let  $C_0[3]$ be the subgroup scheme of $C_0$ consisting of points of order $3$;  over 
$\FF_4$, this becomes abstractly isomorphic to $\Z/3 \times \Z/3$.  The group $G_{48}$ acts 
linearly on $C_0[3]$ and choosing  a basis for the $\FF_4$-points of $C_0[3]$ determines an isomorphism 
$G_{48} \cong GL_2(\Z/3)$.
\end{rem}

We will be interested in various subgroups of $G_{48}$. 
The following subgroups will play an important role in this
paper.  
\begin{enumerate}

\item $C_2 = \{\pm 1\} \subseteq Q_8$;

\item $C_6 = C_2 \times \F_4^\times$;

\item $G_{24}$ and $G_{48}$ themselves.
\end{enumerate}

 For the formal group $\Gamma_2$ of $C_0$, the group
$\SS_2 =\Aut(\Gamma_2/\FF_2)$ has a very concrete description. See \S 3 of \cite{Paper2} for a proof of the following result.
\begin{prop}\label{rem:where-did-t-arise}
Let $\WW = W(\FF_4)$ be the
Witt vectors on $\FF_4$. The endomorphism ring of $\Gamma_2$ is the non-commutative extension
of $\WW$ 
\[
\WW\langle T \rangle/(T^2=-2,a^{\sigma}T=Ta)
\]
where $a \in \W$ and $(-)^\sigma$ is the action of the Frobenius on $\WW$.  Consequently, the group $\SS_2$ is the group of units in this
ring. 
\end{prop}

Up to conjugacy in $\SS_2$, the subgroup $Q_8 \subseteq \SS_2$ is the unique
maximal finite 
$2$-primary subgroup. Similarly, up to conjugacy in $\SS_2$, $G_{24}$ is the unique maximal finite subgroup. 
See \cite{bujard} for a further discussion on these topics as well as Lemma 2.4.3 of \cite{Paper1} for an explicit embedding of $G_{24}$ in $\mathbb{S}_2$.
However, for $\SS^1_2$, neither of these
statements remains true. 
More precisely, any finite subgroup of $\SS_2$ is necessarily in $\SS_2^1$, but if $\gamma \in \SS_2$ 
is any element such that $N(\gamma)$ is a topological generator of $\Z_2^{\times}/(\pm 1)$ for $N$ as defined in \eqref{norm-defined}, then $\gamma G_{24}\gamma^{-1}$ 
is not conjugate to $G_{24}$ in $\SS_2^1$. Specifically, we define the following elements of $\WW^{\times} \subseteq \mathbb{S}_2$ that will play a key role:
\begin{defn}\label{rem:g24prime}
For $\omega \in \WW$ a primitive cube root of unity. Let
\[\pi = 1 +2\omega \in \WW^{\times}\] 
be as in (2.3.3)  of \cite{Paper1} and
\[\alpha = \frac{1-2\omega}{\sqrt{-7}} \in \WW^{\times}\]
be as in (2.3.4) of \cite{Paper1}.
Here $\sqrt{-7}$ is the unique element of $\WW^{\times}$ which squares to $-7$ and is congruent to $1$ modulo $4$.
\end{defn}

The element  $\pi$ of \Cref{rem:g24prime} satisfies $\det(\pi) =3$ and so 
\[N(\pi)\equiv 3 \in \Z_2^{\times}/(\pm 1).\] 
We use this specific element to define
\begin{equation}\label{eq:g24prime}
G'_{24} := \pi G_{24}\pi^{-1}
\end{equation}
The subgroup  $G'_{24} \subseteq \mathbb{S}_2^1$ is a representative of this other conjugacy class.

\begin{rem}\label{rem:alpha}
The element $\alpha$ also arises at various points which is why we included it in \Cref{rem:g24prime}. It satisfies $\det(\alpha)=-1$ and so $N(\alpha)\equiv 1  \in \Z_2^{\times}/(\pm 1) $. Furthermore, from the definition of $\alpha$ and of $\mathbb{S}_2$ in \Cref{rem:where-did-t-arise} above, we have that 
\[
\alpha \equiv 1+\omega T^2 + T^4\qquad \mathrm{modulo}\ (T^6).
\]
\end{rem}

We have been discussing $G_{24}$ as a subgroup of $\SS_2$,
but it can also be thought of as a quotient as well. The following result combines facts that were shown in Section 2.5 of \cite{Paper1}:

\begin{prop} \label{rem:all-the-splittings}
The group $\SS_2$ has a normal torsion-free
pro-$2$-subgroup  $K$ so that:
\begin{enumerate}[(a)]
\item The group  $K$ is a Poincar\'e duality group of dimension $4$. 
\item For $\mathbb{S}_2^1$ as defined in \Cref{subgroupsection}, the group $K^1 = \SS_2^1 \cap K$ is a Poincar\'e duality group of dimension $3$.
\item The composition 
\[
G_{24} \longr \SS_2 \longr \SS_2/K
\]
is an isomorphism, giving rise decompositions $K \rtimes G_{24} \cong \SS_2$ and 
$K^1 \rtimes G_{24} \cong \SS_2^1$. 
\item The composition
\[
G_{24}' \longr \SS_2 \longr \SS_2/K
\]
remains an isomorphism and defines an alternate splitting. 
\end{enumerate}
\end{prop}
\noindent
The subgroups $K$ and $K^1$
play a key technical role in this paper, and we will supply more information when we need it
below.

\begin{rem}\label{input-for-the-DSS} In order to calculate the algebraic duality spectral sequences
or topological duality spectral sequences defined below in \Cref{sec.dualityres}
we will need some information on $H^\ast(F,E_\ast)$ and $\pi_\ast E^{hF}$ for the finite subgroups
$F=C_6$ and $F=G_{24}$. There is a detailed summary of literature
in \S 2 of \cite{BobkovaGoerss} and here we make only  a few remarks. The spectral sequence
\[
H^s(C_6,E_t) \Longrightarrow \pi_{t-s}E^{hC_6}
\]
is relatively simple and can be deduced from \cite{MR}. The spectral sequence
\[
H^s(G_{24},E_t) \Longrightarrow \pi_{t-s}E^{hG_{24}}
\]
is much harder, but extensively studied; see \cite{tbauer} or \cite{tmfbook}. For now we record that
there are classes $c_4$, $c_6$, $\Delta$, and $j$ in $H^0(G_{24},E_\ast)$, of degrees
$8$, $12$, $24$, and $0$ respectively. These are the modular forms of the curve \eqref{ss-curve-def},
hence necessarily invariant under the automorphisms of the curve. Then there is an isomorphism
\[
\WW[[j]][c_4,c_6,\Delta^{\pm 1}]/I \cong H^0(G_{24},E_\ast)
\]
where $I$ is the ideal generated by
\[
c_4^3 - c_6^2 = (12)^3 \Delta\qquad\mathrm{and}\qquad j\Delta = c_4^3.
\]
If we quotient out by $2$, there is also an isomorphism.
\[
\F_4[[j]][v_1, \Delta^{\pm1}]/( j\Delta = v_1^{12}) \cong H^0(G_{24},E_\ast/2) \ .
\]
We will give more information in \Cref{rem:hur} below. 
\end{rem}

\subsection{The duality resolutions}\label{sec.dualityres}

A key computational tool in this paper is the algebraic duality resolution of \cite{Paper1} and its
topological analog from \cite{BobkovaGoerss}. We begin with the algebraic version.

If $X = \lim\ X_\alpha$ is a profinite set, define $\ZZ_2[[X]] = \lim_{n,\alpha} \ZZ/2^n[X_\alpha]$, 
where $\ZZ/2^n[Y]$ denotes the free $\ZZ/2^n$-module on the set $Y$. 
If $G$ is a profinite group let $IG \subseteq \ZZ_2[[G]]$ be the augmentation ideal; that is, the kernel of the 
augmentation $\epsilon\colon \ZZ_2[[G]] \to \ZZ_2$.

The following can be found in Theorems 1.2.1 and 1.2.6  of \cite{Paper1}. 
The subgroups $G_{24}$,  $C_6$, and $G'_{24}$ have been defined in \Cref{sec:ellback}, and the subgroup $K^1$ appeared in \Cref{rem:all-the-splittings}. 

\begin{thm}[\textbf{Algebraic Duality Resolution}]\label{thm:bob}  Let $F_0 = G_{24}$, 
$F_1= F_2 = C_6$ and $F_3 =G_{24}'$.
\begin{enumerate}[(1)]
\item There is an exact sequence of complete left $\Z_2[[\mathbb{S}_2^1]]$-modules
\begin{equation*}
\xymatrix@=13pt{ 0 \ar[r] & \Z_2[[\mathbb{S}_{2}^1/F_3]]  \ar[r]^{\partial_3}   &
\Z_2[[\mathbb{S}_{2}^1/F_2]]  \ar[r]^{\partial_2}    & \Z_2[[\mathbb{S}_{2}^1/F_1]]   \ar[r]^{\partial_1}  & 
\Z_2[[\mathbb{S}_{2}^1/F_0]]  \ar[r]^-\epsilon   & \Z_2 \ar[r] &0.}
\end{equation*}
\item The maps $\partial_1$ and $\partial_3$ are trivial modulo $IK^1$. 
Furthermore, $\partial_2$ is multiplication by $2$ modulo $(8,IS_2^1)$.   
\end{enumerate}
\end{thm}

\begin{rem}\label{rem:bobrefined}  Note that part (2) of \Cref{thm:bob} implies that 
modulo $I\SS_2^1$ we have $\partial_1 = 0 = \partial_3$ and 
$\partial_2$ is multiplication by $2$ modulo $(8,I\SS_2^1)$.
For many arguments, this is all we will need.
\end{rem}

\begin{defn}[\textbf{Algebraic Duality Spectral Sequence}]\label{rem:algdualresSS}
If $M$ is a profinite $\ZZ_2[[\SS_2^1]]$-module,  we have a natural isomorphisms
\begin{equation}\label{eqn:schapiro1}
\Ext^q_{\ZZ_2[[\SS_2^1]]}(\ZZ_2[[\SS_2^1/F_p]],M) \cong H^q(F_p,M)
\end{equation}
and part (1) of \Cref{thm:bob} immediately gives a spectral sequence
\begin{equation*}
\xymatrix{E_1^{p,q} = H^q(F_p, M ) \ar@{=>}[r] & H_c^{p+q}( \mathbb{S}_2^1,M)}
\end{equation*}
with differentials $d_r \colon E_r^{p,q}  \to E_r^{p+r,q-r+1}$. We will call this the {\it
algebraic duality spectral sequence}, which we may abbreviate as ADSS. 
\end{defn}
The construction of this spectral sequence is provided in Section 3.2 of \cite{Paper1}. The facts needed about homological algebra over the completed group ring $\Z_p[\![G]\!]$ of a profinite groups $G$ can also be found in the Appendix of \cite{Paper1} and a more extensive reference on this topic is \cite{sw_analyticgroups}.

We can induce the exact sequence of \Cref{thm:bob} up to an exact sequence of 
complete $\Z_2[[\mathbb{G}_2]]$-modules 
\begin{align*}
0 \to \Z_2[[\mathbb{G}_{2}/F_3]]  &\mathop{\longr}^{\partial_3}  
\Z_2[[\mathbb{G}_2/F_2]]\nonumber\\
&\mathop{\longr}^{\partial_2}  \Z_2[[\mathbb{G}_{2}/F_1]]   
\mathop{\longr}^{\partial_1} \Z_2[[\mathbb{G}_{2}/F_0]]  \to \Z_2[[\GG_2/\SS_2^1]] \to 0\ .
\end{align*}
For any closed subgroup of $F$ of $\GG_2$, the isomorphism of (\ref{eq:EhomofF}) gives us
an isomorphism of twisted $\GG_2$-modules  
\[
E_*E^{hF}\cong \Hom_{\ZZ_2}(\ZZ_2[[\GG_2/F]],E_\ast) .
\]
From these observations we get the following result.
Note that since $G_{24}$ and $G'_{24}$ are conjugate in $\GG_2$, we have $E^{hG'_{24}} \simeq
E^{hG_{24}}$.

\begin{cor}\label{cor:morava-dual} There is an exact sequence of twisted $E_*$-$\GG_2$-modules
\begin{equation*}
0 \to E_\ast E^{h\SS_2^1}\to E_\ast E^{hG_{24}} \to 
E_\ast E^{hC_6} \to E_\ast E^{hC_6} \to E_\ast E^{hG_{24}}\to 0.
\end{equation*}
\end{cor}

The following is the main theorem of \cite{BobkovaGoerss}. This is the 
{\it topological duality resolution.}
Note that it follows from \Cref{input-for-the-DSS} that  multiplication by 
$\Delta^k$ induces an isomorphism of Morava modules 
$E_\ast \Sigma^{24k}E^{hG_{24}} \cong E_\ast E^{hG_{24}}$
for all $k$, but this extends to an equivalence  $\Sigma^{24k}E^{hG_{24}} \simeq  E^{hG_{24}}$ 
if and only if $k\equiv 0$ modulo $8$, as $\Delta^k$ is a permanent cycle in the 
homotopy fixed point spectral sequence only for those $k$. See \cite{tmfbook} for history and
details. This all means that the  suspension factor on the last spectrum in the following topological 
resolution is significant. 

\begin{thm}\label{thm:top-dual-res} The algebraic resolution of \Cref{cor:morava-dual} can
be realized by a sequence of spectra
\begin{equation*}
E^{h\SS_2^1} \mathop{\longrightarrow}E^{hG_{24}} \rightarrow E^{hC_6} \rightarrow
E^{hC_6} \rightarrow \Sigma^{48}E^{hG_{24}}.
\end{equation*}
In this sequence all compositions and all Toda brackets are zero modulo indeterminacy. The
sequence can be refined to a tower of fibrations
\begin{equation}\label{duality-tower}
\xymatrix{
E^{h\SS_2^1} \ar[r] & Z_2 \ar[r] & Z_1 \ar[r]& E^{hG_{24}} \\
\Sigma^{45}E^{hG_{24}} \ar[u]
& \Sigma^{-2}E^{hC_6}  \ar[u] &
\Sigma^{-1}E^{hC_6}  \ar[u]
}
\end{equation}
\end{thm}

\begin{defn}[{\bf Topological Duality Spectral Sequence}]\label{rem:topdualresSS} We will
write $\sE_p$ for the $p$th spectrum in the topological duality resolution of \Cref{thm:top-dual-res}.
Thus
\[
\sE_p = \begin{cases}
E^{hG_{24}},\quad p=0;\\
E^{hC_6},\quad p=1,2;\\
\Sigma^{48}E^{hG_{24}},\quad p=3.
\end{cases}
\]
Using the notation of \Cref{rem:algdualresSS}, we get spectral sequences
\[
H^s(F_p,E_t) \Longrightarrow \pi_{t-s} \sE_p.
\]
As a complement to the spectral sequence of 
\Cref{rem:algdualresSS},
the tower of (\ref{duality-tower}) gives a spectral sequence
\[
E_1^{p,q} = \pi_q\sE_p \Longrightarrow \pi_{q-p} E^{h\SS_2^1}.
\]
We refer to this spectral sequences as the {\it topological duality spectral sequence}.
We may abbreviate this as the TDSS. 
\end{defn}
\begin{rem}
There are variants we use; for example, we could 
smash the tower of \eqref{duality-tower} with a spectrum $Y$ to get a spectral sequence
converging to the homotopy groups $\pi_\ast L_{K(2)}(E^{h\SS_2^1} \wedge Y)$.
\end{rem}

Here are more details on the two duality spectral sequences.
The resolution of \Cref{thm:top-dual-res} can be refined  to a tower over $E^{h\SS_2^1}$
\begin{equation}\label{duality-cofs}
\xymatrix@C=10pt@R=12.5pt{
E^{h\SS_2^1} \ar[dr] && \ar@{-->}[ll] D_1\ar[dr] && \ar@{-->}[ll] D_2\ar[dr] && 
\ar@{-->}[ll] D_3\ar[dr]^\simeq\\
&E^{hG_{24}}\ar[ur]&&E^{hC_6}\ar[ur]&&E^{hC_6}\ar[ur]&&\Sigma^{48}E^{hG_{24}}\ .
}
\end{equation}
As before we write
\[
\xymatrix@C=10pt@R=12.5pt{
X \ar[dr] && \ar@{-->}[ll]Z\\
&Y\ar[ur]
}
\]
for a cofiber sequence $X \to Y \to Z \to \Sigma X$. The dotted arrows in \eqref{duality-cofs} have 
Adams-Novikov filtration 1.
The resulting spectral sequence is isomorphic to the topological duality spectral sequence.
For more details see Section 3 of \cite{BobkovaGoerss}.

If we apply $E_\ast (-)$ to the diagram \eqref{duality-cofs}, using the exactness of the sequence in \Cref{cor:morava-dual}, we get a sequence of short exact sequences and hence
of long exact sequences in cohomology
\begin{equation*}\label{duality-cofs-coh}
\xymatrix@C=5pt@R=12.5pt{
H_c^\ast(\GG_2,E_\ast D_{p-1})\ar[dr] && \ar@{-->}[ll] H_c^\ast(\GG_2,E_\ast D_p)\\
&H_c^\ast(\GG_2,E_\ast E^{hF_{p-1}})\ar[ur]\ .
}
\end{equation*}
Here the dotted arrow raised cohomological degree by $1$ and $D_0 = E^{h\SS_2^1}$. Spliced together, these long exact sequences give a 
spectral sequence
\[
E_1^{p,q} = H_c^q(\GG_2,E_tE^{hF_p}) \Longrightarrow H_c^{p+q}(\GG_2,E_tE^{h\SS_2^1}). 
\]
Since $E_\ast E^{hF} \cong \map_{cts}(\GG_2/F,E_\ast)$ for any closed subgroup $F$, Schapiro's Lemma implies
\[
H_c^\ast(\GG_2,E_\ast E^{hF}) \cong H_c^\ast(F,E_\ast)
\]
and this spectral sequence is isomorphic to the ADSS. 

We now have two ways of calculating $\pi_\ast E^{h\SS_2^1}$ from $H^\ast(F_p,E_\ast)$ encoded in the two ways
around the following square
\[
\xymatrix@C=40pt{
H^s(F_p,E_t) \ar@{=>}[d]_{ADSS} \ar@{=>}[r]^-{HFPSS} & \pi_{t-s}E^{hF_p} \ar@{=>}[d]^{TDSS} \\
H_c^{s+p}(\SS_2^1,E_t) \ar@{=>}[r]_-{ANSS} & \pi_{t-s-p}E^{h\SS_2^1}.
}
\]
In general these two methods have a relatively complicated relationship, but under certain hypotheses we can deduce
certain information. We won't need much. The hypotheses of the following result are crafted to ensure there are no exotic jumps of
filtration in related Adams-Novikov Spectral Sequences.

\begin{lem}\label{lem:chasing-detectors} Let $x \in \pi_{n} E^{h\SS_2^1}$. Suppose
\begin{enumerate}

\item the class $x$ is detected by $\alpha \in H^{p}(\SS_2^1,E_t)$ in the ANSS, so $n=t-p$;

\item the class $x$ is detected by $y \in \pi_{t}E^{hF_p}$ in the TDSS; and,

\item the class $y$ is detected by $\beta \in H^{0}(F_p,E_t)$. 
\end{enumerate}
Then $\alpha$ is detected by $\beta$ is the ADSS.
\end{lem} 

\begin{proof} We refer to \eqref{duality-cofs} and the remarks on the ADSS and TDSS following. By hypothesis, there is
a class $z \in \pi_\ast D_p$ which maps to $y \in \pi_\ast F_p$ and to $x \in \pi_\ast E^{h\SS_2^1}$. The hypothesis
on $\beta$ forces $z$ to be detected by a class $\gamma \in H^0(\GG_2,E_\ast D_p)$. By the Geometric 
Boundary Theorem \Cref{rem:whats-up-GBT} $x$ must be detected by the image of $\gamma$ under the map
$H^0(\GG_2,E_\ast D_p) \to H^p(\GG_2,E_\ast E^{h\SS_2^1})$. Since $x$ is detected by $\alpha$
the result follows.
\end{proof}


\section{Recollections from homotopy theory}\label{sec:recollections}

In this section we gather together some of the material we need from the homotopy
theory literature. First, we discuss some qualitative aspects of  the homotopy groups of $V(0)$, 
making explicit 
some interesting behavior which can be traced back to the fact that the order of the identity of
$V(0)$ is not $2$. Then we give some basic background on how the Hopf map $\sigma \in \pi_7S^0$ 
appears in the Adams-Novikov Spectral Sequence. Both of these topics arise repeatedly in the following
sections.

\subsection{Some basic homotopy theory of $V(0)$}\label{sec:remonV0} 
Let $\iota \colon S^0 \to V(0)$ be the inclusion of the bottom cell and $p \colon V(0) \to S^1$ the
collapse map onto the top cell so that
\begin{align}\label{eq:cofV0}
\xymatrix{S^0 \ar[r]^-{\times 2} &S^0 \ar[r]^-\iota & V(0) \ar[r]^-p &S^1}\end{align}
is the standard cofiber sequence for $V(0)$.
For a spectrum $X$,
let 
\[j=X \wedge \iota \colon X \simeq X \wedge S^0  \to X\wedge V(0)\]
and
$q = X \wedge p\colon X \wedge V(0) \to \Sigma X$.  We then have a cofiber sequence

\[
\xymatrix{
X \ar[r]^-{\times 2}& X \ar[r]^-j & X \wedge V(0) \ar[r]^-q & \Sigma X.
}
\]
We have a long exact sequence in homotopy
\[
\xymatrix{
\cdots \longr \pi_{n+1} X \wedge V(0) \ar[r]^-{q_\ast} & \pi_nX \ar[r]^-{\times 2} & \pi_nX \ar[r]^-{j_\ast} & \pi_{n} X \wedge V(0) \longr
\cdots.
}
\]

\begin{defn}\label{rem:betabock}
We let $\beta \colon \pi_n(X \wedge V(0)) \to \pi_{n-1}(X \wedge V(0))$ be the 
homomorphism induced by the composite
\begin{equation*}
\xymatrix{
X \wedge V(0) \ar[r]^-q  & \Sigma X \ar[r]^-j & \Sigma X \wedge V(0).
}
\end{equation*}
\end{defn}

\begin{rem}
Note that any element in the image of $\beta$ has order $2$ as the image of $j_* \colon \pi_*X \to \pi_*(X \wedge V(0))$ is isomorphic 
to $(\pi_*X)/2$.
\end{rem}

\begin{notation}\label{not:keep-it-str} We will write $x\alpha \in \pi_{n+k}Y$ for the composition
\[
\xymatrix{
S^{n+k} \ar[r]^-\alpha & S^n \ar[r]^-x & Y.
}
\]
If it is defined we will write $\langle x,\alpha, \beta \rangle$ for the Toda brackets of the triple of maps
\[
\xymatrix{
S^{n+j+k} \ar[r]^-\beta & S^{n+k} \ar[r]^-\alpha & S^n \ar[r]^x & Y.
}
\]
\end{notation}

We prove the following basic fact, which we use many places throughout this paper.
\begin{lem}\label{lem:twoxtildes10} Let $x \in \pi_nX$ have order $2$. Then there is a class
$y  \in \pi_{n+1} (X \wedge V(0))$  such that $q_\ast (y) = x$ and
\[
2y =  j_\ast (x)\eta  \in \pi_{n+1} (X \wedge V(0)).
\]
\end{lem}

\begin{proof} We begin with the basic case of $\iota \in \pi_0V(0)$. We have a short exact sequence in 
homotopy
\[
0 \to \ZZ/2 \cong \pi_1 V(0) \xrightarrow{j_*} \pi_1(V(0) \wedge V(0)) \xrightarrow{q_*} \pi_0V(0) 
\cong \ZZ/2 \to 0
\]
where $\pi_1V(0)$ is generated by $\iota\eta$ and $\pi_0V(0)$ is generated by $\iota$. The cofiber sequence
\[
\xymatrix{
V(0) \ar[r]^-j & V(0) \wedge V(0) \ar[r]^-q & \Sigma V(0).\
}
\]
cannot split, as the Steenrod operation $\mathrm{Sq}^2$ is non-zero in $H^\ast (V(0) \wedge V(0),\FF_2)$. Hence
there can be no element of order $2$ in $\pi_1(V(0) \wedge V(0))$ that maps to $\iota$ and we have
\[
\pi_1(V(0) \wedge V(0)) \cong \ZZ/4.
\]
If $y_0 \in \pi_1(V(0) \wedge V(0))$ is either of the two elements with $q_\ast(y_0) = \iota$, we have $2y_0 = \iota\eta$. 

Now consider the general case. Extend $x$ to a map $\overline{x}\colon \Sigma^n V(0) \to X$ and then contemplate
the  diagram
\[
\xymatrix{
\Sigma^n V(0) \ar[r] \ar[d]_-{\overline{x}} & \Sigma ^n V(0) \wedge V(0) \ar[d]^-{\overline{x}\wedge V(0)} 
\ar[r] & \Sigma^{n+1} V(0)\ar[d]^-{\Sigma \overline{x}}\\
X \ar[r]_-j & X \wedge V(0) \ar[r]_-q & \Sigma X\ .
}
\]
Then we can let $y$ be the image of a generator $y_0 \in \pi_1(V(0) \wedge V(0))$ under the
map $(\overline{x}\wedge V(0))_\ast$.
\end{proof} 

\begin{rem}\label{rem:basic-toda} We can refine \Cref{lem:twoxtildes10} in the case where
$x = z\eta$ and $z \in \pi_{n-1}X$ itself has order $2$. Then the Toda bracket
$\left<z,2,\eta\right> \subseteq \pi_{n}X$ 
is defined with indeterminacy
\[
z \pi_2S^0 + \pi_{n}X\eta   =\pi_{n}X\eta.
\]
It follows that $\left<z,2,\eta\right>2$ has no indeterminacy and we can shuffle
\begin{equation}\label{order4one}
\left<z,2,\eta\right>2 = z\left<2,\eta,2\right> = z\eta^2.
\end{equation}
Thus if we have $x = z \eta$ we can take any element of $y \in \left<z,2,\eta\right>$ and $2y =z \eta^2= x \eta $. 
We can turn this thought around as well: if $z$ has order $2$ and $z \eta^2 \ne 0$, then $z \eta^2$ {\it must} be 
divisible by $2$. 
\end{rem}
 
\begin{exmp}\label{rem:pi1V0smV0} Here are a few simple applications of \Cref{lem:twoxtildes10} and 
\Cref{rem:basic-toda}. We will only use \eqref{eq:the-more-basic-rel} in the sequel, but 
\eqref{eq:the-basic-rels} and the charts of \Cref{fig:V0} make for a more complete story.

Let $v_1 \in \pi_2V(0)$ be either of the two classes which map to $\eta \in \pi_1V(0)$ under the boundary
map $\pi_2V(0) \to \pi_1S^0$. Since $v_1 \in \langle \iota,2,\eta \rangle$, (\ref{order4one}) gives that
\begin{equation}\label{eq:the-more-basic-rel}
2v_1 = \eta^2\ne 0.
\end{equation}

Let $i_0 \in \pi_0(V(0) \wedge V(0)) \cong \ZZ/2$ be the generator and $i_1 \in \pi_1(V(0) \wedge V(0)) \cong \ZZ/4$ be 
a choice of generator. (We will fix a choice below in \eqref{eq:choice-ione}.)
Let us also write $v_1$ for $j_\ast (v_1) \in \pi_2(V(0)\wedge V(0))$. 
The image of $j_\ast$ has order two in $\pi_2(V(0)\smsh V(0))$, and hence $2v_1=0$. We can then 
choose an element
\[
v_1^2 \in \langle v_1, 2,  \eta  \rangle \subseteq \pi_4(V(0)\smsh V(0)),
\]
so named because of its Hurewicz image in $BP_\ast (V(0)\smsh V(0))$.
Then  \Cref{lem:twoxtildes10}
and \eqref{order4one} yield relations in $\pi_\ast (V(0) \wedge V(0))$:
\begin{align}\label{eq:the-basic-rels}
2i_1 &=  i_0\eta\\
2v_1^2 &=  v_1\eta^2\ .\nonumber
\end{align}
\end{exmp}

The basic relations of (\ref{eq:the-basic-rels}) come as exotic
extensions in any Adams-Novikov  Spectral Sequence. Using the standard calculations
of \cite{MRW} or even Table 2 of \cite{ravnovice}, it is an exercise to compute
the $BP$-based Adams-Novikov Spectral Sequence for $V(0)$ and $V(0) \wedge V(0)$
in a small range. The following two charts display the 
$E_\infty$-pages for $V(0)$ and $V(0) \wedge V(0)$ respectively. The charts
display additive extensions, $\eta$-multiplications, and $\nu$-multiplications. A vertical line
denotes an exotic multiplication by $2$. An important differential in the Adams-Novikov Spectral Sequence for $V(0)$ is
$d_3(v_1^2) = \eta^3$ (Theorem 5.13 (a) \cite{ravnovice}). This differential will appear often in this paper.

\begin{center}
\begin{figure}[H]
\begin{minipage}{0.49\textwidth}
\center
\includegraphics[width=\textwidth]{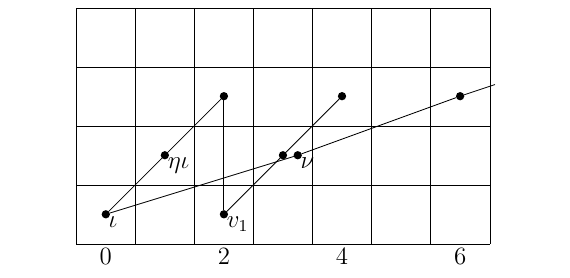}
\end{minipage}
\begin{minipage}{0.49\textwidth}
\center
\includegraphics[width=\textwidth]{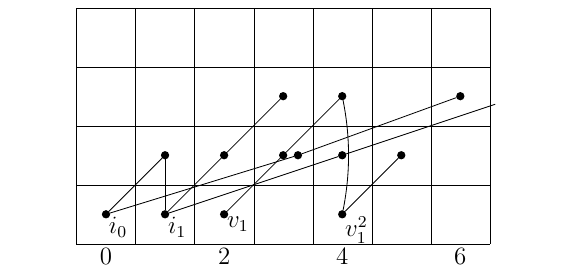}
\end{minipage}
\captionsetup{width=\textwidth}
\caption{The Adams-Novikov $E_{\infty}$-pages for $\pi_*V(0)$ (left) and $\pi_*(V(0) \smsh V(0))$ (right).}
\label{fig:V0}
\end{figure}
\end{center}

We close this subsection with a lemma about Spanier-Whitehead duality. If $Y$ is a spectrum, 
let $DY$ denote its Spanier-Whitehead dual. We have natural isomorphisms of homotopy
classes of maps
\[
 [X,Z\wedge DY] \cong [X \wedge Y,Z].
\]
Explicitly, this isomorphism sends a homotopy class $f:X \to Z \wedge DY$ to the homotopy class $g:X \wedge Y \to Z$
given by the composition
\[
\xymatrix@C=40pt{
X \wedge Y \ar[r]^-{f \wedge Y}  & Z \wedge DY \wedge Y  \ar[r]^-{ Z\wedge m} & Z \wedge S^0 = Z
}
\]
where $m \colon DY \wedge Y \to S^0$ is the duality pairing.  We call $g$
the SW-{\it adjoint} of $f$. If $X$ is a finite spectrum the natural map $X \to D^2X$ is an equivalence. 

We now specialize to the Moore spectrum. If
\[
f \colon S^0 \to V(0) \wedge DV(0) = V(0) \wedge \Sigma^{-1}V(0)
\]
is SW adjoint to the identity then 
\begin{equation}\label{eq:choice-ione}
i_1 = \Sigma f: S^1 \to V(0) \wedge V(0)
\end{equation}
is a choice of generator of $\pi_1 (V(0) \wedge V(0)) \cong \ZZ/4$.

\begin{lem}\label{what-ew-need-sw} Let  $y \colon S^n \to X \wedge V(0)$, let
$x \colon \Sigma^{n-1}V(0) \simeq \Sigma^nDV(0)\to X$ be the SW-adjoint to $y$, and 
consider the map
\[
x \wedge V(0) \colon \Sigma^{n-1}V(0) \wedge V(0) \longr X \wedge V(0).  
\]
Then
\begin{align*}
x \wedge V(0)_\ast (i_1) & = y\\
x \wedge V(0)_\ast (i_0) & = \beta(y)
\end{align*}
where $\beta \colon \pi_n(X \wedge V(0)) \to \pi_{n-1}(X \wedge V(0))$ is the homomorphism  of \Cref{rem:betabock}. 
\end{lem}

\begin{proof} The map $i_1:S^1 \to V(0) \wedge V(0)$ is the suspension of the SW adjoint to the identity
map $V(0) \to V(0)$, so the formula for $x \wedge V(0)_\ast (i_1)$ follows. We then check that $\beta(i_1) = i_0$
to get the formula for $x \wedge V(0)_\ast (i_0)$.
\end{proof}

\subsection{Finding $\sigma$ in the $K(n)$-local sphere}\label{findsigma}

A crucial actor in our proof of the decomposition of $L_{K(1)}L_{K(2)}S^0$ is the Hopf class 
$\sigma \in \pi_7S^0\cong \ZZ/16$. This class remains non-zero 
in $\pi_\ast L_{K(2)}S^0$ and also in 
$\pi_\ast E^{h\SS_2^1}$, and here we discuss how it is detected.

We must be precise about what we mean by $\sigma$. It will be the  generator of 
$\pi_7S^0$ detected in $\Ext_{BP_*BP}^{1,8}(BP_*, BP_*)$  by the Greek letter construction of 
Miller, Ravenel, and Wilson. Recall from Theorem 4.3.2 of \cite{ravgreen} that there is an isomorphism
\[
\FF_2[v_1] \cong \Ext^{0,\ast}_{BP_\ast BP}(BP_\ast,BP_\ast/2)
\]
and that the $\alpha$-family in $\Ext^{1,\ast}_{BP_\ast BP}(BP_\ast,BP_\ast)$ is defined 
by taking the image of the powers of $v_1$ under various Bockstein homomorphisms. Specifically, from
Corollary 4.23 of  \cite{MRW} we know that the class
\begin{equation}\label{eqn:whatv}
v := v_1^4-8v_1v_2 \in BP_8
\end{equation}
becomes a comodule primitive in $BP_8/16$ and in fact we have an isomorphism
\begin{equation}\label{eq:pesky-referee}
\ZZ/16 \cong \Ext_{BP_\ast BP}^{0,8}(BP_\ast, BP_\ast/16) 
\end{equation}
generated by $v$. Let $S^0/16$ be  the cofiber of the map $16\colon S^0 \to S^0$. The Adams-Novikov chart for $S^0/16$
in this range is displayed in Table 3 of \cite{ravnovice}. The class $v$ of \eqref{eqn:whatv} is clearly displayed there,
which is why we have used that notation.

We define $\alpha_{4/4}$ 
to be the image of this class under 
the connecting Bockstein homomorphism associated to the short exact sequence 
\begin{align}\label{eqn:bock16}
\xymatrix@C=25pt{
0 \ar[r] &  BP_* \ar[r]^-{\times 16} & BP_* \ar[r] & BP_*/16 \ar[r] & 0 \ .
}
\end{align}
Then $\pi_7 S^0 \to \Ext_{BP_*BP}^{1,8}(BP_*, BP_*)$ is an isomorphism 
and we define $\sigma$ to be the unique homotopy class which maps to $\alpha_{4/4}$ under this 
map. We also write  $\sigma$ for this class (that is, for $\alpha_{4/4}$) in $\Ext_{BP_*BP}^{1,8}(BP_*, BP_*)$.
This is an abuse of notation, but a standard one: we typically do this for $\eta$ and $\nu$ as well.

Using the Geometric Boundary Theorem of \Cref{rem:whats-up-GBT} and the isomorphism of \eqref{eq:pesky-referee}
it follows that the permanent cycle $v$ itself detects a class
$x \in \pi_8 (S^0/16)$ which maps to $\sigma$ under the map
$\partial\colon \pi_8 S^0/16 \to \pi_7 S^0$. 
The class $x$ is not unique, 
as the kernel of $\partial$ is $\ZZ/2 \times \ZZ/2$, 
but the elements of the kernel all have
higher Adams-Novikov filtration, so all choices of $x$ are detected by $v$. 

Let $E = E_n$ be a choice of a Lubin-Tate theory at height $n$. 
We now write down a detection 
result for $\sigma \in \pi_7E^{hH}$, where 
$H \subseteq\GG_n$ is a closed subgroup. We note  that the image of  
$v_1\in \Ext_{BP_*BP}^{*,*}(BP_*,BP_*/2)$ in $H_c^*(H,E_*/2)$ --
see diagram (\ref{eq:compare-ANSS}) -- is the element $v_1$ 
of \Cref{rem:v1-plus-ss}. 
Consider the Bockstein homomorphism  in cohomology
\[
\bock{4}=\delta^{(4)}_H \colon H_c^0(H,E_8/16) \to H_c^1(H,E_8)
\]
determined by the short exact sequence
\[
\xymatrix{
0 \ar[r] & E_8 \ar[r]^-{\times 16} & E_8 \ar[r]& E_8/16 \ar[r] & 0.
}
\]

\begin{prop}\label{prop:finding-sigma} Let $H \subseteq \GG_n$ be a closed subgroup
and let $R = H_c^0(H,E_0)$. Let $c \in E_8/16$ and assume that
\begin{enumerate}

\item $c \equiv v_1^4$ modulo $2$,  

\item $c$ is invariant under the action of $H$, and 

\item $H_c^0(H,E_8/2)$ is a cyclic $R$-module generated by $v_1^4$. 
\end{enumerate}
Then, up to multiplication by a unit in $R$, the image of $\sigma \in \pi_7E^{hH}$
is detected in the spectral sequence
\[
H_c^s(H,E_t) \Longrightarrow \pi_{t-s}E^{hH}
\]
by the class $\bock{4}(c) \in H_c^{1}(H,E_8)$.
\end{prop} 

\begin{proof} 
As above, let $x \in \pi_8(S^0/16)$ be any class which maps to $\sigma$ under the boundary map
$\pi_8(S^0/16) \to \pi_7S^0$.  We will follow this element through the following diagram of spectral 
sequences
\[
\xymatrix{\Ext^{\ast,\ast} _{BP_\ast BP}(BP_\ast,BP_\ast/16) \ar@{=>}[r]&\pi_\ast (S^0/16)\\
\ar[u]^\cong \Ext^{\ast,\ast}_{MU_\ast MU}(MU_\ast,MU_\ast/16) \ar[d] \ar@{=>}[r] &
\ar[u]_\cong \pi_\ast (S^0/16)\ar[d]\\
H_c^\ast(H,E_\ast/16) \ar@{=>}[r] & \pi_\ast (E^{hH}\wedge S^0/16)\ .
}
\]
To obtain this diagram we combine the diagram of spectral sequences in \eqref{eq:compare-ANSS}
with the diagram of spectral sequences in \eqref{eq:more-compare-ANSS} with
$F_1 = H \subseteq \GG_2 = F_2$.

The class $x$ is detected by $v$ in $\Ext^{\ast,\ast} _{BP_\ast BP}(BP_\ast,BP_\ast/16) $. 
Since $v \equiv v_1^4$ in $BP_\ast/2$, we have that $v$ maps to an  element $w \in H_c^0(H, E_{8}/16)$ congruent to
$v_1^4$ modulo $2$. Again we use the  identification of $v_1$ from \Cref{rem:v1-plus-ss}. 
It now follows that  $x$ is non-zero in $\pi_8(E^{hH}\wedge S^0/16)$. Indeed, the image $w$ of $v$ is a non-zero permanent cycle 
and cannot be hit by a differential since it is in the lowest filtration. 
By the Geometric Boundary Theorem of \Cref{rem:whats-up-GBT},  $\delta^{(4)}(w)$ detects $\sigma$. 
\end{proof}

Now, we turn to the role of $c$. The assumptions imply that the element $c$ generates $H_c^0(H, E_{8}/16)$.
It follows that $c=aw$ for some unit $a \in R/16$ and, therefore, that $c$
detects $a x$ for some unit $a \in R/16$.  Therefore, $\delta^{(4)}(c)$ detects $a\sigma$.

Note that we are not asserting that $\sigma \ne 0 \in \pi_7E^{hH}$. We will take this up at various
points later. See \Cref{sec:k1-local-calcs} and \Cref{sec:k1k2s} and, in particular,
 \Cref{prop:homotopy-v0-k1} and \Cref{thm:what-we-hope-to-prove}.

\section{A review of calculations in $K(1)$-local homotopy theory at $p=2$}\label{sec:k1-local-calcs}

In our main arguments, we will identify wedge  summands of $L_{K(1)}L_{K(2)}S^0$ equivalent to
$L_{K(1)}S^0$ and $L_{K(1)}V(0)$. As background for this, we record here
known results from $K(1)$-local homotopy theory at the prime $2$. For example,
the class $\zeta_1 \in \pi_{-1}L_{K(2)}S^0$ is closely related to the Hopf map 
$\sigma \in \pi_7S^0$  and it is important for us to make this relationship explicit.
We will also discuss $L_{K(1)}V(0)$ and 
$L_{K(1)}Y$ where $Y = V(0) \wedge C(\eta)$ is the spectrum studied by Mahowald in his proof of the 
telescope conjecture at $n=1$ and $p=2$. Here, $C(\eta)$ denotes the cofiber of $\eta \colon S^1 \to S^0$. See Section 2 of \cite{MahImJ}.

None of the material in this section is new; it can be put together 
from \cite{MahImJ}, \cite{MRW}, and \cite{ravgreen} among many sources. 

We begin with some basic calculations in the $K(1)$-local 
Adams-Novikov Spectral Sequence. As noted earlier, we can choose $2$-completed $K$-theory 
for our version of Lubin-Tate theory at height $1$, so we will write $K$ for $E_1$. 

The group $\GG_1 = \ZZ_2^\times$ acts on $K_\ast X = \pi_\ast L_{K(1)}(K \wedge X)$.
We write the action of $k \in \ZZ_2^\times$ using the Adams operations notation; that is, if $x \in 
K_\ast X$, we write $\psi^k(x)$ for the action of $k$ on $x$.

We are interested in the spectral sequence
\[
H_c^s(\ZZ_2^\times,K_tX) \Longrightarrow \pi_{t-s}L_{K(1)}X.
\]

\subsection{Calculating $\pi_\ast L_{K(1)}S^0$ and $\pi_\ast L_{K(1)}V(0)$} We begin with 
these, the most basic spectra. We take up the important auxiliary spectrum $Y = V(0) \wedge C(\eta)$
in \Cref{sec:all-about-y}.

Recall that  $K_{2n} = \ZZ_2u^{-n}$ where $u \in K_{-2}$. The operations $\psi^k$ 
are $\ZZ_2$-linear and 
\begin{equation}\label{eq:adamsop}
\psi^k(u^n) = k^nu^n.
\end{equation} 

\begin{defn}\label{rem:many-alphas}
Here are a few crossed homomorphisms 
defining elements in the first cohomology groups $H_c^1(\Z_2^{\times},K_\ast)$.\footnote{Our notation differs from that of Ravenel in Lemma 2.1 of \cite{RavCoh}.}
\begin{enumerate}

\item $\ravclass_1 \colon \ZZ_2^\times \to \ZZ/2 =K_0/2$ is given by 
\[
\ravclass_1(1 + 2k_0) = \overline{k}_0.
\]
where $\overline{k}_0$ is the mod $2$ reduction of $k_0$. 

\item $\zeta_1 \colon \ZZ_2^\times \to \ZZ_2=K_0$ is defined as the composite 
\[\xymatrix{
\ZZ_2^{\times} \ar[r] &  \ZZ_2^{\times}/ \{\pm 1\} \ar[r]^\cong & (1+4\ZZ_2)  \ar[rr]^-{\frac{1}
{4}\log(-)}_-{\cong} & & \ZZ_2}.
\]
Here the first map is the projection, the second map is the 
inverse of the isomorphism defined by the composition
\[
1+4\ZZ_2 \subseteq \ZZ_2^\times \to \ZZ_2^{\times}/ \{\pm 1\}\ , 
\]
and $\log(1+x) = \sum_{n\geq 1} (-1)^{n+1}{x^n}/{n}$.

\item If $n$ is odd, $\alpha_n\colon \ZZ_2^\times \to K_{2n}$ is given by
\[
\alpha_n(k) = \frac{1}{2}(k^{-n}-1)u^{-n}.
\]

\item If $n = 2^i(2t+1)$ with $i > 0$, then $\alpha_{n/(i+2)}\colon \ZZ_2^\times \to K_{2n}$ is given by
\[
\alpha_{n/(i+2)}(k) = \frac{1}{2^{i+2}}(k^{-n}-1)u^{-n}.
\]
\end{enumerate}
\end{defn}
\begin{rem}
We had a class $\alpha_{4/4}$ above; see after (\ref{eqn:whatv}).
\Cref{lem:some-Bocks-at-1}  below
implies that this new class $\alpha_{4/4}$ will be a unit multiple of the image of the older class, and
we won't need to distinguish between the two as we try to 
find $\sigma$ in the homotopy groups of the localizations of $V(0)$, $Y$ and $V(0)\wedge Y$. 
The class $\zeta_1$ was already discussed in \S 2. See \eqref{def-zeta-n}, and 
\Cref{prop:zeta-permanent}. 
\end{rem}

Since we are at height one, we have $v_1 = u^{-1} \in K_2/2$. 

\begin{lem}\label{lem:reduced-two} The crossed homomorphisms of \Cref{rem:many-alphas}
satisfy the following formulas:
\begin{enumerate}

\item if $n$ is odd, then $v_1^n\ravclass_1 \equiv \alpha_n$ modulo $2$;
\medskip

\item if $n$ is even, then $v_1^n\zeta_1 \equiv \alpha_{n/(i+2)}$ modulo $2$.
\end{enumerate}

\end{lem} 

\begin{proof} As a topological abelian group $\ZZ_2^\times$ is generated by $-1$ and $5$; thus we
need only show that the identities hold when evaluating at these elements. The formula (1) 
is a simple calculation. For formula (2) we use 
\[\frac{1}{4}\log(1+4k) \equiv k\qquad\mathrm{modulo}\ 2. \qedhere\]
\end{proof} 

We also have the following result, which is immediate from \Cref{rem:many-alphas}.
Let $\bock{i}$ be the $i$th Bockstein
\[
\bock{i}\colon H_c^s(\ZZ_2^\times,K_\ast/2^i) \to H_c^{s+1}(\ZZ_2^\times,K_\ast)\ .
\]
Write $\bockn = \bock{1}$. 

\begin{lem}\label{lem:some-Bocks-at-1} We have formulas for the Bockstein homomorphisms:
\begin{enumerate}

\item if $n$ is odd, $\bockn(v_1^n) = \alpha_n$; 
\medskip

\item if $n = 2^i(2t+1)$ with $i > 0$, then $\bock{i+2}u^{-n} = \alpha_{n/(i+2)}$ and
$u^{-n}$ reduces to $v_1^n$ modulo $2$.
\end{enumerate}
\end{lem} 

The next proposition gives some basic detection results. 

\begin{prop}\label{prop:eta-sigma-k1} (1) The class $\eta \in \pi_1S^0$ is non-zero in
$\pi_\ast L_{K(1)}S^0$ detected by $\alpha_1$. 

(2) The class $\sigma \in \pi_7S^0$  is non-zero in $\pi_\ast L_{K(1)}S^0$ detected by
a unit multiple of $\alpha_{4/4}$.
\end{prop}

\begin{proof} Part (1) follows from \Cref{lem:some-Bocks-at-1} and the fact
that $\eta$ is always detected by the Bockstein of $v_1$. Part (2) follows from
\Cref{prop:finding-sigma}, \Cref{lem:some-Bocks-at-1}, and the isomorphisms
\[
\FF_2[v_1^{\pm 1}] \cong K_\ast/2 \cong H_c^0 (\ZZ_2^\times,K_\ast/2)\ .\qedhere
\]
\end{proof}

\begin{rem}\label{rem:unid3} Differentials in the $K(1)$-local Adams-Novikov Spectral Sequence at 
$p=2$ are largely determined by a standard $d_3$. We expand on this observation. 
In \cite[p.430]{ravnovice}, there is a generator
\[
\alpha_3 \in \Ext_{BP_*BP}^{1,6}(BP_*,BP_*) \cong \ZZ/2
\]
which is the Bockstein on $v_1^3$. There, it is shown that
$d_3(\alpha_3) = \eta^4$. Furthermore, $\alpha_3$ reduces to
$\eta v_1^2$ in $ \Ext_{BP_*BP}^{1,6}(BP_*,BP_*V(0))$, so that $\eta d_3(v_1^2) = \eta^4$.
Since there is no $\eta$ torsion
on the $E_3$-term of the Adams-Novikov Spectral Sequence for $V(0)$ in
bidegree $(3,6)$, this forces the differential $d_3(v_1^2) = \eta^3$. 

Since $\alpha_3$ is the Bockstein on $v_1^3$, it maps to the class we named $\alpha_3$ in
\Cref{rem:many-alphas}. See \Cref{lem:some-Bocks-at-1}. 

In general, for a $2$-local $MU$-algebra spectrum $E$, the $E$-based Adams-Novikov Spectral Sequence 
for a spectrum $X$ is a module over the $BP$-based 
Adams-Novikov Spectral Sequence for the sphere. There is a universal 
$d_3$-differential
\[
d_3(\alpha_3 z) = \eta^4 z + \alpha_3 d_3(z).
\]
Further, if $2$ annihilates $E_*(X)$, this gives a universal differential
\[
d_3(\eta v_1^2 z) = \eta^4 z + \eta v_1^2 d_3(z).
\]
If there is no $\eta$-torsion on the $E_3$-term, this implies that $d_3(v_1^2 z) =
\eta^3 z +  v_1^2 d_3(z)$.
\end{rem}

\begin{warn}\label{rem:warningcopy} 
The spectrum $V(0)$ is not a ring spectrum. However, since $BP_\ast V(0) \cong BP_\ast/(2)$ is a graded
commutative comodule algebra, the $E_2$-term of the Adams-Novikov Spectral Sequence for $V(0)$ is often
a bigraded commutative ring. For this reason, we often write $E_2$-terms of Adams-Novikov spectral sequences
for $V(0)$ as graded commutative rings. It is to be implicitly understood that the $E_2$-term is the corresponding
underlying graded abelian group and that the spectral sequence is not multiplicative.
Typically the pages of the spectral sequence loose their ring structure at
$E_3$, where we have $d_3(v_1) = 0$ and $d_3(v_1^2)  =\eta^3$. See \Cref{rem:unid3}. 
These issues are classical, and we hope this doesn't cause confusion.
\end{warn}

\Cref{rem:unid3} implies the following well-known result. See \Cref{fig:LK1V0}. We use that $V(0$) has
a $v_1^4$-self map. The fact that $\zeta_1$ is a permanent cycle was
covered in \Cref{prop:zeta-permanent} and the additive extension is from
\eqref{eq:the-more-basic-rel}. 
In this result, $E(-)$ denotes the exterior algebra over $\FF_2$.

\begin{prop}\label{prop:homotopy-v0-k1} We have an isomorphism 
\[
\FF_2[v_1^{\pm 1},\eta] \otimes E(\zeta_1) \cong
\FF_2[v_1^{\pm 1},\eta] \otimes E(\sigma) \cong H_c^\ast(\ZZ_2^\times,K_\ast/2)
\]
with $v_1^4\zeta_1 = \sigma$. 
All non-zero differentials in the spectral sequence
\[
H_c^\ast(\ZZ_2^\times,K_\ast/2) \Longrightarrow \pi_\ast L_{K(1)}V(0)
\]
are determined by $v_1^4$-linearity, the facts that $v_1$, $\eta$, $\sigma$ and $\zeta_1$ 
are permanent cycles and 
\[
d_3(v_1^2x) = \eta^3 x + v_1^2 d_3(x).
\]
The spectral sequence collapses at $E_4$ and the only additive extensions are implied by
$v_1^4$-linearity, multiplication by $\zeta_1$, and 
\[
2 v_1 = \eta^2\ .
\]
\end{prop} 

\begin{figure}
\center
\includegraphics[width=\textwidth]{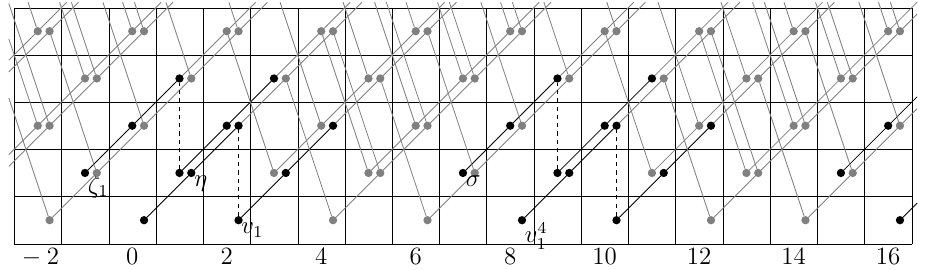}
\captionsetup{width=\textwidth}
\caption{The spectral sequence computing the homotopy groups of $L_{K(1)}V(0)$. 
A $\bullet$ denotes a copy of $\Z/2$. Dashed vertical lines denote exotic multiplication by $2$.}
\label{fig:LK1V0}
\end{figure}

Using \Cref{prop:homotopy-v0-k1} and naturality it is possible to work out
the spectral sequence for the homotopy of $L_{K(1)}S^0$. See \Cref{fig:LK1S}.
Here are the
results in brief. 
\begin{prop}\label{rem:thek1-local-sphere}  
In the spectral sequence
\[
H_c^\ast(\ZZ_2^\times,K_\ast) \Longrightarrow \pi_\ast L_{K(1)}S^0,
\]
there are non-trivial differentials 
\begin{align*}
d_3(\alpha_{4t+3}) &=\eta^3\alpha_{4t+1}\\
d_3(\alpha_{2(2^it+1)/3})&=\eta^3\alpha_{2^{i+1}t/i+3},\quad t\not\equiv 0\quad\mathrm{mod}\ 2\\
d_3(\alpha_{2/3}) &= \eta^3\zeta_1\ .
\end{align*}
\end{prop}
The last formula can be thought of as a case of the second formula with $i=\infty$. 
There are additive extensions as well. In fact, by (\ref{order4one}) we see that $\eta^2\alpha_{4t+1}$ 
must be divisible by $2$. This implies:
\begin{cor}\label{cor:4times}
For any $t\in \ZZ$  
\[
4(2\alpha_{4t+2/3}) = \eta^2 \alpha_{4t+1}. 
\]
\end{cor}

\begin{figure}
\center
\includegraphics[width=\textwidth]{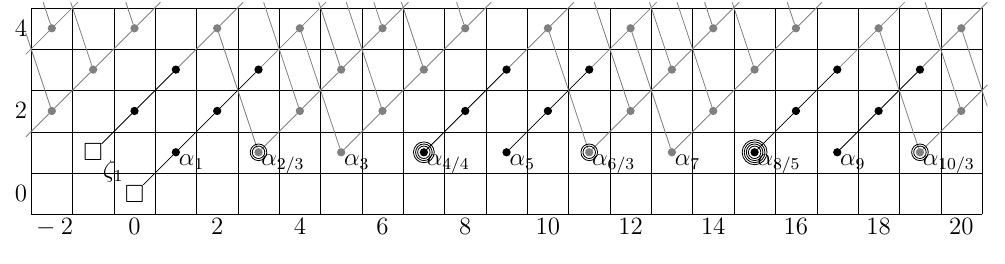}
\includegraphics[width=\textwidth]{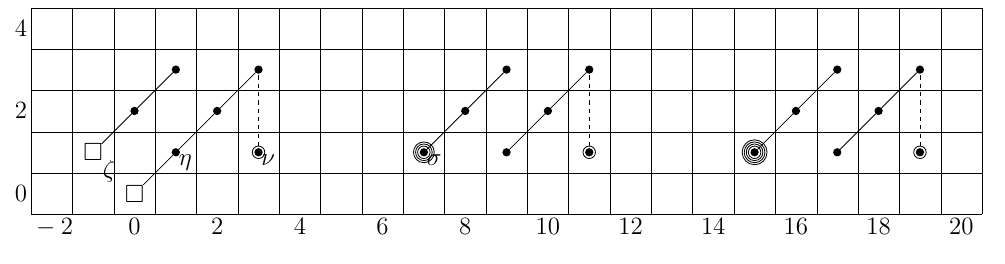}
\captionsetup{width=\textwidth}
\caption{The $E_2$ (top) and $E_{\infty}$ (bottom) pages of the spectral sequence for the homotopy 
of $L_{K(1)}S^0$. Here, a $\Box$ denotes a copy of $\Z_2$, a $\bullet$ denotes a copy of $\Z/2$, a 
\circled{$\bullet$} a copy of $\Z/4$ and so on. Dashed lines denote exotic multiplications by $2$.}
\label{fig:LK1S}
\end{figure}

\Cref{prop:homotopy-v0-k1} has the following consequence. Recall that we are writing
$\sigma$ for both the element in homotopy and the class in the $E_2$-term of the Adams-Novikov
Spectral Sequence which detects it. 

\begin{cor}\label{cor:sigma-squared-zero} (1) In $H_c^\ast(\ZZ_2^\times,K_\ast)$ we have
\[
\alpha_{4/4}^2 = 0 = \sigma^2.
\]

(2)  In $\pi_\ast L_{K(1)}S^0$ we have $\sigma^2 = 0$.
\end{cor}

\begin{proof} For part (1), it is sufficient to prove $\alpha_{4/4}^2=0$ as $\sigma$ is a unit multiple of
$\alpha_{4/4}$. We use that
\[
H_c^s(\ZZ_2^\times,K_\ast) \to H_c^s(\ZZ_2,K_\ast/2)
\]
is an injection if $s > 1$. By part (2) of \Cref{lem:reduced-two} we have 
$\alpha_{4/4} \equiv v_1^4\zeta_1$ modulo
$2$. Since $\zeta_1^2=0$, the result follows. 

For part (2), we deduce from \Cref{prop:homotopy-v0-k1} that
$\pi_{14}L_{K(1)}V(0) = 0$. From the long exact sequence on homotopy groups, it follows that 
$\pi_{14}L_{K(1)}S^0=0$. Alternatively we could read part (2) off of \Cref{fig:LK1S}. \end{proof}

\Cref{prop:homotopy-v0-k1}  is explicit about the Hopf maps $\eta$ and $\sigma$. 
The other Hopf map $\nu$ plays a more subtle role. See \Cref{fig:LK1V0} and \Cref{fig:LK1S}.

\begin{prop}\label{prop:wither-nu} 
(1) Let $\nu \in \pi_3S^0 \cong \ZZ/8$ be a generator. Then $\nu$ is non-zero
in $\pi_\ast L_{K(1)}S^0$ and detected by a unit multiple of $2\alpha_{2/3}$.
\medskip

(2) The class $\nu$ is non-zero in $\pi_\ast L_{K(1)}V(0)$ detected by $v_1\zeta_1\eta^2$. 
\medskip

(3) In $\pi_\ast L_{K(1)}V(0)$, $\nu$ is a multiple of $\eta$.   
\end{prop}

\begin{proof} From the discussion in Remark 4.1.8 we know that $\pi_3(L_{K(1)}S^0)$
sits in a short exact sequence with kernel given by $E_4^{3,6}\cong \Z/2$
generated by $\eta^3$ and quotient given by $E_4^{1,4}\cong \Z/4$ generated by $2\alpha_{2/3}$.
The result follows because of 
\[4\nu=\eta^3\in \pi_3(S^0).\]

For the second statement, we have that $\nu \ne 0 \in \pi_\ast L_{K(1)}V(0)$ and also
that $2\alpha_{2/3} = 0 \in H_c^\ast(\ZZ_2^\times,K_\ast/2)$. It follows that
$\nu$ must be detected by a class in filtration $s=2$ or higher and, by
\Cref{prop:homotopy-v0-k1}, the only class available at $E_\infty$
is $v_1\zeta_1\eta^2$. 
 
For the third statement, we have that $\nu$ is detected by a class which is a multiple of 
$\eta$ in the $E_\infty$-page for $\pi_\ast L_{K(1)}V(0)$.
Since there are no non-zero elements of higher filtration in the $t-s=3$ stem, the claim follows. 
\end{proof} 

\subsection{The $K(1)$-local homotopy of the spectrum $Y$}\label{sec:all-about-y}

Let $C(\eta)$ be the cone on $\eta\in \pi_1S^0$ and let $Y = V(0)\wedge C(\eta)$. 
The spectrum $Y$ is a type $1$ complex with a $v_1$-self map $v_1\colon \Sigma^2 Y \to Y$. 
The map $v_1$ is not unique, but the induced
map $\pi_\ast \Sigma^2  L_{K(1)}Y \to \pi_\ast L_{K(1)}Y$ is independent of the choice. Indeed, for 
any
$k$, $\pi_k L_{K(1)}Y$ is one-dimensional over $\ZZ/2$. See \Cref{prop:k1-of-y} below. 

\begin{rem}\label{rem:Ycomodule} By the construction of $Y$, there
is a cofiber sequence
\begin{align}\label{eq:cofV0Y}
\xymatrix@C=35pt{
\Sigma V(0) \ar[r]^-{V(0)\wedge\eta} & V(0)\ar[r]^-{\jmath}  & Y \ar[r]^-{h} & \Sigma^{2}V(0)
}
\end{align}
which, since $BP_*(V(0)\wedge\eta)=0$, gives rise to a short exact sequence of $BP_\ast BP$-comodules
\[
0 \to BP_\ast V(0) \to BP_\ast Y \to BP_\ast \Sigma^2V(0)\to 0\ .
\]
This is not split, but is the non-zero element
\[
\eta \in \Ext^1_{BP_\ast BP}(\Sigma^2 BP_\ast/2,BP_\ast/2) \cong \ZZ/2.\ 
\] 
\end{rem}

\begin{prop}\label{prop:k1-of-y} We have an isomorphism of 
$\FF_2[v_1^{\pm 1}]$-modules 
\[
\FF_2[v_1^{\pm 1}] \otimes E(\zeta_1) \cong
\FF_2[v_1^{\pm 1}] \otimes E(\sigma) \cong H_c^\ast(\ZZ_2^\times,K_\ast Y)
\]
with $v_1^4\zeta_1 = \sigma$. The spectral sequence
\[
H_c^\ast(\ZZ_2^\times,K_\ast Y) \Longrightarrow \pi_\ast L_{K(1)} Y
\]
collapses. If $\iota\colon S^0 \to Y$ is the inclusion of the bottom cell, then
$\pi_*L_{K(1)}Y$ is a free module over $\F_2[v_1^{\pm 1}]$
on generators $\iota$ and $\iota\sigma$ of degrees $0$ and $7$ respectively.
\end{prop}

\begin{proof} By Landweber exactness and \Cref{rem:Ycomodule} we have 
 a short exact sequence of $\ZZ_2^\times$ modules
\[
0 \to K_\ast V(0) \to K_\ast Y \to K_\ast \Sigma^2V(0)\to 0\ .
\]
Furthermore, in the long exact sequence in group cohomology, the boundary
map is given by multiplication by $\eta$. The claim about $H_c^\ast(\ZZ_2^\times, K_\ast Y)$
now follows from \Cref{prop:homotopy-v0-k1}. 
The spectral sequence for $\pi_*L_{K(1)}Y$ must  collapse for degree reasons.
\end{proof}

Let $\iota \colon S^0 \to Y$ and $p \colon Y \to S^3$ be the inclusion of the bottom cell and 
the collapse to the top cell of $Y$ respectively, and similarly for $\iota\colon S^0 \to V(0)$ and 
$p \colon V(0) \to S^1$.

\begin{prop}\label{prop:2Y} Let $2 \colon Y \to Y$ be the degree $2$ map. Then there is a factoring
\[
\xymatrix{
Y \ar[r]^p \ar@/_1pc/[rr]_2 & S^3 \ar[r]^\nu & Y
}
\]
where $\nu$ is the composite $\xymatrix{S^3 \ar[r]^\nu & S^0 \ar[r]^\iota & Y.}$
After localization, the degree $2$ map
\[
2 \colon L_{K(1)}Y \to L_{K(1)}Y
\]
is null-homotopic.
\end{prop}

\begin{proof} We noted in the proof of \Cref{lem:twoxtildes10} that $\pi_\ast V(0)$ has elements of order
$4$; hence $0 \ne 2: V(0) \to V(0)$. Consider the diagram
\[
\xymatrix{
S^0 \ar[r] \ar[d]_2 & V(0) \ar[r]^-p \ar[d]^2 & S^1 \ar@{-->}[dl]\\
S^0 \ar[r] & V(0).
}
\]
Since $S^0 \xrightarrow{2} S^0 \to V(0)$ is zero, the dotted factoring exists and must be non-zero. Since
$\pi_1V(0) \cong \ZZ/2$ generated by $\eta$ we conclude there is factoring 
\[\xymatrix{
V(0) \ar[r]^-p \ar@/_1pc/[rr]_-{2}& S^1 \ar[r]^-\eta &V(0)\ .
}\]
We then have a diagram 
\[
\xymatrix@R=15pt{
V(0) \ar[r] \ar@/_1.5pc/[dd]_-{2} \ar[d]^p & Y \ar[r]^-h \ar[dd]^2 & \Sigma^2 V(0) \ar@{-->}[ddl]\\
S^1 \ar[d]^-\eta\\
V(0) \ar[r] & Y.
}
\]
Since $\eta=0$ in $\pi_\ast Y$, the dotted factoring exists and gives a factoring of $2 \colon Y \to Y$ as a map
\[
\xymatrix{
Y \ar[r]^-{h}  \ar@/_1pc/[rr]_-{2}& \Sigma^2V(0) \ar[r]^-f &Y.
}
\]
Since $K_* h \colon K_\ast  Y \to K_*  \Sigma^2 V(0)$ is onto and $2 K_*  Y = 0$, we have that
$K_*  f = 0$. Since $\pi_2 Y \to K_2 Y$ is an injection onto the summand generated
by $v_1$, this implies $f$ factors as a composition
\[
\xymatrix{
\Sigma^2V(0) \ar[r]^-{\Sigma^2p} & S^3 \ar[r]^-g &Y.
}
\]
The mod $2$ cohomology of $Y$ is cyclic over the Steenrod algebra and, hence, there can be no splitting
\[
H^*(Y \smsh V(0))\cong H^\ast Y \oplus H^\ast \Sigma Y
\]
as modules over the Steenrod algebra. Thus the order of the identity on $Y$ is not $2$ and we see
$g\ne 0$. Finally, since $\pi_3 Y \cong \Z/2$ generated by
$\nu$, the first statement follows. The second statement follows as 
$\nu = 0$ in $\pi_\ast L_{K(1)}Y$; indeed, we showed in \Cref{prop:wither-nu} that $\nu$ is divisible
by $\eta$ in $\pi_\ast L_{K(1)}V(0)$ and $Y= V(0) \wedge C(\eta)$. 
\end{proof}

Recall that classes $i_0$ and $i_1$ in $\pi_*V(0)\smsh V(0)$ of degree $0$ and $1$ 
respectively were defined in \Cref{rem:pi1V0smV0}.
We abuse notation and let $i_0=(\jmath\smsh V(0))_*i_0$ 
and $i_1=(\jmath\smsh V(0))_*i_1$ be 
the corresponding classes in $\pi_*(Y \smsh V(0))$.

\begin{cor}\label{cor:yandv0}
There is an equivalence
\[
L_{K(1)}(Y \smsh V(0)) \simeq L_{K(1)}Y \vee \Sigma L_{K(1)}Y \ .
\]
Furthermore,
$\pi_*L_{K(1)}( Y \smsh V(0))$ is a free module over $\F_2[v_1^{\pm 1}]\otimes E(\sigma)$ 
on generators $i_0$ and $i_1$ in degrees $0$ and $1$. 
\end{cor}
\begin{proof} 
This is an immediate consequence of \Cref{prop:2Y}.
\end{proof}


\section{Height $2$ cohomology calculations}\label{sec:height2}

In this section we collect together some calculations of the group cohomology of $\GG_2$ and
many of its closed subgroups. Here we will focus on cohomology with coefficients in the Witt vectors $\WW = W(\FF_4)$ 
and related quotient and subrings. The main result is \Cref{thm:cohS21} which gives a structure result for
$H_c^\ast(\SS_2^1,\ZZ_2)$. 

We will extend the calculations to other coefficients in the next 
section, when we prove \Cref{conj:chrom-vanish} in the case $n=p=2$. 

\subsection{Preliminaries and recollections}

A key to many of  our calculations is the behavior and properties of the classes in
$H_c^1(\GG_2,\FF_2)$; these play a central role in the story we are telling here. This is also a point 
where the prime $2$ has extra phenomena not seen at odd primes. The first part of this
section is devoted to analyzing these cohomology classes. We will also include some
material on the cohomology of some Poincar\'e duality subgroups of $\SS_2$
collected from \cite{Paper1}. 

In \eqref{det-defined} we defined the
extended determinant map 
\[
\mathrm{det}\colon \GG_2 \longr \ZZ_2^\times\ . 
\]
We use it in the following definition of some key cohomological classes.

\begin{defn}\label{rem:hone} 
Let 
\[
p_1\colon \ZZ_2^\times \to \ZZ_2^\times/\{\pm 1\} \cong \ZZ_2\qquad\mathrm{and}\qquad
p_2\colon \ZZ_2^\times \to \ZZ/4^\times \cong \{\pm 1\}
\]
be the natural projections. 
Using the explicit isomorphism $\ZZ_2^\times/\{\pm 1\} \cong \ZZ_2$ of
\Cref{rem:many-alphas}, define surjective homomorphisms
\footnote{{\bf Warning!} In Lemma 2.1 of Ravenel \cite{RavCoh} the
mod $2$ reductions of these
classes have the names $\zeta_2 + \rho_2$ and $\zeta_2$ respectively.} 
\begin{equation}\label{eqn:honeclasses}
\myzeta_{2}:=p_1\circ \det\colon \GG_2 \longr \ZZ_2\qquad\mathrm{and}\qquad
\myoneb_{2}:= p_2\circ \det\colon \GG_2 \longr \ZZ/2 \ .
\end{equation}
By abuse of notation, we call the corresponding cohomology classes by the same
names; thus, we have classes $\myzeta_2 \in H_c^1(G,\ZZ_2)$ and
$\myoneb_2 \in H_c^1(G,\FF_2)$ for any closed subgroup $G \subseteq \GG_2$. 
We write $\myzeta =\myzeta_2$ and $\myoneb = \myoneb_2$ when there can be no 
confusion.

We will write $\myzeta_2 \in H_c^1(\GG_2,\FF_2)$
to be the reduction of  $\myzeta_2 \in H_c^1(\GG_2,\ZZ_2)$. We define
\[
\mytwob \in H_c^2(\GG_2,\ZZ_2)
\]
to be the image of $\myoneb_2$ under the Bockstein $H_c^1(\GG_2,\FF_2) \to
H_c^2(\GG_2,\ZZ_2)$.
\end{defn}

\begin{rem}
The cohomology class $\myzeta_2$ is the reduced determinant discussed earlier in 
\eqref{norm-defined} and \eqref{def-zeta-n}; see also \Cref{prop:zeta-permanent}.
The class $\myoneb_2$ is the analog of the class $\ravclass_1$ defined in
\Cref{rem:many-alphas} for $n=1$. 

Looking back at \Cref{rem:many-alphas} we see 
that the projections $p_1$ and $p_2$ of \Cref{rem:hone} have arisen before. In fact
$p_1 = \zeta_1$ and $p_2 = \chi_1$, up to the isomorphism $\ZZ/4^\times \cong \ZZ/2$. 
We can write
\[
\zeta_2=\zeta_1\circ \det\qquad \mathrm{and}\qquad  \chi_2=\chi_1\circ \det.
\]
\end{rem}

Here are some more refined
details about the behavior of $\zeta_2$ and $\chi_2$. From
\Cref{rem:all-the-splittings} we have decompositions $K \rtimes G_{24} \cong \SS_2$ and
$K^1 \rtimes G_{24} \cong \SS_2^1$. In  \Cref{rem:g24prime}, we introduced elements $\pi$ and
$\alpha$ of  $\WW^{\times}\subseteq \SS_2$  that satisfy $\det(\pi)=3$ and
$\det(\alpha)=-1$. Furthermore, $\alpha$ and $\pi$ are elements of $K \subseteq \SS_2$.
\medskip

(1) Since $\mathrm{det}(\alpha) = -1$, it follows that $\myoneb_2(\alpha) = -1$ and
$\myzeta_2(\alpha)=0$. Thus $\alpha \in K^1$ and $\myoneb_2 \colon \SS_2^1 \to \FF_2$
restricts to a non-zero class of the same name in  $H_c^1(\SS_2^1,\FF_2)$ and even in 
$H_c^1(K^1,\FF_2)$. By definition $\SS_2^1$ is in the kernel of the reduced determinant; hence,
$\myzeta_2$ restricts to zero in $H_c^1(\SS_2^1,\ZZ_2)$.
\medskip

(2) Note also that $\myzeta_2^2 = 0$ in
cohomology, wherever it appears, as it is the image of a 
generator of $H_c^1(\ZZ_2,\ZZ_2)$. The exponent of $\myoneb_2$ is an important point 
addressed below in \Cref{prop:b0cubedbig}.
\medskip

We now begin our calculations.

\begin{lem}\label{lem:hone-of-stwo} (1) The cohomology group $H_c^1(\SS_2,\FF_2)$ is of dimension
$2$ over $\FF_2$ with basis $\myoneb_2$ and $\myzeta_2$. 

(2) The cohomology group $H_c^1(\SS_2,\ZZ_2)$ is free of rank one over $\ZZ_2$
generated by $\myzeta_2$. 
\end{lem}

\begin{proof} 
The first statement follows from Theorem 6.3.12 of Ravenel \cite{ravgreen} by taking invariants with respect to the residual action of $
\FF_4^{\times}$. Alternatively, it follows from Corollary 5.4 of \cite{VIASM} by noting that $H_c^1(\SS_2,\FF_2)$ is dual to 
$H_1(\SS_2, \FF_2)$ and that
\[
H_1(\SS_2, \FF_2) \cong H_1(S_2, \FF_2)_{\FF_4^{\times}}.
\]
Specifically, in  Proposition 5.2, we can take $c=\omega$. The class represented by $1+4\omega$ and the class represented
by $1+\omega S^2$ are invariant for the residual action of $\F_4^{\times}$ (given by conjugation by $\omega$) while
the generators $1+\omega S$, $1+\omega^2 S$ and their product gets cyclicly permuted by $\FF_4^{\times}$.
The generator we call $\myzeta_2$ is dual to the class detected by $1+ 4 \omega$ and the generator we call
$\myoneb_2$ is dual to the class detected by $1+\omega S^2$ in the proof of Proposition 5.2 of \cite{VIASM}.

We will see below in \Cref{prop:orderb01}, that $0 \ne \myoneb_2^2 \in H_c^2(\SS_2^1,\FF_2)$, and hence in 
$H_c^2(\SS_2,\FF_2)$. The second statement then follows from examining the long exact sequence on coefficients induced by 
\[
0 \to \Z_2 \xrightarrow{2} \Z_2 \to \FF_2  \to 0
\] and noting that the composite of the connecting homomorphism with reduction modulo $2$ (i.e., the Bockstein) is the squaring 
operation in cohomology.
\end{proof}

We now add some further recollections on the cohomology of the subgroups $K$ and $K^1$ of $\SS_2$.
Since $\det(\pi)=3$, $\zeta_2(\pi)$ maps to a topological  generator of $\Z_2^{\times}/\{\pm 1\}$. Then since $\pi \in K$
the composition 
\[
\xymatrix{
K \ar[r]^\subseteq& \SS_2 \ar[r]^{\myzeta_2}& \ZZ_2
}
\]
remains surjective and defines a non-zero cohomology class $\myzeta_2 \in H_c^1(K,\ZZ_2)$.
Any choice of splitting of this surjection will define isomorphisms $\SS_2^1 \rtimes \ZZ_2
\cong \SS_2$ and $K^1 \rtimes \ZZ_2 \cong K$. An obvious choice of a splitting sends
a generator of $\ZZ_2$ to $\pi$.

The following can be found in Corollary 2.5.12 and Theorem 2.5.13 of \cite{Paper1}.\footnote{In Section 2.5 of \cite{Paper1}, the restriction of $\myoneb$ to $ 
H_c^1(K^1,\FF_2)$ is denoted by $x_0$ and the restriction of $\myzeta$ to $H_c^1(K,\FF_2)$ by $x_4$. 
We use the same name for the restrictions as for the original classes.}

\begin{lem}\label{lem:reph1} (1) The subgroups $K^1$ and $K$ are oriented
Poincar\'e duality groups of dimensions $3$ and $4$ respectively. 
\medskip

(2) There is an isomorphism
 \[
H_c^\ast (K,\FF_2) \cong 
\FF_2[\myzetaK,\myonebK,x_1,x_2]/(\myzetaK^2,\myonebK^2,x_1^2 + \myonebK x_1,x_2^2 + \myonebK x_2)
\]
where $\myzetaK, \myonebK$, $x_1$ and $x_2$ are in degree $1$.
\medskip

(3) There is an isomorphism
 \[
H_c^\ast (K^1,\FF_2) \cong 
\FF_2[\myonebK,x_1,x_2]/(\myonebK^2,x_1^2 + \myonebK x_1,x_2^2 + \myonebK x_2)
\]
where $\myonebK$, $x_1$ and $x_2$ are in degree $1$.
\end{lem}

We also write down the cohomology of
$Q_8 \rtimes \FF_4^\times \cong G_{24}$ and $C_6$. We collect these results in the following remark.

\begin{rem}\label{rem:coho-of-g24} Recall that $Q_8$ has periodic
cohomology with a periodicity  class $k \in H^4(Q_8,\ZZ_2)$ of order $8$. See, for example, \cite[Ch. IV. Lemma 2.10]{MR2035696}.
We will also write $k \in 
H^4(Q_8,\FF_2)$ for the reduction of $k$.

We have an isomorphism
\[
\FF_2[x,y,k]/(x^2+xy+y^2,x^2y+xy^2) \cong H^\ast(Q_8,\FF_2)
\]
where $x$ and $y$ are in degree $1$. For a choice $\omega \in \FF_4$ of a primitive cube root of 
unity we have $\omega_\ast x = y$, $\omega_\ast y = x+y$, and $\omega_\ast k = k$; 
from this it follows that there is an isomorphism
\[
\FF_2[z,k]/(z^2) \cong H^\ast(Q_8,\FF_2)^{\FF_4^\times} \cong H^\ast(G_{24},\FF_2)
\]
where $z=x^2y=xy^2$. Finally, since $k$ has order $8$ in $H^4(Q_8,\ZZ_2)$, the Universal 
Coefficient Theorem gives an isomorphism
\[
 \ZZ_2[k]/(8k)  \cong H^\ast (G_{24},\ZZ_2)\ .
\]

Recall further that
\[
H^\ast(C_6,\FF_2) \cong \FF_2[h]
\]
for a class $h$ in degree $1$ and that 
\[
H^\ast(C_6,\ZZ_2) \cong \ZZ_2[g]/(2g)
\]
for a class $g$ in degree $2$ which is the image of $h$ under the connecting homomorphism for
$\ZZ_2 \xra{2} \ZZ_2 \to \FF_2$.
The inclusion $C_6 = \{\pm 1\} \times \FF_4^\times \subseteq G_{24}$ yields a map on cohomology
\[
H^\ast (G_{24},\ZZ_2) \cong \ZZ_2[k]/(8k) \to \ZZ_2[g]/(2g) \cong H^\ast(C_6,\ZZ_2)
\]
sending $k$ to $g^2$. With $\FF_2$ coefficients the map
\[
H^\ast (G_{24},\F_2) \cong \FF_2[z,k]/(z^2) \to \FF_2[h] \cong H^\ast(C_6,\FF_2)
\]
sends $z$ to $0$ and $k$ to $h^4$.
\end{rem}

\subsection{The cohomology of $\SS_2^1$}

In this section we use the algebraic duality spectral sequence of \Cref{sec.dualityres}  to 
calculate the integral and mod $2$ cohomology of $\SS_2^1$ as graded modules
over $H^\ast(G_{24},\ZZ_2)$. From \Cref{rem:algdualresSS} we have
\begin{equation}\label{SS2new}
\xymatrix{E_1^{p,q} = H^q(F_p, M ) \ar@{=>}[r] & H_c^{p+q}( \mathbb{S}_2^1,M)}\ .
\end{equation}
We are particularly interested in the cases $M=\ZZ_2$ and $M=\FF_2$.

This spectral sequence has a split edge homomorphism.  
The augmentation map $\ZZ_2[[\SS_2^1/G_{24}]] \to \ZZ_2$ induces, through the isomorphism
of \eqref{eqn:schapiro1}, an edge homomorphism
\[
H_c^\ast(\SS_2^1,\ZZ_2) \to H^\ast(G_{24},\ZZ_2)
\]
of the spectral sequence \eqref{SS2new}. This is induced by the inclusion of 
$G_{24} \subseteq \SS_2^1$. By \Cref{rem:all-the-splittings}
there is  a projection $\SS_2^1 \to \SS_2^1/K^1 \cong G_{24}$ which splits this inclusion, so we 
immediately have the following result.

\begin{lem}\label{lem:gexists} Let $M = \ZZ_2$ or $\FF_2$. The map of algebras
\[
H_c^\ast(\SS_2^1,M) \to H^\ast(G_{24},M)
\]
induced by the inclusion of the subgroup $G_{24}$ has an algebra splitting.
\end{lem} 

From \Cref{rem:coho-of-g24} and \Cref{lem:gexists} we get an injective map
\[
\ZZ_2[k]/(8k) \cong H^\ast(G_{24},\ZZ_2) \to H_c^\ast(\SS_2^1,\ZZ_2).
\]
We confuse $k$ with its image in $H_c^4(\SS_2^1,\ZZ_2)$
and in the cohomology of its subgroups. An example of such a subgroup is $G'_{24}$, the conjugate group defined in \eqref{eq:g24prime}. Then the composition
\[
H^\ast(G_{24},\ZZ_2) \to H_c^\ast(\SS_2^1,\ZZ_2) \to H^\ast(G'_{24},\ZZ_2)
\]
is an isomorphism; this follows from the fact that the inclusion $G'_{24}\subseteq \SS_2^1$ also splits the
projection $\SS_2^1 \to G_{24}$. 
Also, as in \Cref{rem:coho-of-g24}, the map 
\[
H^\ast(G_{24},\ZZ_2) \to H_c^\ast(\SS_2^1,\ZZ_2) \to H^\ast(C_6,\ZZ_2) \cong \ZZ_2[g]/(2g)
\]
sends $k$ to $g^2$.  

It will be useful to compare the algebraic
duality spectral sequence for $\SS_2^1$ to a spectral sequence for a quotient group.
Let $C_2 = \{ \pm 1\} \subseteq \SS_2$ be the central subgroup of order $2$.
For any group $G \subseteq \Sn_2$ which contains $C_2$ define $PG = G/C_2$. 
Note that $C_2$ is a subgroup of all of the groups $\SS_2^1$, $C_6$, $G_{24}$, and $G_{24}'$.
Whenever $C_2 \subseteq F$ we have isomorphisms of $\SS_2^1$-coset spaces
$\SS_2^1/F \cong P\SS_2^1/PF$ and, hence isomorphisms of continuous $\SS_2^1$-modules 
\[
\ZZ_2[[\SS_2^1/F]] \cong \ZZ_2[[P\SS_2^1/PF]]\ .
\]
The resolution of \Cref{thm:bob} is in fact constructed as a resolution of
continuous $P\SS_2^1$-modules
\begin{align}\label{eq:prof-dual-res}
0 \to \Z_2[[P\mathbb{S}_{2}^1/PF_3]]  &\mathop{\longr}^{\partial_3}  
\Z_2[[P\mathbb{S}_{2}^1/PF_2]]\nonumber\\
&\mathop{\longr}^{\partial_2}  \Z_2[[P\mathbb{S}_{2}^1/PF_1]]   
\mathop{\longr}^{\partial_1} \Z_2[[P\mathbb{S}_{2}^1/PF_0]]  \to \Z_2 \to 0\ .
\end{align}
We have
\begin{align*}
PF_0 &\cong A_4 := (C_2 \times C_2) \rtimes C_3\\ 
PF_1 &\cong PF_2 \cong C_3\\
PF_3 &\cong  A'_4
\end{align*}
with $A'_4=\pi A_4\pi^{-1}$, where $\pi = 1+2\omega \in \WW^{\times}$. Note, in particular, 
that if $p=1$ or $2$, then
$\Z_2[[P\mathbb{S}_{2}^1/PF_p]] \cong \Z_2[[P\mathbb{S}_{2}^1/C_3]]$
is a projective $\ZZ_2[[P\SS_2^1]]$-module. This often makes arguments
with this resolution simpler. For example, if $M$ is a profinite
$P\SS_2^1$ module we have an analog of the algebraic duality
 spectral sequence
\begin{equation*}
H^q(PF_p,M) \Longrightarrow H_c^{p+q}( P\mathbb{S}_2^1,M)
\end{equation*}
with 
\begin{equation}\label{rem:comtops21}
E_1^{p,q} = H^q(C_3,M) = 0
\end{equation} if $p=1$, $2$ and $q > 0$.

Finally, these considerations give a diagram of spectral sequences for any
profinite $P\SS_2^1$-module $M$
\begin{equation*}
\xymatrix{
H^q(PF_p, M ) \ar[d]\ar@{=>}[r] & H_c^{p+q}( P\mathbb{S}_2^1,M)\ar[d]\\
H^q(F_p, M )  \ar@{=>}[r] & H_c^{p+q}(\mathbb{S}_2^1,M)
} 
\end{equation*}
with the vertical maps are induced by the evident quotient homomorphisms.

\begin{lem}\label{lem:PS21collapse}
Let $M = \ZZ_2$ or $\FF_2$. The spectral sequence
\[\xymatrix{H^q(PF_p, M ) \ar@{=>}[r] & H_c^{p+q}( P\mathbb{S}_2^1,M)}\]
collapses at the $E_2$-page. Furthermore, if $M=\FF_2$, it collapses at the $E_1$-page.
\end{lem}
\begin{proof}
Since $H^q(PF_p, M) = 0$ for
$p=1,2$ and $q > 0$, and since $H^\ast (PF_0, M)$ is
a retract of $H_c^\ast (P\SS_2^1,M)$, the spectral sequence collapses at $E_2$.
The $d_1$-differential is induced by the maps $\partial_i$ of \Cref{thm:bob}.
By part (2) of \Cref{thm:bob} we have that $d_1 \equiv 0$ modulo $2$, so 
the spectral sequence collapses at $E_{1}$ for $M=\FF_2$.
\end{proof}
 
It is also useful to compare these spectral
sequences to yet another one defined for a subgroup of $\SS_2^1$.
Recall there is a semidirect product decomposition
$\SS_2^1 \cong K^1 \rtimes G_{24}$; this implies a decomposition
$P\SS_2^1 \cong K^1 \rtimes A_4$. 
Thus, the resolution (\ref{eq:prof-dual-res}) is a resolution of projective
$\ZZ_2[[K^1]]$-modules. For any profinite $P\SS_2^1$-module $M$, we get a diagram of spectral 
sequences  
\begin{equation}\label{diag:PS21toK}
\xymatrix{
E_1^{p,q} \cong \Ext^q_{\ZZ_2[[P\SS_2^1]]}(\ZZ_2[[P\SS_2^1/PF_p]], M) \ar[d]\ar@{=>}[r] &
H_c^{p+q}( P\mathbb{S}_2^1,M) \ar[d]\\
E_1^{p,q} \cong  \Ext^q_{\ZZ_2[[K^1]]}(\ZZ_2[[P\SS_2^1/PF_p]], M)  \ar@{=>}[r] & H_c^{p+q}(K^1,M)
}
\end{equation}
and in the bottom row the $\Ext$ groups vanish if $q >0$. This allows us to prove
the following lemma. Note that part (1) of \Cref{lem:reph1} implies $H_c^3(K^1,M) \cong M$
for $M$ either $\ZZ_2$ or $\FF_2$.

\begin{lem}\label{lem:comtoKat3} Let $M = \ZZ_2$ or $\FF_2$. 
The sequence 
\[
\xymatrix{
0 \to H^3(A_4,M) \ar[r] & H_c^3(P\SS_2^1,M) \ar[r] &H_c^3(K^1,M) \cong M \to 0
}
\]
is split short exact, where the maps are induced by the projection to $A_4$ and the inclusion
of $K^1$ in $P\SS_2^1 \cong K^1 \rtimes A_4$. Furthermore the restriction homomorphism
\[
H_c^3(\SS_2^1,M) \longrightarrow H_c^3(K^1,M)
\]
is split surjective. 
\end{lem}

\begin{proof}
The second statement follows from the first since the map $H_c^3(P\SS_2^1, M) \to H_c^3(K^1, M)$ 
factors through $H_c^3(\SS_2^1, M)  \to H_c^3(K^1, M)$.

For the first statement, we use the map of spectral sequences \eqref{diag:PS21toK}. 
Since $\ZZ_2[[P\SS_2^1/A'_4]] \cong \ZZ_2[[K^1]]$ as a $K^1$-module, 
$\Hom_{\ZZ_2[[K^1]]}(\ZZ_2[[P\SS_2^1/A'_4]], M) \cong M$
and, at $E_1^{3,0}$, the map \eqref{diag:PS21toK} is the inclusion of the invariants
\[M^{A'_4} \cong  \Hom_{\ZZ_2[[P\SS_2^1]]}(\ZZ_2[[P\SS_2^1/A'_4]], M) \to 
\Hom_{\ZZ_2[[K^1]]}(\ZZ_2[[P\SS_2^1/A'_4]], M) \cong M.\]
The action of $A_4'$ on $M$ is trivial, so this is an isomorphism. 

By part (2) of \Cref{thm:bob}, $E_2^{3,0} \cong E_1^{3,0}$ for both $K^1$ and $P\SS_2^1$ 
as the differential $d_1\colon E_1^{2,0} \to E_1^{3,0} $ is induced by the map $\partial_3$, 
which is zero modulo $IK^1$. Both spectral sequences collapse at the $E_2$-term. 
For $K^1$ this follows for degree reasons and for $P\SS_2^1$, this is \Cref{lem:PS21collapse}.
So the map in \eqref{diag:PS21toK} at $E_{\infty}^{3,0}$ is an isomorphism. Finally, for $P\SS_2^1$
we have $E_{\infty}^{2,1} = E_{\infty}^{1,2}=0$ by \eqref{rem:comtops21}.
Since $E_{\infty}^{0,3} \cong H^3 (PF_0, M)$ is
a retract of $H_c^3(P\SS_2^1,M)$, the extension is split.
\end{proof} 

\begin{rem}\label{rem:HS1action} In results to follow we will use
extra structure on the algebraic duality  spectral sequence.
For all profinite $\SS_2^1$-modules $M$, there is a natural action of
\[
\Ext^\ast_{\ZZ_2[[\SS_2^1]]}(\ZZ_2,\ZZ_2) \cong H_c^\ast(\SS_2^1,\ZZ_2)
\]
on $\Ext^\ast_{\ZZ_2[[\SS_2^1]]}(M,\ZZ_2)$ so, in particular, the algebraic duality  
spectral sequence 
\[
H^q(F_p,\ZZ_2) \Longrightarrow H_c^{p+q}(\SS_2^1,\ZZ_2)
\]
is a spectral sequence of $H_c^\ast(\SS_2^1,\ZZ_2)$-modules.
Furthermore, the action of
$H_c^{\ast}(\SS_2^1,\ZZ_2)$ on $E_1^{p,\ast} \cong H^\ast(F_p,\ZZ_2)$ is through the restriction
homomorphism induced by the inclusion $F_p \subseteq \SS_2^1$.

We go into more detail on this algebra action
in \Cref{sec:bock}; see the material after Lemma \ref{lem:bockstein1}.
\end{rem} 

\begin{rem}\label{rem:figure1explained} We are now ready to use the algebraic 
duality  spectral sequence to compute the cohomology of $\SS_2^1$ with $\ZZ_2$ and 
$\FF_2$ coefficients.
In arguments below we may use the notation $E_r^{\ast,\ast}(\FF_2)$ or  $E_r^{\ast,\ast}(\ZZ_2)$ for the
algebraic duality  spectral sequences converging to $H_c^\ast(\SS_2^1,\FF_2)$ or
$H_c^\ast(\SS_2^1,\ZZ_2)$ respectively.

\Cref{fig:ADSS}  displays the $E_1$-term with $\FF_2$-coefficients in the left
column, the $E_1$-term with
$\ZZ_2$-coefficients in the middle column, and the $E_\infty$-term with $\ZZ_2$-coefficients in the
right column. We will have $E_1= E_\infty$ for $\FF_2$-coefficients. See \Cref{cor:cohS21at2}. 

We explain the notation in this figure; let $M=\FF_2$ or $\ZZ_2$. 
Then a vertical subcolumn in an $E_1$-term
displays a copy of $H^\ast(F_p,M)$ where $F_p=G_{24}$ if $p=0$, $C_6$ if $p=1$ or $2$, or $G'_{24}$
if $p=3$. The cohomology rings of $G_{24}$ and $C_6$ were discussed in \Cref{rem:coho-of-g24}.
The symbol \circled{\circled{$\bullet$}} denotes a copy of $\ZZ/8$, while a $\bullet$ denotes a copy
of $\ZZ/2$. A $\square$ denotes a copy of $\ZZ_2$.

We write $\assa \in H^0(F_0, M)$, $\assb  \in H^0(F_1, M)$, $\assc  \in H^0(F_2, M)$,  
$\assd  \in H^0(F_3, M)$ for the generators corresponding to the ring units.  This notation is used to facilitate
references to \cite{Paper2}. Thus, for example, $E^{1,\ast}_1 \cong H^\ast(C_6,M)\assb$
as a module over $H_c^\ast(\SS_2^1,M)$. 

The generators of $\la_0$, $\lb_0$, $\lc_0$, and $\ld_0$ are, strictly speaking, each an element of
$H^0(F_p,M)$ for some $p$. In the next few results, when elements survive to the $E_2$-term, we will conflate them with their
images under the edge homomorphism
\[
E_2^{p,0} \longrightarrow E_{\infty}^{p,0} \subseteq H_c^p(\SS_2^1,M)
\]
of the algebraic duality spectral sequence. 
\end{rem}

\begin{figure}
\center
\begin{minipage}{0.3\textwidth}
\includegraphics[width=\textwidth]{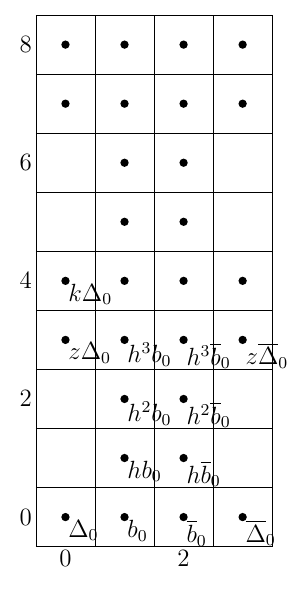}
\end{minipage}
\begin{minipage}{0.3\textwidth}
\center
\includegraphics[width=\textwidth]{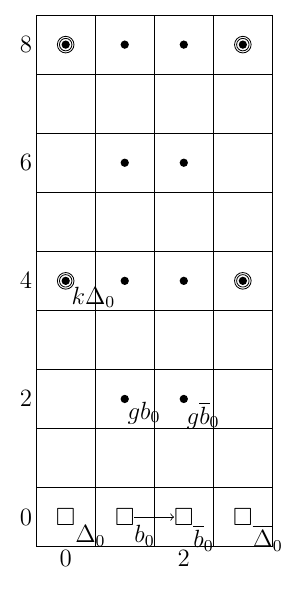}
\end{minipage}
\begin{minipage}{0.3\textwidth}
\center
\includegraphics[width=\textwidth]{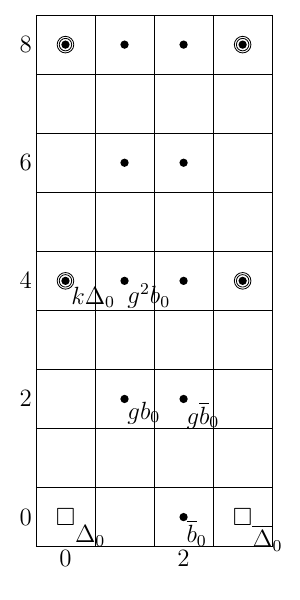}
\end{minipage}
\captionsetup{width=\textwidth}
\caption{The $E_1=E_{\infty}$-term (left) of the ADSS for $H_c^*(\Sn_2^1,\F_2)$ and the
$E_1$-term (center) and $E_2=E_{\infty}$-term (right) of the ADSS for $H_c^*(\Sn_2^1,\Z_2)$.}
\label{fig:ADSS}
\end{figure}

Before getting to our main theorems, we give some preliminary results about the classes $\assb$,
$\assc$, and $\assd$ in $E_1^{\ast,0}$ for both $\FF_2$ and $\ZZ_2$. Recall from 
\Cref{rem:hone} that $\mytwob \in H_c^2(\GG_2,\ZZ_2)$ is the Bockstein on
$\myoneb = \myoneb_2  \in H_c^1(\GG_2,\FF_2)$.

\begin{lem}\label{whatisd0anyway} (1) The class $\assb \in E_1^{1,0}(\FF_2)$ detects $\myoneb \in
H_c^1(\SS_2^1,\FF_2)$.
\medskip

(2) The class $\assc  \in E_1^{2,0}(\FF_2)$ detects $\myoneb^2 \in H_c^2(\SS_2^1,\FF_2)$, and the
class $\assc  \in E_1^{2,0}(\ZZ_2)$ detects the class $\mytwob \in H_c^2(\SS_2^1,\ZZ_2)$.
\medskip

(3) There is a torsion-free class $e \in H_c^3(\SS_2^1,\ZZ_2)$ 
detected by $\assd \in  E_2^{3,0}(\ZZ_2)$ which restricts to a generator of
\[
H_c^3(K^1,\ZZ_2) \cong \ZZ_2 .
\]
\end{lem}

\begin{proof} For part (1) we see that in the spectral sequence for $H_c^\ast (\SS_2^1,\FF_2)$
(the left-hand column in \Cref{fig:ADSS}) we have $E_\infty^{1,0}(\FF_2) = \FF_2$ 
generated by $\assb$ and $E_\infty^{0,1}= 0$; hence $H_c^1(\SS_2^1,\FF_2) \cong \FF_2$. 
Since $\myoneb \in H_c^1(\SS_2^1,\FF_2)$ is not zero, it
must be detected by $\assb$. 

For part (2) we recall from  part (2) of \Cref{thm:bob} that in the integral spectral sequence
(the central column in \Cref{fig:ADSS}) 
\[
d_1\colon E_1^{1,0}(\ZZ_2) \to E_1^{2,0}(\ZZ_2) 
\]
is multiplication by $2$. Indeed, we are working $2$-locally, $\partial_1 \equiv 2$ modulo $(8, IS_2^1)$ and the action of $IS_2^1$ on $\Z_2$ is trivial.
Thus $H_c^2(\SS_2^1,\ZZ_2) \cong \ZZ/2$ generated
by a class detected by $\assc$. This implies that the connecting homomorphism
$H_c^1(\SS_2^1,\FF_2) \to H_c^2(\SS_2^1,\ZZ_2)$ must be non-zero. Since $H_c^1(\SS_2^1,\FF_2)\cong 
\FF_2$ generated by $\myoneb$, we must have that the generator of $H_c^2(\SS_2^1,\ZZ_2)$ 
is $\mytwob$. From this it follows that $\assc$ detects $\myoneb^2 \in H_c^2(\SS_2^1,\FF_2)$.

This leaves part (3). The integral algebraic duality  spectral sequence 
has an edge homomorphism
\[
\ZZ_2 \cong \Hom_{\ZZ_2[[\SS_2^1]]}(\ZZ_2[[\SS_2^1/G'_{24}]],\ZZ_2) 
= E_1^{3,0} \to H_c^3(\SS_2^1,\ZZ_2).
\]
This map is injective. To see this, note that $\assc$ is a
permanent cycle. Thus, the differential $d_1\colon E_1^{2,0} \to E_1^{3,0}$ is zero. All other differentials with
target $E_r^{3,0}$ have zero source.

Let $e$ be the image of $\assd$.  That $e$ restricts
to a generator of $H_c^3(K^1,\ZZ_2)$ follows from \Cref{lem:comtoKat3}.
\end{proof} 

We can now give the calculation of $H_c^\ast(\SS_2^1,\ZZ_2)$ as
a module over $H^\ast(G_{24},\ZZ_2)$. We give the integral calculation first 
as there are fewer possible differentials. Some of the generators 
in this result are written as products; this is meant only to be evocative
of their antecedents in the spectral sequence. Thus, for example, $g\myoneb$ is not a product in the
cohomology ring $H_c^\ast(\SS_2^1,\ZZ_2)$, but a class detected by $g\assb$.

\begin{thm}\label{thm:cohS21} As an $H^\ast(G_{24},\ZZ_2)\cong \ZZ_2[k]/(8k)$-module
$H_c^\ast(\SS_2^1,\ZZ_2)$ is generated by elements
\[
1,\ g\myoneb,\ g^2\myoneb, \ \mytwob,\ g\mytwob,\ e
\]
of degrees $0$, $3$, $5$, $2$, $4$, and $3$ respectively
and subject only to the relations
\[
2g\myoneb = 2g^2\myoneb =2\mytwob = 2g\mytwob=0.
\]
\end{thm}

\begin{proof} As a reminder, we are using the algebraic
duality  spectral sequence of \Cref{rem:algdualresSS} with 
$M=\ZZ_2$. The $E_1$-page is determined as an 
$H^\ast(G_{24},\ZZ_2)$-module by \Cref{rem:coho-of-g24} and the result is displayed
in the center column of \Cref{fig:ADSS}. Then
\Cref{thm:bob} part (2) implies that $0=d_1 \colon E_1^{p,\ast} \to E_1^{p+1,\ast}$ if
$p=0$ or $p=2$ and the same result implies
\[
d_1 \colon E_1^{1,q} \to E_1^{2,q}
\]
is multiplication by $2$ and, hence, non-zero only if $q=0$. From $E_2$ onwards, the spectral 
sequence is too sparse for differentials. The result will now follow if we can show that there are no extensions.

Because we have a spectral sequence of modules over 
$H_c^\ast(\SS_2^1,\ZZ_2)$, as in \Cref{rem:HS1action}, by \Cref{lem:gexists} we also have a 
spectral sequence of $H^\ast(G_{24},\ZZ_2)$-modules. 
By periodicity with respect to $k$, we need only check that the extensions
\[
0 \to \ZZ_2 \cong E_\infty^{3,0} \to H_c^3(\SS_2^1,\ZZ_2) \to E_\infty^{1,2} \cong \FF_2 \to 0
\]
and
\[
0 \to \Z/2  \cong E_\infty^{2,2} \to H_c^4(\SS_2^1,\ZZ_2) \to E_\infty^{0,4} \cong \ZZ/8 \to 0
\]
are split. That the first is split follows from 
the second statement of \Cref{whatisd0anyway} with
$M =\ZZ_2$. That the second is split follows from  \Cref{lem:gexists}. 
\end{proof}

We now turn to the calculation of $H_c^\ast (\SS_2^1,\FF_2)$, again using the algebraic
duality  spectral sequence of \Cref{rem:algdualresSS} with $M=\FF_2$.
The result is displayed in the left column of \Cref{fig:ADSS}.

As is \Cref{rem:coho-of-g24}, write $H^\ast(G_{24},\FF_2) \cong
\FF_2[z,k]/(z^2)$ and $H^\ast(C_6,\FF_2) \cong \FF_2[h]$.
The class $k \in H^4(G_{24},\ZZ_2)$ reduces to the class of the
same name in $H^4(G_{24},\FF_2)$ and the class $g \in H^2(C_6,\ZZ_2)$ reduces to
$h^2 \in H^2(C_6,\FF_2)$.

\begin{cor}\label{cor:cohS21at2} The algebraic duality spectral sequence for
$H_c^\ast(\SS_2^1,\FF_2)$ collapses at $E_1$. 
As a module over $\FF_2[k]$, $H_c^\ast(\SS_2^1,\FF_2)$
is freely generated by classes
\[
1,\ z, \ e,\ ze,\ h^i\myoneb,\ h^i\myoneb^2
\]
with $0 \leq i \leq 3$. These classes are of degrees $0$, $3$, $3$, $6$, $1+i$, $2+i$
respectively.  
\end{cor}

\begin{proof} The $E_1$-page of the spectral sequence is displayed in the left column of 
\Cref{fig:ADSS}. That this spectral sequence collapses follows from \Cref{lem:gexists},
the Universal Coefficient Theorem and \Cref{thm:cohS21}. 
\end{proof} 

A key input for our main result on the homotopy type of $L_{K(1)}L_{K(2)}S^0$
is the exact nilpotence order of the class $\myoneb=\myoneb_2 \in H_c^1(\GG_2,\FF_2)$.
The following is a preliminary step. The final result is below in \Cref{prop:b0cubedbig}.

\begin{prop}\label{prop:orderb01} Let $\myoneb \in H_c^1(\SS_2^1,\FF_2)$  be the restriction
of $\myoneb \in H_c^1(\GG_2,\FF_2)$. Then $\myoneb^2  \ne 0$ and $\myoneb^3 =0$ 
in $H_c^\ast(\SS_2^1,\FF_2)$.
\end{prop}

\begin{proof} We already have $\myoneb^2 \ne 0$, by \Cref{cor:cohS21at2}, so
we need to show $\myoneb^3 = 0$.
The homomorphism $\myoneb \colon \SS_2^1\to \FF_2$ is trivial on the central element $-1\in \SS_2$, 
hence it comes from a unique class  $b \in H_c^1(P\SS_2^1,\FF_2)$. This class restricts to zero in 
$H^1(A_4,\FF_2)$. We will show that $b^3=0$.

To see this, recall the decomposition $P\SS_2^1 = K^1 \rtimes A_4$.
\Cref{lem:comtoKat3} implies that the map
\[
H_c^3 (P\SS_2^1,\FF_2) \to H_c^3 (K^1,\FF_2) \times H^3 (A_4,\FF_2)
\]
defined by the two restriction maps  is an isomorphism. Then, from part (3) of
\Cref{lem:reph1}, and the 
fact that $\myoneb$ (and hence $b$) restricts to the class with the same name in $ H_c^1(K^1,\FF_2)$,
we have that $b^3$ restricts to $0$ in $H_c^3(K^1,\FF_2)$.
Since $b$ restricts to $0$ in $H^1(A_4,\FF_2)$, the result follows. 
\end{proof}

\subsection{The exponent of $\myoneb_2$ in the cohomology of $\SS_2$}  We next turn to the analysis of the cohomology
of $\SS_2$ itself. It is natural to ask for a complete calculation of $H_c^\ast(\SS_2,\FF_2)$; however, at this point,
we have no succinct story to tell, and we won't need this calculation in later sections, so we won't follow this idea here.\footnote{See also Theorem 3.4 of \cite{RavCoh} where the author computes an associated graded for $H^*(S(2))$ where $S(2)$ is the Morava stabilizer algebra. This can be related to $H^*_c(S_2, \FF_4)$ after extending coefficients from $\FF_2$ to $\FF_4$.}
Some of the difficulty arises because we do not have an algebraic duality spectral sequence for $\SS_2$.
For this reason, to prove the next result we will examine the Lyndon-Hochschild-Serre Spectral Sequence (LHSSS) for the split
group extension $K \to \SS_2 \to G_{24}$ of \Cref{rem:all-the-splittings}.

\begin{prop}\label{prop:b0cubedbig} In $H_c^\ast(\SS_2,\FF_2)$, the class $\myoneb = \myoneb_2 \in 
H_c^1(\SS_2,\FF_2)$ has nilpotence order exactly 3; that is, $\myoneb^2 \ne 0$ and
$\myoneb^3 = 0$.
\end{prop}

\begin{proof}
We write $H^\ast (G) = H_c^\ast(G,\FF_2)$ in this argument. 
Recall from \Cref{prop:orderb01} that $\myoneb$ maps to the like-named
class $\myoneb \in H^1(\SS_2^1)$ and there we have $\myoneb^2 \ne 0$ and $\myoneb^3 = 0$. 
It follows immediately that  $\myoneb^2 \ne 0$ in $H^2(\SS_2)$. It remains to show
$\myoneb^3=0$ in $H^\ast(\SS_2)$. 

We will use that $\myoneb^3$ restricts to zero in $H^\ast(\SS^1_2)$ and
a comparison of Lyndon-Hochschild-Serre Spectral Sequences induced
by the inclusion $i \colon \SS_2^1 \to \SS_2$:
\begin{equation}\label{eqn:LHSSSb}
\xymatrix{
H^p(G_{24},H^q (K)) \ar[d]_{i^\ast}\ar@{=>}[r] & H^{p+q}(\SS_2) \ar[d]^{i^\ast}\\
H^p(G_{24},H^q (K^1)) \ar@{=>}[r] & H^{p+q}(\SS^1_2)\ .
}
\end{equation}
From \Cref{lem:reph1} we have a decomposition $H^\ast(K) \cong H^\ast(K^1) \otimes E(\myzetaK)$. 
Since $\myzetaK$ lifts to $H^\ast(\GG_2)$,
it is necessarily $G_{24}$-invariant; hence, this is an isomorphism of 
$G_{24}$-algebras and we have an isomorphism 
\[
H^\ast(G_{24},H^\ast(K)) \cong H^\ast(G_{24},H^\ast(K^1)) \otimes E(\myzetaK)\ .
\]
Note that $\myzetaK \in E_2^{0,1} \cong H^0(G_{24},H^1(K))$; therefore, 
we can deduce an isomorphism 
\begin{equation}\label{eq:veryrefinede2}
E_2^{p,q} \cong
H^p(G_{24},H^q(K)) \cong H^p(G_{24},H^q (K^1)) \oplus H^p(G_{24},H^{q-1} (K^1))\myzetaK\ .
\end{equation}
The map $i^\ast$ of $E_2$ terms in 
(\ref{eqn:LHSSSb}) is the algebra map sending $\myzetaK$ to zero. The class $\myoneb$
is detected in $E_2^{0,1}\cong H^0(G_{24},H^1(K))$. 

As with any spectral sequence, the LHSSS gives a filtration
\[
F^{3,0} \subseteq F^{2,1} \subseteq  F^{1,2} \subseteq F^{0,3} = H^3(\SS_2)
\]
with
\[
E_\infty^{p,q} = F^{p,q}/F^{p+1,q-1}
\]
a subquotient of $E_2^{p,q}\cong H^p(G_{24},H^q (K))$.
There is a similar filtration for $H^2(\SS_2)$. 

We will show that $\myoneb^3 \in F^{3,0}$ and hence it is in the image of the
map $H^\ast(G_{24}) \to H^\ast(\SS_2)$ induced by the projection $\SS_2 \to G_{24}$.
Since this map has a section by \Cref{lem:gexists} and  $\chi$ restricts to $0$ in
$H^1(G_{24}) = 0$, we will have $\myoneb^3 = 0$ in $H^\ast(\SS_2)$.

Since $\myonebK^2 = 0$ in $H^2(K)$ we have $\myoneb^2 \in F^{1,1}$. 
Thus $\myoneb^2$ is detected by a permanent cycle 
\[
\inproofchisquare \in H^1(G_{24},H^1 (K))\ .
\]
It follows that the class $\myoneb^3$ is detected at $E_\infty^{1,2}$ by the class 
$\myonebK \inproofchisquare$.

By \Cref{rem:coho-of-g24} we know $H^1(G_{24})=H^2(G_{24})  = 0$.
Then (\ref{eq:veryrefinede2}) forces the map
\[
i^\ast \colon H^p(G_{24},H^1 (K)) \to H^p(G_{24},H^1 (K^1))
\]
at the $E_2^{p,1}$-term to be an isomorphism for $p=1,2$;
in particular, if $\myonebK \inproofchisquare \ne 0$ in $H^1(G_{24},H^2 (K))$
then $i^\ast(\myonebK \inproofchisquare) = \myonebK i^\ast( \inproofchisquare ) \ne 0$ 
in $H^1(G_{24},H^2 (K^1))$. But $\myoneb^3  = 0$ in 
$H^\ast(\SS_2^1)$, so we have $\myonebK \inproofchisquare = 0$ and we can conclude 
$\myoneb^3 \in F^{2,1}$.

Let $\myoneb^3  \in H^\ast(\SS_2)$ be detected
by a permanent cycle $\inproofz$ in $H^2(G_{24},H^1(K))$. Since $\myoneb^3 = 0$ in 
$H^\ast(\SS_2^1)$ and $i^*$ is an isomorphism at $E_2^{2,1}$, we have an equation 
$d_2(\inproofw) = i^\ast( \inproofz )$ in the LHSSS for 
$H^\ast(\SS_2^1)$. Since the map $i^\ast$ of $E_2$-terms in
(\ref{eqn:LHSSSb}) is a surjection, this implies $\inproofz$ itself is in the image of $d_2$.
Thus $\myoneb^3 \in F^{3,0}$ as needed.
\end{proof}

\section{The cohomology of $H_c^\ast(\GG_2,E_0)$}\label{sec:c-v-section}  We now come to one of the key results
of the paper: we show the inclusion of $\WW \to E_0\cong \WW[[u_1]]$ of continuous $G$-modules yields an 
isomorphism
\[
\xymatrix{
H_c^\ast (G,\WW) \ar[r]^-\cong&  H_c^\ast(G,E_0).
}
\]
See \Cref{thm:const} below. This verifies the Chromatic Vanishing Conjecture (see \Cref{conj:chrom-vanish}) in the case
$n=p=2$. It is also a remarkable simplification and at the heart of much
of what we can prove in $K(2)$-local homotopy theory.  At the end of the section we give some related
rational calculations. 

\subsection{Chromatic vanishing at $n=p=2$} We begin with the following  basic case, which gives a
tight control over the algebraic duality  spectral sequence of \Cref{rem:algdualresSS} 
for $M=E_0/2$. As in \Cref{rem:figure1explained} and \Cref{whatisd0anyway} 
we write $\la_0$, $\lb_0$, $\lc_0$, and $\ld_0$ for the
obvious generator of $H^0(F_p,\FF_2)$, for $p=0$ through $3$ respectively. This gives
generators in $H^0(F_p,E_0/2)$ as well.

\begin{thm}\label{thm:consts21} Let $i \colon \FF_4 \to E_0/2 \cong \FF_4[[u_1]]$ be the inclusion of the 
constants. Then the induced  map of algebraic duality  spectral sequences
\begin{equation}\label{eqn:map-of-const-ss}
\xymatrix{
H^q(F_p,\FF_4) \ar[d]_-{i_*} \ar@{=>}[r] & H_c^{p+q}( \mathbb{S}_2^1,\FF_4)\ar[d]^-{i_\ast} \\ 
H^q(F_p,E_0/2) \ar@{=>}[r] & H_c^{p+q}( \mathbb{S}_2^1,E_0/2)} 
\end{equation}
becomes an isomorphism at $E_2$ and yields an isomorphism
\[
i_\ast \colon H_c^*(\mathbb{S}_2^1, \FF_4) \xrightarrow{\cong} H_c^*(\mathbb{S}_2^1,E_0/2).
\]
\end{thm}

\begin{proof} We first show that the map is one-to-one. Recall from \Cref{rem:v1-plus-ss} that
$v_1 = u^{-1}u_1$ for our formal group. This is a $\GG_2$-invariant element modulo $2$; hence the
quotient map
\[
E_\ast/2 \longrightarrow E_\ast/(2,v_1) \cong \FF_4[u^{\pm 1}]
\]
is $\GG_2$-equivariant. The composite map 
\[
H_c^*(\mathbb{S}_2^1,\FF_4) \to  H_c^*(\mathbb{S}_2^1,E_0/2) \to  H_c^*(\mathbb{S}_2^1,E_0/(2,u_1))
\] 
is then an isomorphism because it is induced by the isomorphism
$\FF_4 \cong E_0/(2,u_1)$ on coefficients. For similar reasons the map on
$E_1$-terms of (\ref{eqn:map-of-const-ss})  must also be an injection.

The hard part is to show that the map induced by $i$ is surjective on the $E_2$-term.
Here we need Theorem 1.2.2 of \cite{Paper2}. From there we  read that the $E_2=E_\infty$-term of 
the algebraic duality  spectral sequence for $E_0/2$ is free over $\FF_4[k]$ on elements
\begin{align*}
\la_0 &\in H^0(F_0,E_0/2)\\
\nu^2y \la_{-1} &\in H^3(F_0, E_0/2)\\
h^i\lb_0  &\in H^i(F_1,E_0/2),\qquad 0 \leq i \leq 3,\\
h^i \lc_{0} & \in H^i(F_2,E_0/2),\qquad 0 \leq i \leq 3,\\
\ld_0 &\in H^0(F_3,E_0/2)\\
\nu^2y  \ld_{-1} &\in H^3(F_3,E_0/2)\ .
\end{align*}
(Here, the elements $\la_0, \lb_0, \lc_0, \ld_0$ are the images of the same named generators from \Cref{rem:figure1explained} under the natural maps $H^0(F_p, \FF_2) \to H^0(F_p, E_0/2)$.)
This, the left column of \Cref{fig:ADSS}, and \Cref{cor:cohS21at2} 
imply that on the $E_2$-pages of spectral sequences (\ref{eqn:map-of-const-ss}) 
$i_\ast$ induces an injective homomorphism of 
free graded $\FF_4[k]$-modules with the same number of generators in 
each bidegree. Thus our map must be onto.
\end{proof}
 
We can now give an integral statement. By \Cref{coh-decomp-galois} there is an isomorphism
\begin{align}\label{eq:isoZ2tensorW}
\WW \otimes_{\ZZ_2} H_c^\ast(\SS_2^1,\ZZ_2) \cong H_c^\ast (\SS_2^1,\WW).
\end{align}
Hence \Cref{thm:cohS21} gives an explicit computation of $H_c^\ast (\SS_2^1,\WW)$ using the 
algebraic duality spectral sequence.

\begin{cor}\label{cor:consts21} Let $i\colon \WW \to E_0 \cong \WW[[u_1]]$ be the inclusion of the 
constants. Then $i$ induces an isomorphism
\begin{align*}
H_c^*(\mathbb{S}_2^1, \W) \cong H_c^*(\mathbb{S}_2^1,E_0).\end{align*}
\end{cor}

\begin{proof}
That $i$ induces an isomorphism
\[
H_c^*(\mathbb{S}_2^1, \W/2^n) \cong H_c^*(\mathbb{S}_2^1,E_0/2^n)
\]
follows by induction on $n$. The base case, $n=1$, is \Cref{thm:consts21}.
For the inductive step, we use the five lemma. The integral result is shown by taking inverse limits 
and using the $\lim^1$-exact sequence.
\end{proof}

We record the following companion result to \Cref{cor:consts21}  for use in the proof of \Cref{thm:f4}.
 If $M$ is a profinite $\SS_2^1$-module, let $E^{p,q}_r(M)$ be the $r$th page 
of the algebraic duality spectral sequence of $M$. See \Cref{rem:algdualresSS}.
Recall that $\mytwob \in H_c^2(\GG_2,\ZZ_2)$ is the Bockstein on $\myoneb \in H_c^1(\GG_2,\FF_2)$.

\begin{lem}\label{E2*0-comp} (1) The unit map $i \colon \WW \to E_0$ induces an isomorphism
\[E_2^{*,0}(\WW) \to E_2^{*,0}(E_0) \]
where $E_2^{*,*}(M)$ denotes the $E_2$-term of the algebraic duality spectral sequence for $M$.
\medskip

(2) For $0 \leq p \leq 3$ we have isomorphisms
\[
E_2^{p,0}(E_0) \cong
\begin{cases}
\WW  & p=0 \\
0 & p=1  \\
\FF_4 & p=2 \\ 
\WW & p=3\ .
\end{cases}
\]
Each of the non-zero groups $E_2^{p,0}(E_0)$ is generated by the cohomology class of the unit in 
$E_1^{p,0} = H^0(F_p,E_0)$.
\end{lem}

\begin{proof}
For part (1) we use that $E_2^{*,0}(M)$ is the cohomology of the torsion-free cochain
complexes $E_1^{*,0}=H^0(F_*, M)$. Now, let $M = \WW$ or $E_0$. 
For both choices of $M$ and any $p$, $H^1(F_p, M)=0$. For $M=\WW$ this follows from
\Cref{rem:coho-of-g24} and for $M=E_0$ this follows from \cite{tbauer} and \cite{MR}, 
but see also Section 2 of \cite{BobkovaGoerss}. From this it follows that for all $n>1$
there is a short exact sequences of complexes 
\[
\xymatrix{
0 \ar[r] & H^0(F_{\ast}, M) \ar[r]^-{2^n} & H^0(F_{\ast}, M)  \ar[r] 
& H^0(F_{\ast}, M/2^n)  \ar[r]  & 0.
}
\]
Thus we deduce that we have a short exact sequence of chain complexes
\[
\xymatrix{
0 \ar[r] & H^0(F_{\ast}, M/2^n) \ar[r]^-{2} & H^0(F_{\ast}, M/2^{n+1})  \ar[r] 
& H^0(F_{\ast}, M/2)  \ar[r]  & 0.
}
\]
The map of cochain complexes
\[
H^0(F_{\ast}, \WW/2) \longrightarrow H^0(F_{\ast}, E_0/2)
\]
induces an isomorphism in cohomology by \Cref{cor:consts21}. 
Then, using the five lemma, we have that 
\[
H^0(F_{\ast}, \WW/2^n) \longrightarrow H^0(F_{\ast}, E_0/2^n)
\]
is an isomorphism. The integral result then follows by taking inverse limits and using 
the $\lim^1$-exact sequence.

Part (2) is proved by combining \eqref{eq:isoZ2tensorW}, \Cref{whatisd0anyway}, and \Cref{thm:cohS21}.
\end{proof}

We can now extend \Cref{cor:consts21} to a larger class of groups which includes $
\SS_2$, $\GG_2$, and $\GG_2^1$.

\begin{thm}\label{thm:const} Let $G \subseteq \GG_2$ be any closed subgroup
which contains $\SS_2^1$ as a normal subgroup. Then
the inclusion of $\Z_2[[G]]$-modules
$i \colon \W \rightarrow E_0$ induces an 
isomorphism
\[
i_\ast: H_c^*(G, \W) \cong H_c^*(G, E_0).
\]
\end{thm}

\begin{proof} This follows from \Cref{cor:consts21} and
the following diagram of Lyndon-Hochschild-Serre Spectral Sequences
\begin{equation*}
\xymatrix@R=0.2in{
H_c^p(G/\SS_2^1,H_c^q(\SS_2^1,\WW)) \ar[d]_\cong^{i_*} \ar@{=>}[r] &
H_c^{p+q}(G,\WW)\ar[d] \\ 
H_c^p(G/\SS_2^1,H_c^q(\SS_2^1,E_0)) \ar@{=>}[r] & H_c^{p+q}( G,E_0)} 
\end{equation*}
\end{proof}

\begin{rem}\label{rem:other-const} \Cref{thm:const} extends to an isomorphism
\[
i_\ast: H_c^*(G, \W/2^n) \cong H_c^*(G, E_0/2^n)
\]
for all $n \geq 1$.
\end{rem} 

\begin{rem}\label{rem:z-vs-w} One point that might be worth clarifying is the effect of the
extension of scalars $\ZZ_2 \to \WW$. Using \Cref{coh-decomp-galois} and
\Cref{thm:const} we can conclude that
\[
H_c^\ast(\SS_2,\ZZ_2) \cong H_c^\ast(\GG_2,E_0).
\]
Similar remarks apply to the inclusion $\SS_2^1 \subseteq \GG_2^1$. 
\end{rem}

\begin{rem}\label{we-need-this-in-sec8} \Cref{lem:hone-of-stwo} and \Cref{thm:const} imply
\[
H_c^1(\GG_2,E_0) \cong \ZZ_2
\]
generated by $\zeta_2$ and \Cref{lem:hone-of-stwo} and
\Cref{rem:other-const} imply
\[
H_c^1(\GG_2,E_0/2) \cong \ZZ/2 \times \ZZ/2
\]
generated by $\zeta_2$ and $\myoneb_2$.
\end{rem}

\subsection{The rational cohomology of $\SS_2$ and $\GG_2$} In this subsection we make
the rational calculations needed to deduce the homotopy type of $L_1L_{K(2)}S^0$ from
the homotopy type of $L_{K(1)}L_{K(2)}S^0$.
The main result is \Cref{prop:ratcohE}, which completely
calculates $H_c^*(\mathbb{G}_2,E_*)\otimes \QQ$.
\medskip

By \Cref{whatisd0anyway}, there exists a class
\[
e \in H_c^3(\SS_2^1,\ZZ_2)
\]
which maps to a generator of $H_c^3(K^1,\ZZ_2)\cong \ZZ_2$.
Let $E_{\QQ_2}(-)$ denote exterior algebras over $\QQ_2 \cong \ZZ_2 \otimes \QQ$.

\begin{prop}\label{prop:ratcoh} (1) There is an isomorphism of graded rings
\[
E_{\QQ_2}(e) \cong H_c^\ast(\SS_2^1,\ZZ_2) \otimes \QQ.
\]

(2) There is an isomorphism of graded rings
\[
E_{\QQ_2}(\myzeta, e) \cong H_c^\ast(\SS_2,\ZZ_2) \otimes \QQ.
\]
\end{prop}

\begin{proof} The first part follows from \Cref{thm:cohS21}. For the second
part consider the diagram of Lyndon-Hochschild-Serre Spectral Sequences
induced by the inclusion $K \subseteq \SS_2$:
\[
\xymatrix{
H_c^p(\ZZ_2,H_c^q(\SS_2^1,\ZZ_2)) \ar[r]\ar[d]& H_c^{p+q}(\SS_2,\ZZ_2)\ar[d]\\
H_c^p(\ZZ_2,H_c^q(K^1,\ZZ_2)) \ar[r]& H_c^{p+q}(K,\ZZ_2).
}
\]
Since $K \subseteq \SS_2$ and $K^1 \subseteq \SS_2^1$ are finite index
subgroups, the vertical maps are rational monomorphisms.
By part (1) of  \Cref{lem:reph1}, we have $H_c^4(K,\ZZ_2) \otimes \QQ \cong \QQ_2$.
Thus the action of $\ZZ_2$ on $H_c^3(K^1,\ZZ_2) \otimes \QQ \cong \QQ_2$ must be trivial,
as needed.
\end{proof}

This extends immediately to a much larger calculation. 

\begin{thm}\label{prop:ratcohE} Let $i\colon \WW \to E_0 \cong \WW[[u_1]]$ be the inclusion of the 
constants.This map induces an isomorphism
\[
E_{\QQ_2}(\zeta,e) \cong H_c^\ast(\GG_2,\WW)\otimes \QQ \cong
H_c^*(\mathbb{G}_2,E_*)\otimes \QQ.
\]
\end{thm}

\begin{proof} From \Cref{thm:const}, \Cref{rem:z-vs-w},
and \Cref{prop:ratcoh} we get isomorphisms
\[
E_{\QQ_2}(\zeta,e)\cong H_c^\ast(\GG_2,\WW) \otimes \QQ \cong H_c^\ast (\mathbb{G}_2,E_0)\otimes \QQ.
\]

To complete the proof, we must show $H_c^\ast(\mathbb{G}_{2},E_t)$ is torsion for $t \ne 0$. 
This  is very standard, but we give the proof as it is short.

Define $\Z_2^\times \to \SS_2 \subseteq \G_2$ by sending $\ell$
to the $\ell$-series $[\ell](x)$ of our chosen formal group law. This identifies $\Z_2^\times$ with the 
center of $\GG_2$. Let $\Z_2 \subset 
\Z_2^\times$ be the subgroup topologically generated by
an integer $\ell$ with $\ell \equiv 1$ modulo $4$. Now consider
the Lyndon-Hochschild-Serre Spectral Sequence for the extension
\[
0 \to \Z_2 \to \G_2 \to \G_2/\Z_2 \to 0.
\]
If $t \ne 0$, then $\ell$ acts on $E_{2t}$ by multiplication by $\ell^{-t}$, hence
$H_c^0(\Z_2,E_{2t}) = 0$ and $H_c^1(\Z_2,E_{2t})$ is torsion. Since
$E_{2t+1} = 0$, the result follows.
\end{proof} 

\begin{rem}\label{rem:eis4dzero} By \Cref{prop:ratcoh} there must be a class in
$\pi_{-3}L_{K(2)}S^0$ which maps to a multiple of $e \in \QQ \otimes \pi_{-3}L_{K(2)}S^0$. 
We can be very specific about this. We will see below in \Cref{rem:backfill-on-pi-3} that
\[
\pi_{-3}L_{K(2)}S^0 \cong \ZZ_2 \oplus \ZZ/2 \subseteq H_c^3(\GG_2,E_0)
\]
generated by the classes $4e$ and $\zeta_2\mytwob$.

A similar phenomenon happened at $p=3$ and $n=2$; there was a class $e$ so that 
$3e$ detects a copy of $\Z_3$. See \cite{GoerssSplit}. 

\end{rem}


\section{The class $\myoneb_2$ is a $d_3$-cycle}\label{sec:d3cycle}
\setcounter{subsection}{1}

As should be clear by now, the class $\myoneb=\myoneb_2 \in H_c^1(\GG_2,E_0/2)$ is one key to the 
extra subtleties we encounter at $p=2$ in $K(2)$-local homotopy theory. Over the next few sections we
will work on some of the specific implications of the existence of this class. We show that
$\myoneb$ is a permanent cycle in the $K(2)$-local Adams-Novikov Spectral
Sequence in \Cref{thm:bzeroinvo}; see also \Cref{prop:chiperm} for a $K(1)$-local application.
However, it turns out that evaluating $d_3$ requires an entirely different set of techniques,
which we isolate in this section. The central idea of this section is due to Mike Hopkins.

The goal is to prove the following result. 

\begin{thm}\label{prop:b0d3cycle}
In the spectral sequence
\[
E_2^{s,t} = H_c^s( \G_2, E_tV(0)) \Longrightarrow \pi_{t-s}L_{K(2)}V(0)
\]
the classes $\myoneb$ and $\myoneb^2$ are $d_3$-cycles.
\end{thm}

Before proving this, we give some background to explain our methods.

We proved above in \Cref{thm:const} and \Cref{rem:z-vs-w}
that the inclusion of the constant power series $\ZZ_2 \to E_0$  
induced an isomorphism 
\[
H_c^\ast(\SS_2,\ZZ_2) \cong H_c^\ast(\GG_2,\WW) \cong H_c^\ast (\GG_2,E_0).
\]
We would like to exploit a homotopy fixed point 
spectral sequence for the trivial action of $\SS_2$ on the $2$-completed sphere $S^0_2$; 
the issue is that it would take considerable
effort to define such a spectral sequence. Fortunately, the class $\myoneb^i \in H_c^\ast(\SS_2,\FF_2)$ is
the restriction of the non-zero class $h^i \in H^\ast(C_2,\FF_2)$ under a quotient map 
$\myoneb \colon \GG_2 \to C_2$. See \Cref{rem:hone}. We can then examine the map on cohomology
\[
H^\ast(C_2,\pi_\ast S^0_2) \to H_c^\ast(\GG_2,\pi_\ast S^0_2) \to H_c^\ast(\GG_2,E_\ast).
\]
We will show below in \Cref{lem:sssetup} that this composite map does extend 
to a map of homotopy fixed point spectral sequences. 

For the trivial action of $C_2$ on the $2$-completed sphere spectrum $S^0_2$,  we have
\[
(S^0_2)^{hC_2} \simeq F({BC_2}_+, S^0_2) \simeq F(\R P^{\infty}_+,S^0_2).
\]
Here $F(X,Y)$ is the function spectrum. The homotopy fixed point spectral sequence
\[
H^s(C_2,\pi_tS^0_2) \Longrightarrow \pi_{t-s}F({BC_2}_+, S^0_2)
\]
is the Atiyah-Hirzebruch Spectral Sequence for the cohomotopy of ${\R P^{\infty}}$.
It follows from  Lin's Theorem \cite{LinsTheorem} 
that ${(S^0_2)}^{hC_2} \simeq (S^0 \vee \R P^{\infty}_+)^{\wedge}_2$;
hence, the homotopy fixed point spectral sequence is an impractical approach to Lin's theorem.  
Low-dimensional calculations can still be informative, however, and we use them to give
us the information we need.

We now set up the map of spectral sequences we will use.
We recall some results from Devinatz \cite{DevTrans}, especially sections 2 and 3. This paper 
by Devinatz undertakes a thorough analysis iterating the homotopy fixed point constructions.
Let $H$ and $K$ be closed subgroups of $\GG_n$ and suppose $H$ is normal in $K$. 
This implies that there is a
$K/H$-equivariant map $E^{hK} \to E^{hH}$ of $E_\infty$-ring spectra. Let $X$ be an
$E^{hK}$-module. Then, in the $K(n)$-local category of $E^{hK}$-modules, we can form an $E^{hH}$-based
Adams-Novikov resolution of $E^{hK}$ and apply the mapping space functor 
$F_{E^{hK}}(X,-)$ to obtain a spectral sequence
\begin{equation*}\label{eq:homotopy-LSHSS}
H_c^s(K/H,\pi_tF_{E^{hK}}(X,E^{hH})) \Longrightarrow \pi_{t-s}F_{E^{hK}}(X,E^{hK}).
\end{equation*}
The subtle part is to identify the $E_2$-term. This is Theorem 3.1 of \cite{DevTrans}. 
If $Y$ is a finite CW-spectrum and $DY$ its Spanier-Whitehead dual,
we can set $X = E^{hK} \wedge DY$  to obtain a spectral sequence
\begin{equation}\label{eq:dev-desc1}
H_c^s(K/H,\pi_t(E\wedge Y)^{hH}) \Longrightarrow \pi_{t-s}(E\wedge Y)^{hK}.
\end{equation}
Furthermore, in the Appendix of \cite{DevTrans}, Devinatz shows that if $K/H$ is finite, then this is
the usual homotopy fixed point spectral sequence. 

This is natural in the pair $(K,H)$ in the sense that if $K_0 \subseteq K$ and 
$H_0 \subseteq H$ then we get a diagram of spectral sequences
\begin{equation}\label{eq:dev-desc2bis}
\xymatrix{
H_c^s(K/H,\pi_t(E\wedge Y)^{hH}) \ar[d] \ar@{=>}[r] &\pi_{t-s}(E\wedge Y)^{hK} \ar[d]\\
H_c^s(K_0/H_0,\pi_t(E\wedge Y)^{hH_0}) \ar@{=>}[r] &\pi_{t-s}(E\wedge Y)^{hK_0}.
} 
\end{equation}

Now let $H_0 = \{e\}$, $K_0 = K = \GG_2$, and $H$ the kernel of the map 
\[
\myoneb \colon \G_2\to (\Z/4)^{\times}\cong C_2
\]
of \Cref{rem:hone}. Then $K/H \cong C_2$ and $K_0/H_0 = \GG_2$.
If we set $Y = V(0)$, then \eqref{eq:dev-desc2bis}
gives a diagram of spectral sequences
\begin{equation}\label{eq:dev-desc2}
\xymatrix{
H^s(C_2,\pi_t(E\wedge V(0))^{hH}) \ar[d] \ar@{=>}[r] &\pi_{t-s}L_{K(2)}V(0)\ar[d]\\
H_c^s(\GG_2,E_tV(0)) \ar@{=>}[r] &\pi_{t-s}L_{K(2)}V(0).
}
\end{equation}

Now consider that map
\[
V(0)  \simeq  S^0_2 \wedge V(0) \to E^{hH} \wedge V(0) \simeq (E \wedge V(0))^{hH}
\]
of $C_2$-equivariant spectra, with trivial action on $V(0)$. Since the top row of (\ref{eq:dev-desc2}) 
is the usual homotopy spectral sequence we get a map of spectral sequences
\begin{equation}\label{eq:dev-desc3}
\xymatrix{
H^s(C_2,\pi_tV(0)) \ar[d] \ar@{=>}[r] &\pi_{t-s}V(0)^{hC_2}\ar[d]\\
H^s(C_2,\pi_t(E\wedge V(0))^{hH}) \ar@{=>}[r] &\pi_{t-s}L_{K(2)}V(0). \\
}
\end{equation}
The  top spectral sequence here is in the usual (unlocalized) stable category.
Many variants are possible here; for example, we could use the sphere as well.

We now have the following result. 
\begin{lem}\label{lem:sssetup} There is a map of spectral sequences 
\begin{equation}\label{segalss}
\xymatrix{
H^s(C_2,\pi_tV(0))  \ar[d]_\varphi  \ar@{=>}[r]  &  \pi_{t-s} V(0)^{hC_2}\ar[d]\\
H_c^s(\G_2,E_tV(0))   \ar@{=>}[r] & \pi_{t-s} L_{K(2)}V(0).
}
\end{equation}
from the homotopy fixed point spectral sequence for $V(0)$ to the Adams-Novikov Spectral Sequence 
of $L_{K(2)}V(0)$. The map $\varphi$ at the $E_2$-page is the map on cohomology induced by the 
quotient map $\myoneb \colon \G_2\to (\Z/4)^{\times}\cong C_2$ and the Hurewicz map
$\pi_\ast V(0) \to E_\ast V(0)$. 
\end{lem}

\begin {proof} Combine the maps of spectral sequences (\ref{eq:dev-desc2}) with (\ref{eq:dev-desc3}). 
Note the top spectral sequence of (\ref{eq:dev-desc3}) is the homotopy fixed point spectral 
sequence for $V(0)$ with a trivial $C_2$ action. 
\end{proof} 

Since $E \wedge V(0)$ doesn't have a ring structure, neither of the spectral sequences
of (\ref{segalss}) is a spectral sequence of rings. Nonetheless,
the $E_2$-terms are graded commutative rings and the map $\varphi$ is a ring map. 
Let $h$ be the non-zero element in $H^1(C_2, \pi_0V(0)) \cong \F_2$ and write $h^i$ for its
powers. Then, by definition, $\varphi(h^i)= \myoneb^i$. 

Finally, since the homotopy fixed point spectral sequence for $V(0)$ is the 
Atiyah-Hirzebruch Spectral Sequence for $V(0)^\ast (\R P^\infty_{+})$ we have the following observation.  
The inclusion of the skeleton $\R P^n \to \R P^{\infty}$ induces a map of spectral  sequences 
\[
\xymatrix{
H^s(C_2,\pi_tV(0))  \ar[d]^-{Skel_n}  \ar@{=>}[r]  &    \pi_{t-s} V(0)^{hC_2} \ar[d] \\
H^s(\R P^n_{+}, \pi_{t}V(0))   \ar@{=>}[r]
& \pi_{t-s}F(\R P^n_{+}, V(0))\  . 
}
\]
where the lower spectral sequence is the Atiyah-Hirzebruch Spectral Sequence for 
$V(0)^\ast ( \R P^n_{+})$.
 
\begin{proof}[Proof of \Cref{prop:b0d3cycle}.] To prove that $\myoneb$ is a $d_3$-cycle, we prove that
$d_2(h) =0$ in the fixed point spectral sequence and that the only
target for $d_3(h)$  maps to zero under the map $\varphi$ of (\ref{segalss}). We will
use our result \Cref{prop:b0cubedbig} on the nilpotence order of $\myoneb$.

We begin by analyzing the $d_2$-differential. 
First, note that 
\[Skel_3 \colon H^s(C_2,\pi_tV(0)) \to H^s(\R P^3_{+}, \pi_{t}V(0))   \]
is an isomorphism if $0\leq s \leq 2$ and injective if $s=3$.
Hence, to prove that $h$ is a $d_2$-cycle, it suffices to prove that this holds for $Skel_3(h)$.

First, note that $\R P^{3}_+\simeq S^0\vee \Sigma V(0)\vee S^3$.
Hence,
\[
D\R P^3_+ \smsh V(0)\simeq V(0)\vee \Sigma^{-3}V(0)\vee (\Sigma^{-2}V(0)\wedge V(0)).
\]
Therefore, 
\begin{equation*}
\pi_{-1}(D\R P^3_+ \smsh V(0)) \cong \Z/4 \oplus \Z/4 .
\end{equation*}
The needed calculations are elementary, but see also \Cref{sec:remonV0}.

We now turn to the analysis of the relevant Atiyah-Hirzebruch spectral sequences. In the spectral sequence 
\[
E_2^{s,t} = H^s(\R P^3_{+}, \pi_{t}V(0)) \Longrightarrow \pi_{t-s} (D\R P^3_+ \smsh V(0)) ,
\]
the only non-zero contributions to $\pi_{-1} (D\R P^3_+ \smsh V(0))$ in $E_{2}^{s,t}$ are:
\[
E_{2}^{s,t} = \begin{cases}\Z/4\{v_1 h^3\} & (s,t) = (3,2) \\
 \Z/2\{ \eta h^2 \} &  (s,t) = (2,1) \\
\Z/2\{h\}  &  (s,t) = (1,0) .
\end{cases}
\]

Thus $d_2(h)$ must be zero, or the $E_\infty$-term would be too small to account  
for the size of $\pi_{-1}(D\R P^3_+ \smsh V(0))$.

We turn to the $d_3$-differential. For this we do not need to restrict to 
any skeleton, but can directly  use the diagram of spectral sequences \eqref{segalss}. 
The targets for $d_3(h)$ are in 
\[
E_2^{4,2} = H^4(C_2,\pi_2V(0))\cong \Z/2\{v_1h^4\}.
\] 
If $d_3(h)=\lambda v_1h^4$ for $\lambda\in \Z/2$, 
then, by naturality, \[d_3(\myoneb)=\lambda v_1\varphi(h^4)=\lambda v_1\myoneb^4.\] 
However, by \Cref{prop:b0cubedbig}, $\myoneb^4 = 0$. This finishes the proof that
$\myoneb$ is a  $d_3$-cycle. By the Geometric Boundary Theorem (see \Cref{rem:whats-up-GBT})
and the fact that  $\myoneb^2$ is the Bockstein of $\myoneb$, we have that
$\myoneb^2$ is also a $d_3$-cycle.
\end{proof}

\section{The decomposition of $L_{K(1)}L_{K(2)}S^0$}\label{sec:k1k2s}

In this section, we prove one of our main results; that is, we show that
\begin{align*}
L_{K(1)}L_{K(2)}S^0 \simeq L_{K(1)}\big(S^0 \vee S^{-1}\vee\Sigma^{-2}V(0)\vee\Sigma^{-3}V(0)\big)\  .
\end{align*}
See  \Cref{thm:the-big-one} below. In \Cref{sec:getting-to-l1} we will use this result to calculate
$L_1L_{K(2)}S^0$.

The subtle point to the argument is to make an analysis of the action of the element
$\zeta_1 \in \pi_\ast L_{K(1)}S^0$ on $\pi_\ast L_{K(1)}L_{K(2)}S^0$. In \Cref{prop:homotopy-v0-k1}
we saw that $v_1^4\zeta_1 = \sigma \in \pi_\ast L_{K(1)}V(0)$, where $\sigma \in \pi_7S^0$ 
is the Hopf map. We gave some further background on the role of $\sigma$ 
in $K(n)$-local homotopy theory in \Cref{findsigma}, but now we get specific and analyze
the action of $\sigma$ in $\pi_\ast L_{K(2)}S^0$.

\subsection{A Bockstein Lemma and detecting products with $\sigma$}\label{sec:bock}
Fix a prime $p$ and let $C = (C^\bullet,\partial_{C})$ be a torsion-free cochain complex
over $\ZZ_{(p)}$. 
We assume all cochain complexes are zero in negative degrees. Write
\[
\bockn=\bockn_C \colon H^{s}(C/p^n) \longrightarrow H^{s+1}(C)
\]
for the connecting homomorphism in the long exact sequence in cohomology induced by the short exact
sequence of cochain complexes
\[
\xymatrix@C=25pt{
0 \ar[r] &C \ar[r]^-{p^n} & C \ar[r] & C/p^n \ar[r] & 0.
}
\]
This is, of course, the $n$th Bockstein homomorphism. We may write $\bock{n}_C$
if we need to emphasize the $n$ and we may abbreviate the Bockstein to $\bockn_G$ if $C$
is a complex for computing the group cohomology of a group $G$ with
coefficients in some module.

In working with the Bockstein we write $\overline{a}$ for the image of $a \in C$ under the reduction
to $C/p^n$ and likewise we write $\overline{y}$ for the image of $y\in H^s(C)$ under the reduction 
to $H^s(C/p^n)$.
By construction, if $x \in H^s(C/p^n)$ is any cohomology class, then 
$\bockn_C(x)$ is represented by $\partial_{C}(a)/p^n$, 
where $a \in C^s$ is such that $\overline{a}$ is a representative for $x$. 
The next result then follows from the Leibniz rule.

\begin{lem}\label{lem:bockstein1} Let $A$ be a torsion-free differential graded algebra over $\ZZ_{(p)}$
and let $a \in H^\ast (A/p^n)$ be a cohomology class. Let  $M$ be a torsion-free differential 
graded module over $A$ and $y \in H^{\ast}(M)$. Then
\[
\bockn_M(a\overline{y}) = \bockn_A(a) y.
\]
\end{lem}

Let $A = (A^\bullet,\partial_A)$ be a torsion-free differential graded algebra and 
\begin{equation}\label{test-case-1}
\xymatrix{
M_0 \ar[r]^{d_1} & M_1 \ar[r]^{d_1} & \cdots \ar[r]^{d_1} & M_{n-1} \ar[r]^{d_1} & M_n \ar[r] & \ldots
}
\end{equation}
a cochain complex of differential graded $A$-modules. This is a bicomplex, but the internal differentials 
$\partial_{M}\colon M_p^q \to M_p^{q+1}$ and the external differentials 
$d_1$ have different behavior with respect to the $A$-module structure:
\begin{equation}\label{eqn:diffstwotypes}
\partial_{M}(ax) = \partial_A(a)x + (-1)^{\deg(a)}a\partial_{M}(x)\qquad\mathrm{and}\qquad 
d_1(ax) = ad_1(x).
\end{equation}
Nonetheless, we get a total complex $T(M)$ and, by filtering by degree in $p$, a spectral sequence 
\[
E_1^{p,q} = H^qM_{p}\Longrightarrow H^{p+q}T(M)
\]
with $d_1 = H^\ast(d_1)$ and $d_r \colon E_r^{p,q} \to E_r^{p+r,q-r+1}$. It follows from 
(\ref{eqn:diffstwotypes}) that this is a spectral
sequence of $H^\ast(A)$ modules. Finally, there is an edge homomorphism
\[
E_2^{p,0} = H^p(H^0(M_\bullet),d_1) \longrightarrow H^p(T(M)).
\]
The following result has generalizations, but will suffice for our purposes. The proof is a diagram chase.

\begin{lem}\label{lem:bocksteininSS} Let $x \in H^p (T(M))$ be the image of a $d_1$-cocycle 
$y \in H^0(M_p)$ under the edge homomorphism.  Let $a \in H^0(A/p^n)$ be a class 
so that there is an element $c\in A$ satisfying

\begin{enumerate}

\item the class $\overline{c}$ is a cocycle representing $a$, and 

\item the class $cy$ is a cocycle for the coboundary operator $\partial \colon M_p^0 \to M_p^1$.
\end{enumerate}
Then $\bock{n}_{T(M)}(a\overline{x}) \in H^{p+1} (T(M))$ is the image under the edge homomorphism of 
\[
\Big[\frac{d_1(cy)}{p^n}\Big] \in H^0(M_{p+1})
\] 
where $[z]$ denotes the cohomology class of $z$.  
\end{lem}

We apply \Cref{lem:bocksteininSS} to the algebraic duality spectral sequence
of  \Cref{sec.dualityres}.
Applying the functor $\Hom_{cts}(-,E_*)$ to the algebraic duality resolution of \Cref
{thm:bob} gives us an exact sequence of twisted $E_*$-$\SS_2^1$-modules
\begin{align*}
0 \to E_\ast \to & \Hom_{cts}(\Z_2[[\SS_2^1/F_0]],E_\ast) \to \Hom_{cts}(\Z_2[[\SS_2^1/F_1]],E_\ast)\\
&\to\Hom_{cts}(\Z_2[[\SS_2^1/F_2]],E_\ast) \to \Hom_{cts}(\Z_2[[\SS_2^1/F_3]],E_\ast) \to 0.
\end{align*}
We now let $A =C^{\bullet}(\SS_2^1,E_\ast)$ be the differential graded algebra of continuous cochains 
with
\[
C^n(\SS_2^1,E_*)=\map_{cts}(\SS_2^1\times\ldots\times \SS_2^1,E_*) 
\cong \Hom_{cts}(\Z_2[[\SS_2^1\times\ldots \times\SS_2^1]],E_*) 
\]  
with $n$ factors of $\SS_2^1$. Furthermore 
let $M_p$ for $0\leq p\leq 3$ be the differential graded 
$A$ module with 
\[
M_p^n=\Hom_{cts}(\Z_2[[\SS_2^1\times\ldots\times  \SS_2^1]]\otimes \ZZ_p[[\SS_2^1/F_p]],E_*)\  , 
\]
again with $n$ factors of $\SS_2^1$. The tensor product has to be taken 
in the appropriate category of 
profinite $\ZZ_2[[\SS_2^1]]$-modules, so in fact it is a completed tensor product. 
Then we get a complex of differential graded $A$-modules 
\[
0\to M_0\to M_1\to M_2\to M_3\to 0
\]
to which we will apply \Cref{lem:bocksteininSS}. Note that the cohomology 
of the differential graded algebra $A$ is equal to $H_c^*(\SS_2^1,E_*)$ while the cohomology 
of $M_p$ is given by $\Ext_{\ZZ_2[[\SS_2^1]]}^*(\ZZ_2[[\SS_2^1/F_p]],E_*)\cong H^*(F_p,E_*)$.
The spectral sequence of the total complex $T(M)$ is isomorphic to 
the algebraic duality spectral sequence of  \Cref{rem:algdualresSS} with $E_1$-term given as 
\[
E_1^{p,q}= H^q(F_p,E_*)\Longrightarrow H_c^{p+q}(\mathbb{S}_2^1,E_*). 
\]

We now get specific about the class $c$ needed for \Cref{lem:bocksteininSS}. Recall from
\Cref{rem:v1-plus-ss} that $v_1 = u_1u^{-1} \in H_c^0(\GG_2,E_2/2)$.  We record a list of facts that are used in
applications of \Cref{lem:bocksteininSS}.

\begin{lem}\label{facts} Let 
\begin{equation*}\label{c4defined}
c_4 = 9 (u_1^4u^{-4}+8 u_1u^{-4}) \in E_8.
\end{equation*} 

\begin{enumerate}

\item  The class $c_4$ is invariant for the action of $G_{24}$ and defines a class
\[
c_4 \in H^0(G_{24},E_8).
\]

\item The reduction $\overline{c}_4 \in E_8/16$ of $c_4$ is invariant with respect to the action of $\GG_2$
and defines a class
\[
\overline{c}_4 \in H_c^0(\GG_2,E_8/16)
\]
which reduces to $v_1^4$ modulo $2$. Furthermore, $H_c^0(\GG_2,E_8/2) \cong \ZZ/2$ generated
by $v_1^4$.
\end{enumerate}
\end{lem}

\begin{proof} This all follows from Theorem 1.2.2 and Lemma 5.2.2 of \cite{Paper2}. 
\end{proof}

We use \Cref{facts} to deduce the following result. Let $\bock{4} \colon H_c^0(H,E_8/16) \to H_c^1(H,E_8)$ denote
the Bockstein  homomorphism.

\begin{prop}\label{prop:c4-and-sigma} \label{rem:sigma-non-zero}  (1) 
Let $H \subseteq \GG_2$ be any closed subgroup. Up to multiplication by a unit in $\ZZ_2$, the image of $\sigma$ in
$\pi_\ast E^{hH}$ is detected in the spectral sequence
\begin{equation*}\label{eq:how-many-more-times}
H_c^s(H,E_t) \Longrightarrow \pi_{t-s}E^{hH}
\end{equation*}
by the class $\bock{4} (\overline{c}_4) \in H_c^1(H, E_8)$.

(2) If $H \subseteq G_{24}$ then 
\[
\bock{4} (\overline{c}_4) =0 \in H_c^1(H,E_8).
\]
\end{prop}

\begin{proof} For part (1), first let $H=\GG_2$ and $\overline{c}_4 \in H^0(\GG_2, E_8/16)$.
Part (2) of \Cref{facts} implies $\overline{c}_4$ satisfies the assumptions of \Cref{prop:finding-sigma} in this case.
Therefore, by that proposition, $\bock{4} (\overline{c}_4) \in H_c^1(\GG_2, E_8)$ detects $\sigma$.
For general $H$, the result follows by restriction in group cohomology. 

For part (2), it is sufficient to do the case $H=G_{24}$. Part (1) of \Cref{facts} says that $\overline{c}_4 \in H^0(G_{24}, E_8/16)$
is the reduction of $c_4 \in H^0(G_{24}, E_8)$, so  $\bock{4} (\overline{c}_4)=0 \in  H^1(G_{24}, E_8)$. 
\end{proof}

If $H \subseteq \GG_2$ is a closed subgroup which contains $\SS_2^1$, then $\sigma \ne 0 \in \pi_\ast E^{hH}$.
This will follow from \Cref{thm:what-we-hope-to-prove} below. If $H \subseteq G_{24}$, then
$\sigma = 0 \in \pi_\ast E^{hH}$.\footnote{To see $\sigma = 0$ in $\pi_\ast E^{hG_{24}}$, use 
Section 8.2 of \cite{tbauer} to see $\sigma = 0$ in $\pi_\ast\mathbf{tmf}$ and the existence of a map $\mathbf{tmf} \to E^{hG_{24}}$.} 
For now we have the following detection results. 

\begin{lem}\label{lem:bocksteininSSc4}  
Let $0\leq p<3$. Let $x \in H_c^{p} (\SS_2^1,E_*)$ be the image
of a $d_1$-cocycle $y \in H^0(F_{p},E_*)$ under the edge homomorphism
\[
H^{p}(H^0(F_\bullet,E_\ast),d_1) \to H_c^{p}(\SS_2^1,E_\ast).
\]
Then $\bock{4}(c_4\overline{x}) \in H_c^{p+1} (\SS_2^1,E_\ast)$ is the image 
under the edge homomorphism of 
\[
\Big[\frac{d_1(c_4y)}{16}\Big] \in H^{p+1}(H^0(F_\bullet,E_\ast),d_1)\ .
\]
\end{lem}

\begin{proof} This is an immediate rephrasing of \Cref{lem:bocksteininSS} in this context.
\end{proof} 

\Cref{lem:bockstein1}, \Cref{prop:c4-and-sigma} and \Cref{lem:bocksteininSSc4} 
can be combined to prove the following result. 

\begin{prop}\label{prop:siggam} \label{lem:finding-sigma-in-s21} 
Let $0\leq s <3$.
Let $x \in \pi_{t-s} E^{h\SS_2^1}$ be detected in the spectral sequence
\[
H_c^s(\SS_2^1,E_t) \Longrightarrow \pi_{t-s}E^{h\SS_2^1}
\]
by the image of a class $y \in H^s(H^0(F_\bullet,E_t))$ under the edge homomorphism 
\[
H^s(H^0(F_\bullet,E_t)) \longrightarrow H_c^s(\SS_2^1,E_t)
\]
of the algebraic duality spectral sequence.
Then $\sigma x$ is detected in $H_c^{s+1}(\SS_2^1,E_{t+8})$ by the image of the class
\[
\Big[\frac{d_1(c_4y)}{16}\Big] \in H^{s+1}(H^0(F_\bullet,E_{t+8}))\ 
\]
under the edge homomorphism.
\end{prop}

\subsection{The group cohomology of $v_1^{-1}E_\ast V(0)$}\label{sec:local-coh}
In the next few results, we identify products with $\sigma$ in the algebraic duality
spectral sequence and then use this to give a description of $v_1^{-1}H_c^\ast(\GG_2,E_\ast V(0))$. See
\Cref{thm:cohG2V0} below. We begin with a connection to homotopy theory. 

If $Z$ is a type $1$ complex with a $v_1$-self map
$f \colon \Sigma^d Z \to Z$ and if $X$ is any spectrum,
let
\[
v_1^{-1}(X \wedge Z) = (1 \wedge f)^{-1}(X \wedge Z).
\]
This is independent of the choice of $f$ and functorial by the essential uniqueness of $v_n$-self maps.
See \cite{NilpotenceII}. For the same reason, this construction preserves cofiber sequences.

The Telescope Conjecture $n=1$ and $p=2$ gives a weak equivalence
\[
v_1^{-1}(X \wedge Z) \simeq L_{K(1)}(X \wedge Z)
\]
and, more generally, a weak equivalence
\[
L_{K(1)}X \simeq \holim v_1^{-1} (X \wedge S/2^n)
\]
where $S/2^n$ is the mod $p^n$-Moore spectrum. See Section 2 of \cite{MahImJ}. 

From this we have that for any closed subgroup $G \subseteq \GG_2$ there
is a localized spectral sequence 
\[
v_1^{-1}H_c^\ast(G,E_\ast V(0)) \Longrightarrow
 \pi_\ast L_{K(1)}(E^{hG}\wedge V(0))\ .
\]
This spectral sequence has a strong horizontal vanishing line and converges strongly.
See \eqref{def:strong-converge}, \Cref{def:horiz-van},
\Cref{lem:two-cond-van-line}, and \Cref{thm:en-vanishing-ex-1}. 

Recall from \Cref{rem:warningcopy} that even though $V(0)$ is not a ring spectrum,
we often describe the $E_2$-term of a spectral sequences for the homotopy of a localization
of $V(0)$ as a ring. Again, the $E_2$-terms are the underlying graded abelian groups of the rings we write down and the spectral 
sequences are not multiplicative.

The following result from \cite{Paper2}  is the key to all computations. The cohomology
classes $\myoneb=\myoneb_2$ and $e$ appeared in \Cref{cor:cohS21at2}.

\begin{thm}\label{rthm:freeoverv1eta} The localized cohomology group $v_1^{-1}
H_c^\ast(\SS_2^1,E_\ast V(0))$
is free on six generators over the subring $\FF_4[v_1^{\pm 1},\eta]$. Specifically, in the 
algebraic duality resolution spectral sequence
\[
H^q(F_p, E_t V(0)) \Longrightarrow H_c^{p+q}(\SS_2^1,E_t V(0))
\]
there are permanent cycles 
\begin{enumerate}
 
\item $\la_0$ of degree $(p,q,t)=(0,0,0)$ detecting the unit;
 
 \item $\lb_0$ of degree $(p,q,t)=(1,0,0)$ detecting $\myoneb$;
 
 \item $\lb_1$ of degree $(p,q,t)=(1,0,6)$;
 
 \item  $\lc_0$ of degree $(p,q,t)=(2,0,0)$ detecting $\myoneb^2$;

\item $\lc_1$ of degree $(p,q,t)=(2,0,6)$; 
  
 \item  $\ld_0$ of degree $(p,q,t)=(3,0,0)$ detecting $e$.
\end{enumerate}
After inverting $v_1$, there is an isomorphism of $ \F_4[v_1^{\pm 1}, \eta]$-modules
\begin{equation}\label{eqn:abcd}
v_1^{-1}H_c^*(\Sn_2^1, (E_2)_*V(0)) \cong \F_4[v_1^{\pm 1}, \eta] \{\la_0,\lb_0,\lb_1, \lc_0,\lc_1, \ld_0 \}
\end{equation}
\end{thm}

\begin{proof} The only point to add to Corollary 1.2.3 of  \cite{Paper2} 
is the explanation of why $\lb_0$ detects
$\myoneb$, $\lc_0$ detects $\myoneb^2$, and $\ld_0$ detects $e$. But this follows by combining  
\Cref{whatisd0anyway} and \Cref{thm:consts21}.
\end{proof} 

\begin{rem}\label{rem:justify-abuse} The generators of $\la_0$, $\lb_0$, $\lb_1$, $\lc_0$,
$\lc_1$ and $\ld_0$ of \eqref{eqn:abcd} are, strictly speaking, each an element of
$H^0(F_p,E_\ast V(0))$ for some $p$. We have conflated them with their images under
the edge homomorphism
\[
E_2^{p,0} \longrightarrow E_{\infty}^{p,0} \subseteq H_c^p(\SS_2^1, E_\ast V(0))
\]
of the algebraic duality spectral sequence. There should be no
ambiguity. This issue was also discussed in \Cref{rem:figure1explained}. We also let $\lb_1$ and $\lc_1$ denote the cohomology 
classes they detect in $H_c^*(\mathbb{S}_2^1, E_*V(0))$. This allows us to avoid the proliferation of names.
\end{rem}

Our main new computation is the following. It will allow us to rewrite the 
classes $\lb_1$, $\lc_1$, and $\ld_0 = e$ in terms of $\sigma$, $v_1$ and the other generators
of \Cref{rthm:freeoverv1eta}. Note that this result is a statement about cohomology 
classes {\it before} $v_1$-localization.

\begin{thm}\label{thm:what-we-hope-to-prove} Multiplication by $\sigma$ gives the following
identities in the cohomology ring $H_c^*(\Sn_2^1, (E_2)_*V(0))$:
\begin{enumerate}

\item $\sigma = \sigma \la_0= v_1 \lb_1$; 

\item $\sigma \lc_0 = v_1^4 \ld_0$ or, equivalently, $\sigma \myoneb^2 = v_1^4 \ld_0$; and,  

\item $\sigma \lb_0   = \epsilon_0 v_1 \lc_1 +\epsilon_1 v_1^4 \lc_0$
for some $\epsilon_0 \in \F_4^{\times}$ and $\epsilon_1 \in \F_4$. Equivalently,
\[
\sigma \myoneb   = \epsilon_0 v_1 \lc_1 +\epsilon_1 v_1^4 \myoneb^2.
\]
\end{enumerate}
\end{thm}

This is proved below in \Cref{prop:siga0}, \Cref{prop:sigc0}, and
\Cref{prop:sigb0}. For now we will draw some consequences. 
In the following two results, $E(-)$ denotes the exterior algebra over $\FF_2$.

\begin{thm}\label{prop:cohS21v1inv}
There is an isomorphism of graded algebras
\[
\F_4[v_1^{\pm 1}, \eta] \otimes E(\sigma) 
\otimes \F_2[\myoneb]/(\myoneb^3) \cong v_1^{-1}
H_c^*(\Sn_2^1, (E_2)_*V(0))\ .
\]
\end{thm}

\begin{proof} By \Cref{cor:sigma-squared-zero} and \Cref{prop:b0cubedbig} 
we have $\sigma^2 =0$ and $\myoneb^3=0$ in $v_1^{-1}H_c^\ast(\Sn_2^1, (E_2)_*V(0))$.
The result now follows from \eqref{eqn:abcd}, and \Cref{thm:what-we-hope-to-prove}. 
\end{proof}

\begin{thm}\label{thm:cohG2V0}
There are isomorphisms of graded algebras
\begin{align*}
\F_2[v_1^{\pm 1}, \eta] \otimes E(\sigma) \otimes \F_2[\myoneb]/(\myoneb^3) & \cong
v_1^{-1}H_c^*(\G_2^1, (E_2)_*V(0))
\end{align*}
and
\begin{align*}
\F_2[v_1^{\pm 1}, \eta] \otimes E(\sigma, \myzeta_2) \otimes \F_2[\myoneb]/(\myoneb^3) & \cong
v_1^{-1}H_c^*(\G_2, (E_2)_*V(0))\ .
\end{align*}
\end{thm}
\begin{proof} 
The first statement follows immediately from \Cref{prop:cohS21v1inv}
by taking Galois fixed points since $v_1$, $\eta$, $\sigma$ and $\myoneb$ are Galois invariant.
See \Cref{more-split}.

To prove the second statement, note that the elements $v_1$, $\eta$, $\sigma$ and
$\myoneb$ are all in the image of the restriction map
\[
H_c^*(\G_2, (E_2)_*V(0)) \to H_c^*(\G_2^1, (E_2)_*V(0))\ .
\]
Thus the Lyndon-Hochschild-Serre Spectral Sequence for the split extension
$\GG_2^1 \to \GG_2 \to \ZZ_2$ reads
\[
 v_1^{-1}H_c^*(\G_2^1, (E_2)_*V(0)) \otimes E(\zeta_2)
 \Longrightarrow  v_1^{-1}H_c^*(\G_2, (E_2)_*V(0)).
\]
This must collapse and, since $\zeta_2^2=0$ in $H_c^\ast(\GG_2,E_0)$, there can be no 
multiplicative extensions. 
\end{proof}

We now turn to the proof of \Cref{thm:what-we-hope-to-prove}. Our basic computational tool
is \Cref{prop:siggam}. The proofs involve studying $\sigma$ in the integral algebraic
duality spectral sequence 
\begin{equation}\label{eqn:ADSSint}
E_1^{p,q,\ast} \cong H^q(F_p, E_*)  \Longrightarrow H_c^{p+q}(\Sn_2^1, E_*)
\end{equation}
and then reducing modulo $2$. We need the following preliminary result.

\begin{prop}\label{prop:sigraise}
Let $X$ be a finite spectrum. In the algebraic duality spectral sequence 
\[E_1^{p,q} = H^q(F_p, E_*X) \Longrightarrow H_c^{p+q}(\SS_2^1, E_*X)\]
multiplication by $\sigma$ raises filtration in $p$.
\end{prop}

\begin{proof} We saw in \Cref{rem:HS1action} that this is a 
spectral sequence of modules over $H_c^*(\SS_2^1,E_*)$. The module 
structure on the $E_1$-term is given via the restriction homomorphisms 
$H_c^*(\SS_2^1,E_*)\to H^*(F_p,E_*)$, hence it suffices to show that the restriction of 
$\sigma$ considered as a class in $H_c^1(\SS_2^1,E_8)$ restricts trivially to
$H^1(F_p,E_8)$ for all $p$. Because $F_1=F_2\subset F_0$ it suffices to show this for 
$F_0=G_{24}$ and $F_3=G_{24}'$.  In fact, because $\sigma$ comes from
$H_c^1(\SS_2,E_8)$ -- and even from  $H_c^1(\GG_2,E_8)$ -- and because 
$F_0$ and $F_3$ are conjugate in $\SS_2$, it suffices to consider the  case of $G_{24}$. 
Now use Part (2) of \Cref{prop:c4-and-sigma}. 
\end{proof}

We now come to part (1) of \Cref{thm:what-we-hope-to-prove}. 

\begin{prop}\label{prop:siga0}
The mod-$2$ reduction of $\sigma \in H_c^1(\Sn_2^1, E_8)$ to $H_c^1(\Sn_2^1, E_8V(0))$ satisfies 
the equation
\[\sigma = \sigma \Delta_0 =v_1\lb_1.\]
\end{prop}

\begin{proof} By \Cref{rthm:freeoverv1eta}, the class
\[
{\la}_0 \in H^0(G_{24},E_0) = H^0(F_0,E_0)
\]
detects the unit in $H_c^\ast(\SS_2^1,E_\ast V(0))$. 
By \Cref{prop:siggam}, ${d_1(c_4{\la}_0)}/{16}$ detects $\sigma {\la}_0$
in the algebraic duality spectral sequence.
It follows from Definition 5.2.3 and Proposition 5.2.4 of \cite{Paper2} that
\[
\frac{d_1(c_4{\la}_0)}{16} \equiv v_1\lb_1\qquad \mathrm{mod}\ 2 \ .
\]
As a consequence, 
\[
\sigma \la_0 = v_1b_1 \in E_{\infty}^{1,0,8} \subseteq H_c^1(\SS_2^1, E_8V(0)).  \qedhere
\]  
\end{proof}

The proof of the next result uses heavily the techniques and 
language of Section 5 of \cite{Paper2} and in particular
the finer structure of the group $\SS_2$. Recall that in \Cref{rem:g24prime}, we defined
the element $\pi,\alpha \in \WW^{\times} \subseteq \SS_2$ and $\alpha$ was discussed further in \Cref{rem:alpha}. In particular, 
$\alpha \equiv 1+\omega T^2 + T^4$ modulo $(T^6)$. 
Note that $\alpha$ and $\pi$ are both in $\WW^\times$ and hence they commute in $\SS_2$. In \Cref{sec:ellback} we also discussed 
the subgroup $Q_8 \subseteq \SS_2$ and its elements
$i$, $j$, and $k$. See \eqref{eq:ijkreln}.

We can now prove part (2) of \Cref{thm:what-we-hope-to-prove}.

\begin{prop}\label{prop:sigc0} 
In $H_c^3(\Sn_2^1, (E_2)_8V(0))$ there is an equation 
\[
\sigma \myoneb^2 = v_1^4e.
\]
\end{prop}

\begin{proof} We show that $\sigma \lc_0 = v_1^4 \ld_0$. In this proof we will write
\[
E_1^{p,q,t} = H^{p}(F_q,E_t) \Longrightarrow H_c^{p+q}(\SS_2^1,E_t)
\]
for the algebraic duality spectral sequence computing $H_c^\ast(\SS_2^1,E_\ast)$. 
As in \Cref{fig:ADSS} and  \Cref{rem:figure1explained}, we also write 
\[
\lc_0 \in  H^0(C_{6},E_0) = H^0(F_2,E_0) \cong E_1^{2,0,0}
\]
for the integral lift of the class $\lc_0$ of \Cref{rthm:freeoverv1eta}. 
This integral lift is the unit in $H^0(C_{6},E_0)$ and detects $\mytwob$ by \Cref{whatisd0anyway}.
\Cref{prop:siggam} implies that $\sigma \mytwob$ is detected by $d_1(c_4\lc_0)/16$.  By Theorem 1.1.1
of \cite{Paper2}, $d_1 \colon E_1^{2,0,t} \to E_1^{3,0,t}$
is given by the action of
\[
\pi (e+i+j+k) (e-\alpha^{-1}) \pi^{-1}=  \pi (e+i+j+k) \pi^{-1} (e-\alpha^{-1}).
\]
Here, $e,i,j,k \in Q_8 \subseteq G_{24}$ correspond to the unit $e$ and generators $i,j,k$ as in \eqref{eq:ijkreln}.

Lemma 5.2.2 of \cite{Paper2} applied to $\alpha^{-1} \equiv 1 + \omega T^2$
modulo $T^{3}$ together with the fact 
that $\pi^{-1}_*$ is an isomorphism congruent to the identity modulo $(2,u_1)$ 
implies that 
\[
\pi^{-1}_*(e-\alpha^{-1})_*(c_4) \equiv 16 u_1u^{-4} = 16v_1v_2\mod (32,16u_1^2) 
\]
where $v_1 = u_1u^{-1}$ and $v_2 = u^{-3}$. Since $\pi$ and $\alpha$ commute with the 
elements of $C_6$, $\pi^{-1}_*(e-\alpha^{-1})_*(c_4)$ is an element of $E_8^{C_6}$, 
and we can improve the congruence to 
\[
\pi^{-1}_*(e-\alpha^{-1})_*(c_4) \equiv  16v_1 v_2 \mod (32,16u_1^4). 
\]
We compute the action of $ e+i+j+k$ on the right hand side. Since we are reducing modulo $32$, 
it suffices to study the action modulo $2$.

The elements $i$, $j$ and $k$ fix $v_1 \equiv u_1u^{-1}$ modulo $2$. Furthermore, modulo $2$
the formulas in Section 2.4 of \cite{Paper2} give 
\begin{align*}
i_*(v_2) &= v_2(u_1+1)^3, & j_*(v_2) &= v_2(\omega u_1+1)^3, &  k_*(v_2) &= v_2(\omega^2 u_1+1)^3,
\end{align*}
where $\omega \in \WW$ is a third root of unity.
A direct computation using the fact that $1+\omega+\omega^2 =0$ implies that
\[
(e+i+j+k)_*(v_1v_2) = v_1^4 \mod (2).
\]
Since $e+i+j+k$ sends the ideal $(2,u_1^4)$ to $(2,u_1^5)$ and $\pi$ is an isomorphism congruent to the
identity modulo $(2,u_1)$, we conclude that 
\[
(\pi (e+i+j+k) \pi^{-1} (e-\alpha^{-1}))_*(c_4) = 16 v_1^4 \mod (32,16u_1^5).
\]
Therefore $16^{-1}d_1(c_4\lc_0) \equiv v_1^4 \ld_0 \mod (2, u_1^5)$, so that
\[
\sigma \lc_0 \equiv v_1^4 \ld_0 \mod (2, u_1^5)
\]
in $E_{2}^{3,0,8}$.  

By functoriality, this identity also holds in the Algebraic Duality Spectral Sequence
with coefficients $E_*V(0)$. By Theorem 1.2.2 of \cite{Paper2} we have
$ \FF_4\{v_1^4\ld_0\} \cong E_{\infty}^{3,0,8} \subseteq H_c^3(\SS_2^1, E_8V(0))$ in that spectral sequence and
this proves the claim.  
\end{proof}
 
We now come to the proof of the final part of \Cref{thm:what-we-hope-to-prove}.
 
\begin{prop}\label{prop:sigb0}
There exists $\epsilon_0 \in \F_4^{\times}$ and $\epsilon_1 \in \F_4$ such that
\[
\sigma \myoneb   = \epsilon_0 v_1 \lc_1 +\epsilon_1 v_1^4 \myoneb^2
\]
in $H_c^2(\Sn_2^1, (E_2)_8V(0))$. 
\end{prop}

\begin{proof}
We show that $\sigma \lb_0   = \epsilon_0 v_1 \lc_1 +\epsilon_1 v_1^4 \lc_0$ for $\epsilon_0$ and $\epsilon_1$ as in the statement of the proposition.
Since $\lb_0 \in E_{\infty}^{1,0,0}$ and, by \Cref{prop:sigraise}, multiplication by $\sigma$ raises
filtration in the algebraic duality spectral sequence, 
\[
\sigma \lb_0 \in E_{\infty}^{2,0,8}  \cong \F_4\{ v_1\lc_1, v_1^4\lc_0\}.
\]
The calculation of $E_{\infty}^{2,0,8}$ is from Theorem 1.2.2 of \cite{Paper2}.
Because of this, we have an equation
\[
\sigma \lb_0 = \epsilon_0 v_1\lc_1 + \epsilon_1 v_1^4\lc_0
\]
with $\epsilon_i \in \F_4$. Let $\beta$ be the Bockstein homomorphism 
for $H_c^*(\SS_2^1, E_*V(0))$ associated to the short exact sequence 
\[
0 \to (E_2)_*/2  \to (E_2)_*/4 \to (E_2)_*/2 \to 0\ .
\]
Since $\lb_0$ detects $\myoneb$ and $\lc_0$ detects $\myoneb^2$, we have
$\beta(\lb_0)=\lc_0$. Then, because $\sigma$ is an integral class, we have
\[
\beta(\sigma \lb_0) = \sigma \lc_0 = v_1^4 \ld_0 \neq 0.
\]
The second equality is \Cref{prop:sigc0}. However, $v_1^4$ is invariant modulo $4$, so
$\beta(v_1^4 \lc_0) =  v_1^4 \beta(\lc_0)= 0$. This implies that $\epsilon_0 \neq 0$. 
\end{proof}

\subsection{The localized Adams-Novikov Spectral Sequence for $L_{K(1)}L_{K(2)}V(0)$}\label{sec:K1locy}

In \Cref{thm:cohG2V0} we calculated the $E_2$-term of the localized spectral sequence
to be
\[
\F_2[v_1^{\pm 1}, \eta] \otimes E (\sigma, \myzeta_2) \otimes \F_2[\myoneb]/(\myoneb^3)  \cong
v_1^{-1}H_c^*(\G_2, (E_2)_*V(0))\ .
\] 
In this section, we compute the differentials and extensions. The classes $1$, $\eta$, $\sigma$ and $\myzeta_2$ are permanent cycles. We will show that the 
classes $\myoneb$ and $\myoneb^2$ are permanent cycles as well, which is the key step in the computation.
We will see below in \Cref{thm:bzeroinvo} 
that, in fact, they are permanent cycles before $v_1$-localization, but that requires considerably more work.

We break up the computation into a number of results. We will use the universal differential of 
\Cref{rem:unid3}. There is no $\eta$-torsion at $E_3$ so this becomes 
\begin{equation}\label{eq:unidiffapplied}
d_3(v_1^2z)=v_1^2d_3(z)+\eta^3z \ .
\end{equation}

Our first result is a list of basic structural properties of our spectral sequences.

\begin{lem}\label{lem:serious-prelims} The localized Adams-Novikov Spectral Sequences
\begin{align*}
v_1^{-1}H_c^*(\G^1_2, (E_2)_*V(0))&\Longrightarrow \pi_\ast L_{K(1)}(E^{\GG_2^1}\wedge V(0))\\
\\
v_1^{-1}H_c^*(\G_2, (E_2)_*V(0))&\Longrightarrow \pi_\ast L_{K(1)}L_{K(2)}V(0)
\end{align*}
have the following properties: 
\begin{enumerate}[(1)]

\item Both spectral sequences are modules over $\FF_2[v_1^{\pm 4},\eta] \otimes E(\sigma,\zeta_2)$.
\item In the first spectral sequence $\zeta_2$ acts by zero.
\item In both spectral sequences $d_{2r} = 0$ for all $r \geq 1$. 
\item In both spectral sequences $\eta^3E_4^{\ast,\ast}=0$. 
\item For the first spectral sequence $E^{s,t}_4=0$ for $s \geq 6$ and for the second spectral sequence
$E_4^{s,t}=0$ for $s \geq 7$. 
\item In both spectral sequences the classes $1$ and $v_1$ are permanent cycles.
\end{enumerate}
\end{lem} 

\begin{proof} For the first statement we use that $\eta$, $\sigma$, and $\zeta_2$
are permanent cycles for the sphere and that $V(0)$ has a $v_1^4$ self-map. For the second 
statement, we use that $\zeta_2 = 0$ in $\pi_\ast E^{h\GG_2^1}$. See \Cref{prop:zeta-permanent}.
Since $E_\ast V(0)$ is concentrated in even degrees, the third statement is the
standard sparseness result for the Adams-Novikov Spectral Sequence. 

In particular the first differential is $d_3$. There is no $\eta$-torsion at $E_3$, so if $d_3(z)=0$,
\eqref{eq:unidiffapplied} gives $d_3(v_1^2z)= \eta^3 z$. Part (4) follows.

To get the vanishing lines, note that every element at the $E_2$-term of the spectral
sequence of $\pi_\ast L_{K(1)}L_{K(2)}V(0)$ is an $\eta$-multiple of the classes
\[
v_1^j\myoneb^i\sigma^{\epsilon_1}\zeta_2^{\epsilon_2}
\]
with $j \in \ZZ$, $0 \leq i \leq 2$ and $\epsilon_i = 0$ or $1$. These classes have $s$-filtration 
at most $4$, so the result follows from part (4). In the spectral
sequence for $\pi_\ast L_{K(1)}(E^{h\GG_2^1}\wedge V(0))$, we have  $\epsilon_2=0$, and we get the
better vanishing line. 

The final statement is true for the $BP$-based Adams-Novikov Spectral Sequence for the sphere,
so is true here as well. Compare \Cref{fig:V0}; see also \eqref{eq:compare-ANSS}. 
\end{proof}

In the next proof we will use the homotopy Bockstein 
homomorphism 
\[
\beta \colon \pi_{n}L_{K(1)}L_{K(2)}V(0) \to \pi_{n-1}L_{K(1)}L_{K(2)}V(0).
\]
See \Cref{rem:betabock}. Any element in the image of $\beta$ is of order $2$. 

Our next result is a preliminary calculation with $E^{h\GG_2^1}$. \Cref{fig:LK1LK2V0} below gives a partial illustration.
\begin{lem}\label{prop:chipermG1}
For the spectral sequence
\begin{equation}\label{eq:extra-localized1}
v_1^{-1}H_c^\ast(\G_2^1,(E_2)_*V(0))\Longrightarrow \pi_*L_{K(1)}(E^{h\GG_2^1}\wedge V(0)) 
\end{equation}
we have the following:
\begin{enumerate}

\item The classes $\myoneb$, $\myoneb^2$, and $v_1\myoneb^2$, are permanent cycles. 
If $y$ is any class detected by $\myoneb$, then $\beta(y)$  is detected by $\myoneb^2$, and $2y = \beta(y)\eta$.

\item The spectral sequence collapses at $E_4$.
The non-trivial differentials are determined by the differential
\[
d_3(v_1 \chi) = \eta^2 \chi^2
\]
and the fact that the spectral sequence is a module over $\FF_2[v_1^{\pm 4},\eta] \otimes E(\sigma)$.

\item The class $\eta v_1 \myoneb+v_1^2\myoneb^2$ 
is a permanent cycle. If $x$ is any class detected by $\eta v_1 \myoneb+v_1^2\myoneb^2$, 
then $\beta(x)$ is detected by $\eta^2\chi+\eta v_1\chi^2$,
and $2x = \beta(x)\eta$.
\end{enumerate}
\end{lem}

\begin{proof} Recall we have
\[
\F_2[v_1^{\pm 1}, \eta] \otimes E (\sigma) \otimes \F_2[\myoneb]/(\myoneb^3) \cong
v_1^{-1}H_c^*(\G_2^1,(E_2)_*V(0))\ . 
\]

We begin part (1). The class $\myoneb$ is a $d_3$-cycle by \Cref{prop:b0d3cycle}, hence
it must be a permanent cycle by part (5) of \Cref{lem:serious-prelims}. 
Let $y$ be a class detected by $\myoneb$ in the spectral sequence \eqref{eq:extra-localized1}. 
Since $\myoneb^2$ is the algebraic Bockstein of $\myoneb$, the Geometric
Boundary Theorem (see \Cref{rem:whats-up-GBT}) 
implies $\myoneb^2$ is a permanent cycle for this spectral sequence and detects the
Bockstein $\beta(y)$. By \Cref{lem:twoxtildes10} there must
be an  additive extension $2y=\beta(y)\eta$ in the spectral sequence.

To see $v_1\myoneb^2$ is a 
permanent cycle, note that \Cref{prop:b0d3cycle} and the Geometric Boundary Theorem
imply that $\mytwob \in H_c^2(\GG_2,E_0)$ is a $d_3$-cycle in the Adams-Novikov Spectral
Sequence for $L_{K(2)}S^0$. Since $\myoneb^2$ is the reduction of $\mytwob$ and $v_1$ is a $d_3$-cycles,
we have  $\myoneb^2v_1$ is a $d_3$-cycle. It is then a permanent cycle by
the vanishing line of part (5) of \Cref{lem:serious-prelims}.

We now take on part (2). Since the class $\eta \myoneb^2$ detects a class divisible by $2$, we must have
$\eta^2 \myoneb^2 =0$  at 
$E_\infty$. It follows that there must be a non-trivial differential originating at $E_3^{1,2}$. This vector space
has a basis given by $v_1^{-3}\sigma$ and $v_1 \myoneb$. 
The class $v_1^{-3}\sigma$ is a permanent cycle, since $v_1$ is a permanent cycle
and the spectral sequence is a module over $\F_2[v_1^{\pm 4},\sigma]$; hence, we must have
the claimed differential
\[
d_3(v_1 \myoneb) =\eta^2 \myoneb^2.
\]
We now apply  \Cref{lem:serious-prelims} parts (1), (2), and (5) to get the collapse at $E_4$ and complete
this step in the argument. 

We now turn to part (3). 
From this, the differential from part (2), and \eqref{eq:unidiffapplied}, we can
deduce the following differentials:
\begin{align*}
d_3(v_1^2 \myoneb)&=\eta^3 \myoneb  & d_3(v_1^2\myoneb^2)&= \eta^3 \myoneb^2 \\
d_3(v_1^3 \chi^2) &= \eta^3 v_1 \myoneb^2 & d_3(v_1^3 \myoneb)&= \eta^3v_1 \myoneb + \eta^2 v_1^2\myoneb^2.
\end{align*}
Then $\eta v_1 \myoneb+v_1^2\myoneb^2$ is a $d_3$-cycle and hence a permanent cycle, again by
the the vanishing line from \Cref{lem:serious-prelims}.

Now let $x$ be a class detected by $\eta v_1 \myoneb+v_1^2\myoneb^2$.
We have an equation for the algebraic Bockstein
\[
\beta(\eta v_1 \myoneb+v_1^2\myoneb^2) =    \beta(\eta v_1 \myoneb) = \eta^2\chi+
\eta v_1 \chi^2.
\]
The Geometric Boundary Theorem implies $\eta^2\chi+\eta v_1 \chi^2$ is a permanent cycle and detects the
Bockstein $\beta(x)$. By \Cref{lem:twoxtildes10} there must
be an  additive extension $2x=\beta(x)\eta$ in the spectral sequence.
\end{proof}

\begin{figure} 
\center
\includegraphics[width=0.8\textwidth]{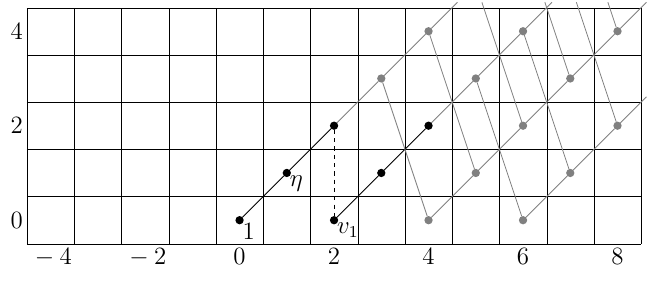}
\includegraphics[width=0.8\textwidth]{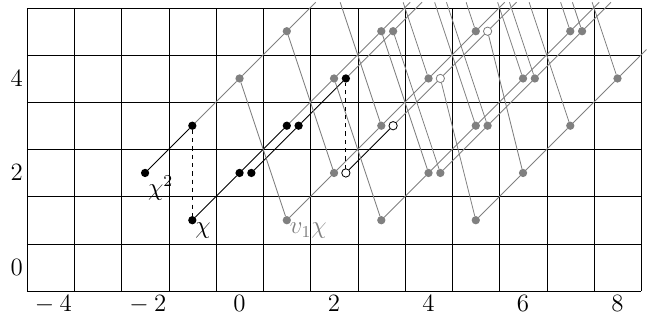}
\captionsetup{width=\textwidth}
\caption{Patterns for the differentials on certain classes in the spectral sequences 
\eqref{eq:extra-localized1} and \eqref{eq:main-localized}.
The top figure illustrates the differential pattern on the unit $1$ and the bottom figure 
the differential pattern on $\chi$ and $\chi^2$. There, a $\circ$ denotes a copy of $\FF_2$ 
generated by an $\eta$-multiple 
of the class $\eta v_1 \myoneb+v_1^2\myoneb^2$. The dashed lines denote the 
extension by $2$ at $E_{\infty}$.
}
\label{fig:LK1LK2V0}
\end{figure}

We now come to one the main calculations of this paper, \Cref{prop:chiperm}. This result shows  
that the localized Adams-Novikov Spectral Sequence \eqref{eq:main-localized} for the spectrum
$L_{K(1)}L_{K(2)}V(0)$ can be obtained from the spectral sequence of \eqref{eq:extra-localized1}
for  $L_{K(1)}(E^{h\GG_2^1}\wedge V(0))$ by tensoring with the exterior algebra $E(\zeta_2)$ 
and extending the differentials
by the Leibniz rule. This is by no means formal. It requires a statement and
argument which, on the surface, looks like a repetition of
\Cref{prop:chipermG1}, but the logic requires us to prove part (2) of \Cref{prop:chipermG1}
before proving part (1) of \Cref{prop:chiperm}.

\begin{thm}\label{prop:chiperm}
The spectral sequence
\begin{equation}\label{eq:main-localized}
v_1^{-1}H_c^*(\G_2,(E_2)_*V(0))\Longrightarrow \pi_*L_{K(1)}L_{K(2)}V(0). 
\end{equation}
collapses at $E_4$. We have 
\begin{enumerate}

\item The classes $\myoneb$, $\myoneb^2$, and $v_1\myoneb^2$, are permanent cycles.
If $y$ is any class detected by
$\myoneb$, then $\beta(y)$  is detected by $\myoneb^2$, and $2y = \beta(y)\eta$.

\item The spectral sequence collapses at $E_4$.
The non-trivial differentials are determined by the differential
\[
d_3(v_1 \chi) = \eta^2 \chi^2
\]
and the fact that the spectral sequence is a module over 
$\FF_2[v_1^{\pm 4},\eta] \otimes E(\sigma,\zeta_2)$.

\item The class $\eta v_1 \myoneb+v_1^2\myoneb^2$ 
is a permanent cycle. If $x$ is any class detected by $\eta v_1 \myoneb+v_1^2\myoneb^2$, 
then $\beta(x)$ is detected by $\eta^2\chi+\eta v_1\chi^2$,
and $2x = \beta(x)\eta$.
\end{enumerate}
\end{thm}

\begin{proof} We have
\[
\F_2[v_1^{\pm 1}, \eta] \otimes E (\sigma,\zeta_2) \otimes \F_2[\myoneb]/(\myoneb^3) \cong
v_1^{-1}H_c^*(\G_2,(E_2)_*V(0))\ . 
\]

We begin with part (1). 
The class $\myoneb$ is a $d_3$-cycle by \Cref{prop:b0d3cycle}.  Thus, 
by parts (3) and (5) of
\Cref{lem:serious-prelims}, $\myoneb$ will be a permanent cycle if it is a $d_5$-cycle.
Using part (4) of \Cref{lem:serious-prelims}, we see that the only possible non-zero differential is
\[
d_5(\myoneb) = v_1^{-4} \zeta_2 \sigma    \eta^2 \myoneb^2.
\]
However,  by naturality with respect to the map $\GG_2^1 \to \GG_2$ and by part (2) of
\Cref{prop:chipermG1}, we have
$d_3(v_1\myoneb) = \eta^2\myoneb^2 +  \zeta_2 x$ for some class 
$x \in E_2^{4,4}$. Since $\zeta_2^2=0$ in $H_c^\ast(\GG_2,E_\ast)$, we can calculate
 \[
 d_3( v_1^{-4}   \zeta_2  \sigma  v_1\myoneb) =v_1^{-4}  \zeta_2\sigma d_3(v_1 \myoneb)
 = v_1^{-4} \zeta_2 \sigma  \eta^2\myoneb^2\ .
 \] 
Hence $v_1^{-4}\zeta_2\sigma\eta^2\myoneb^2$ is zero on the $E_4$-term.
We conclude that $\myoneb$ is a $d_5$-cycle. The remainder of the proof of part (1) goes through
exactly as for part (1) of \Cref{prop:chipermG1}.

We now turn to part (2). As in part (2) of \Cref{prop:chipermG1}, there must be a class 
$a \in E_3^{1,2}$ with $d_3(a) = \eta^2 \myoneb^2$. The vector space $E_3^{1,2}$
has basis $v_1\myoneb$, $v_1\zeta_2$, $v_1^{-3}\sigma$. The class
$v_1\zeta_2$ is a permanent cycle by (1) and (3) of \Cref{lem:serious-prelims}. The
class $v_1^{-3}\sigma$ is a permanent cycle exactly as in the proof of (2) in
\Cref{prop:chipermG1} . Thus we have the indicated differential.

To get the collapse at $E_4$ we use \Cref{lem:serious-prelims} parts (1), (2), (3), and (5).
Part (3) is proved exactly as in part (3) of \Cref{prop:chipermG1}. 
\end{proof} 

\begin{rem}\label{rem:addendum-to-k1-calc} 
We have not listed all the additive extensions in the spectral sequences \eqref{eq:extra-localized1}
and \eqref{eq:main-localized}. In fact we have missed only
those exotic extensions by $2$ implied by the extension
from $v_1$ to $\eta^2$ already visible in \Cref{fig:V0} of \Cref{sec:recollections}. Indeed, it will follow from
\Cref{thm:the-big-one} below that all additive extensions are determined by $v_1^4$ and $\zeta_2$ multiplications and
the exotic extensions by $2$ from $v_1$ to $\eta^2$, from $\myoneb$ to $\myoneb^2\eta$, and from
$\eta v_1 \myoneb+v_1^2\myoneb^2$ to
\[
\eta^3 \myoneb + \eta^2v_1\myoneb^2 \equiv \eta^2v_1\myoneb^2.
\]
The last equivalence follows from the fact that $\eta^3=0$ at $E_4$. These extension are
all shown in \Cref{fig:LK1LK2V0}. The upper figure shows a pattern from $L_{K(1)}V(0)$, and
the lower figure shows a pattern from $L_{K(1)}(V(0)\wedge V(0))$. Compare \Cref{fig:V0}
and \Cref{fig:LK1V0}. 
\end{rem}

\subsection{The homotopy type of $L_{K(1)}L_{K(2)}S^0$}\label{thek1k2subsection}

We are now ready to prove the decomposition of the $K(1)$-localization of $L_{K(2)}S^0$. To begin
we define a key map.

\begin{defn}\label{cor:mapx}
Choose a class
\begin{equation*}
y \in \pi_{-1}L_{K(1)}L_{K(2)}V(0)
\end{equation*}
detected by $\myoneb \in v_1^{-1}H_c^1(\GG_2,E_\ast V(0))$;  this is possible by \Cref{prop:chiperm}. 
If $D(-)$ denotes the Spanier-Whitehead duality functor,
we let 
\[x
\colon \Sigma^{-2}V(0) = \Sigma^{-1}DV(0) \longr  L_{K(1)}L_{K(2)}S^0\
\]
be the SW-adjoint of $y$;  that is, the image of $y$ under the isomorphism
\[
[S^{-1},L_{K(1)}L_{K(2)}V(0) ] \cong [ \Sigma^{-1}DV(0),L_{K(1)}L_{K(2)}S^0 ].
\]
\end{defn}

\begin{rem}
The class $y$ is not unique. Any two choices of $y$ differ by an element of higher Adams-Novikov
filtration. This implies that $x$ is not unique. However,
the proof of \Cref{thm:the-big-one} below comes down to an Adams-Novikov Spectral Sequence 
argument, and we will see that any choice of class $x$ will do. \Cref{thm:bzeroinvo} below will allow us to refine
our choice.
\end{rem}

We prove the following preliminary result.  In \Cref{sec:remonV0}, we gave a short analysis of $\pi_\ast (V(0) \wedge V(0))$. In 
\Cref{rem:pi1V0smV0} and \eqref{eq:choice-ione}, we named classes $i_0$ and $i_1$ of degree $0$ and $1$ respectively
with the property that $\beta(i_1) = i_0$, where $\beta$ is the Bockstein in homotopy as in \Cref{rem:betabock}.

\begin{lem}\label{lem:wither-Bocksteins-x} Let $x$ be as in \Cref{cor:mapx}. Then under the map
\[
x \wedge V(0)\colon \Sigma^{-2}V(0) \wedge V(0) \longr L_{K(1)}L_{K(2)}V(0)
\]
the class $(x \wedge V(0))_\ast (i_0)$ is detected by $\myoneb^2$ and $(x \wedge V(0))_\ast (i_1)$ 
is detected by $\myoneb$ in the spectral sequence
\[
v_1^{-1}H_c^s(\G_2,(E_2)_tV(0))\Longrightarrow \pi_{t-s}L_{K(1)}L_{K(2)}V(0).
\]
\end{lem}

\begin{proof} Since $x$ is the SW-adjoint of $y$ and $y$ is detected by
$\myoneb$, \Cref{what-ew-need-sw} implies $(x \wedge V(0))_\ast (i_1)=y$ 
is detected by $\myoneb$. The same result shows that $(x \wedge V(0))_\ast (i_0) = \beta(y)$.
Then the Geometric Boundary Theorem (see \Cref{rem:whats-up-GBT}) implies
that $(x \wedge V(0))_\ast (i_0)$ is detected by the algebraic Bockstein of $\myoneb$, which is exactly
$\chi^2$.  
\end{proof} 

We now have maps
\begin{align*}\label{the-parts-of-f}
\iota \colon S^0 &\longr L_{K(2)}S^0\\
\zeta_2 \colon S^{-1} & \longr L_{K(2)}S^0\\
x \colon \Sigma^{-2}V(0) & \longr  L_{K(1)}L_{K(2)}S^0\\
x\zeta_2 \colon \Sigma^{-3}V(0) & \longr  L_{K(1)}L_{K(2)}S^0
\end{align*} 
The first of these maps is the unit; the second is the class of \Cref{prop:zeta-permanent}. These maps
assemble into a map
\begin{equation}\label{eq:we-write-the-map}
f = \iota \vee \zeta_2 \vee x \vee x\zeta_2 \colon S^0 \vee S^{-1} \vee \Sigma^{-2}V(0) 
\vee \Sigma^{-3}V(0) \longr L_{K(1)}L_{K(2)}S^0.
\end{equation}
Here is one of our main results.  

\begin{thm}\label{thm:the-big-one} The map $f$ induces a $K(1)$-local equivalence 
\[
L_{K(1)}\big(S^0 \vee  S^{-1}  \vee \Sigma^{-2}V(0) \vee   \Sigma^{-3}V(0)\Big) \simeq
L_{K(1)}L_{K(2)}S^0\ .
\]
\end{thm}

This is a consequence of \Cref{thm:k1k2-y}, but to get there we need some preliminaries. The following
result is standard. 

\begin{lem}\label{lem:type1detects} Let $f \colon X_1 \to X_2$ be a map of spectra. 
Suppose there is a type $n$ complex $Z$ so that $K(n)_\ast Z \ne 0$ and so that the induced map
\[
L_{K(n)} (X_1 \wedge Z) \to L_{K(n)} (X_2 \wedge Z)
\]
is an equivalence. Then $L_{K(n)}X_1 \to L_{K(n)}X_2$ is an equivalence.
\end{lem} 

\begin{proof} 
It suffices to show that $K(n)_*f$ is an isomorphism. We have a diagram
\[
\xymatrix{
K(n)_*(X_1 \smsh Z) \ar[d]_-{\cong} \ar[rr]^-{K(n)_*(f \smsh Z)}_-{\cong} & & K(n)_*(X_2 \smsh Z) \ar[d]^-{\cong}\\
K(n)_*(X_1)\otimes_{K(n)_*}K(n)_*(Z)   \ar[rr]^-{K(n)_*f \otimes 1} & &K(n)_*(X_1)\otimes_{K(n)_*}K(n)_*(Z)  .
}
\]
The vertical maps are the K\"unneth isomorphisms for $K(n)_*$.
Since $Z$ is of type $n$, $K(n)_*Z\neq 0$ and so is a faithful $K(n)_*$-module. 
Therefore, $K(n)_*f$ is an isomorphism.
\end{proof} 

From this last result we see that it is enough to make a judicious choice of type $1$ complex. 
Recall that $Y = V(0) \wedge C(\eta)$. See \Cref{sec:all-about-y}. 
\begin{prop}\label{prop:E2Y}
There are isomorphisms of $\F_2[v_1^{\pm 1}] \otimes E(\sigma)$-modules
\begin{align*}
\F_2[v_1^{\pm 1}] \otimes E(\sigma) \otimes \F_2[\myoneb]/(\myoneb^3) &\cong
v_1^{-1}H_c^*(\G_2^1,E_*Y)\\
\\
\F_2[v_1^{\pm 1}] \otimes E(\sigma, \zeta_2) \otimes \F_2[\myoneb]/(\myoneb^3) &\cong
v_1^{-1}H_c^*(\G_2,E_*Y)\ .
\end{align*} 
\end{prop}

\begin{proof}
By Landweber exactness and \Cref{rem:Ycomodule} we have a short exact sequence of $\GG_2$-modules
\[
0 \to E_\ast V(0) \to E_\ast Y \to E_\ast \Sigma^2V(0)\to 0\ .
\]
Furthermore, in the long exact sequence in group cohomology, the boundary
map is given be multiplication by $\eta$. We then get a long exact sequence for the 
localized cohomology groups $v_1^{-1}H_c^*(\G_2,-)$ and the result follows from \Cref{thm:cohG2V0}. 
\end{proof}

Let $\iota \colon S^0 \to Y$ be inclusion of the bottom cell and let $\zeta_2$ and $\sigma$ be the images 
of the same named class under the induced map
\[
\iota_* \colon \pi_{*}L_{K(1)}L_{K(2)}S^0 \to \pi_{*}L_{K(1)}L_{K(2)}Y \ .
\]
Similarly, let $\jmath \colon V(0) \to Y$ be the inclusion of the bottom two cells; see\eqref{eq:cofV0Y}. Let $y \in \pi_{-1}L_{K(1)}
L_{K(2)}V(0)$ be as in \Cref{cor:mapx} and call the images of $y$ and $\beta(y)$ under the induced map
\[
\jmath_* \colon \pi_{*}L_{K(1)}L_{K(2)}V(0) \to \pi_{*}L_{K(1)}L_{K(2)}Y
\]
by the same names. 

\begin{prop}\label{prop:E2Yhomotopy} The localized spectral sequence 
\[
v_1^{-1}H_c^s(\G_2,(E_2)_tY)\Longrightarrow \pi_{t-s}L_{K(1)}L_{K(2)}Y
\]
strongly converges, collapses at the $E_2$-term, and has no additive extensions.
\end{prop}

\begin{proof} The $E_2$-term was calculated in \Cref{prop:E2Y}. Convergence follows
from \Cref{lem:two-cond-van-line} and \Cref{thm:en-vanishing-ex-1}.
Since $Y$ has a $v_1$-self map,
this is a spectral sequence of modules over $\F_2[v_1^{\pm1}] \otimes E(\sigma,\zeta_2)$. The class
$\zeta_2$ is a permanent cycle by \Cref{prop:zeta-permanent}. The $E_2$-term has a 
horizontal vanishing line at $s=5$ and, by sparseness, $E_2=E_3$. 
By \Cref{prop:b0d3cycle}, the class $\myoneb$ is 
a $d_3$-cycle, and this forces the spectral sequence to collapse.
It follows from \Cref{prop:chiperm} that $y$ is detected by $\myoneb$ and $\beta(y)$
is detected by $\myoneb^2$. 

We turn to extensions; to do this, we must show $2$ annihilates the homotopy of 
$\pi_\ast L_{K(1)}L_{K(2)}Y$. In fact, we have
\[
0 = 2\colon L_{K(1)}L_{K(2)}Y \longr L_{K(1)}L_{K(2)}Y.
\]
To see this, recall that in \Cref{prop:2Y} we showed there is a factoring
\[
\xymatrix{
Y \ar[r]^p \ar@/_1pc/[rr]_2 & S^3 \ar[r]^\nu & Y
}
\]
where $\nu \in \pi_3S^0$ is the Hopf map and $p$ is collapse onto the top cell. In \Cref{prop:2Y}, we
also showed that $0 = \nu \in \pi_\ast L_{K(1)}Y$; hence, 
\[0 = \nu \in \pi_\ast L_{K(1)}L_{K(2)}Y.\qedhere\]  
\end{proof}

\Cref{thm:the-big-one} now follows from \Cref{lem:type1detects} and the next result.

\begin{thm}\label{thm:k1k2-y} The map $f$ of \eqref{eq:we-write-the-map} induces a $K(1)$-local
equivalence
\[
L_{K(1)}\big(S^0\vee S^{-1}\vee\Sigma^{-2}V(0)\vee\Sigma^{-3}V(0)\Big)\wedge Y
\simeq L_{K(1)}L_{K(2)}Y\ .
\]
\end{thm}

\begin{proof} 
As a module over $\FF_2[v_1^{\pm 1}] \otimes E(\sigma)$, $\pi_\ast L_{K(1)}Y$ 
is free of rank $1$ on a class $\iota$ of degree $0$ and $\pi_\ast L_{K(1)}(Y \wedge V(0))$ 
is free of rank $2$ on the classes $i_0$ and $i_1$ of degrees $0$ and $1$ respectively.
See \Cref{prop:k1-of-y}, \Cref{cor:yandv0}, and the remarks before \Cref{cor:yandv0}. By
\Cref{prop:E2Y} and \Cref{prop:E2Yhomotopy} we have that 
$\pi_\ast L_{K(1)}L_{K(2)}Y$ is free of rank $6$ over  
$\FF_2[v_1^{\pm 1}] \otimes E(\sigma)$  on generators $\iota$, $\zeta_2$, $y$, $\beta(y)$, 
$\zeta_2 y$ and $\zeta_2 \beta(y)$ of degrees $0$, $-1$, $-1$, $-2$,  $-2$ and $-3$
respectively.  See the remarks before  \Cref{prop:E2Yhomotopy} as well.

We show that the maps
\begin{align*}
\iota\wedge Y\colon&\ L_{K(1)}Y \longr L_{K(1)}L_{K(2)}Y\\
\zeta_2\wedge Y\colon&\ \Sigma^{-1}L_{K(1)}Y  \longr L_{K(1)}L_{K(2)}Y\\
x\wedge Y\colon&\ \Sigma^{-2}L_{K(1)}(V(0) \wedge Y)  \longr  L_{K(1)}L_{K(2)}Y\\
x\zeta_2\wedge Y\colon &\ \Sigma^{-3}L_{K(1)}(V(0) \wedge Y)  \longr  L_{K(1)}L_{K(2)}Y
\end{align*} 
are all injective on homotopy and exhaust the various summands. 

The map
\[
\iota\wedge Y \colon L_{K(1)}(S^0 \smsh Y)    \to  L_{K(1)}L_{K(2)}(S^0\smsh Y)
\] 
is injective on homotopy and maps onto the $\F_2[v_1^{\pm1}] \otimes E(\sigma)$ summand generated by
$\iota$. Similarly $\zeta_2 \smsh Y$ maps injectively onto the summand generated by $\zeta_2$.

This leaves the other four summands. Using the inclusion $\jmath \colon V(0) \to Y$ and 
\Cref{lem:wither-Bocksteins-x} we have that 
\[
x\wedge Y\colon \Sigma^{-2}L_{K(1)}(V(0) \wedge Y)  \longr  L_{K(1)}L_{K(2)}Y
\]
sends the generators $i_0$ and $i_1$ in degrees $-2$ and $-1$ to non-zero classes detected
by $\myoneb^2$ and $\myoneb$ respectively. By multiplying with $\zeta_2$ we can conclude that
\[
x\zeta_2\wedge Y\colon  \Sigma^{-3}L_{K(1)}(V(0) \wedge Y)  \longr  L_{K(1)}L_{K(2)}Y
\]
sends the generators $i_0$ and $i_1$ in degrees $-3$ and $-2$ to non-zero classes detected
by $\myoneb^2\zeta_2$ and $\myoneb\zeta_2$ respectively. Since $y$ is detected by $\myoneb$
and $\beta(y)$ by $\myoneb^2$, the result follows from \Cref{prop:E2Y}.
\end{proof}

\subsection{The homotopy types of $L_0L_{K(2)}S^0$ and $L_1L_{K(2)}S^0$}\label{sec:getting-to-l1}
We end this section by recording one of the main consequences of \Cref{thm:the-big-one}, 
which was stated in the introduction as \Cref{thm:main}. Recall that $X_p$ denote the
$p$-completion of $X$. 

\begin{thm}[\textbf{Chromatic Splitting}]\label{thm:L1L2}
There is an equivalence
\[L_1L_{K(2)}S^0 \simeq L_1(S^0_{2} \vee S^{-1}_2) \vee L_0(S^{-3}_2 \vee S^{-4}_2) 
\vee L_1(\Sigma^{-2}V(0)\vee\Sigma^{-3}V(0)) \ .
\]
\end{thm}

\Cref{thm:L1L2} is proved exactly as in Theorem 5.11 of \cite{GoerssSplit}, i.e. 
by examining the homotopy pull-back
\[\xymatrix{L_1L_{K(2)}S^0 \ar[d] \ar[r] & L_{K(1)}L_{K(2)}S^0 \ar[d] \\
L_0L_{K(2)}S^0 \ar[r]  & L_0L_{K(1)}L_{K(2)}S^0 \ .}\]
For this, we need the following result which describes $L_0L_{K(2)}S^0$.

\begin{thm}\label{thm:mainrat}
There is an equivalence
\begin{align*}
L_{0}L_{K(2)}S^0 \simeq  L_0(S^0_2 \vee S^{-1}_2 \vee S^{-3}_2 \vee S^{-4}_2).
\end{align*}
\end{thm}
\begin{proof}
By \Cref{prop:ratcohE}, the spectral sequence
\[
E_2^{s,t}=H_c^s(\GG_2, E_t)\otimes \QQ \Longrightarrow \pi_{t-s}L_0L_{K(2)}S^0_2.
\]
has $E_2$-term  $E_{\QQ_2}(\zeta_2, e)$. It collapses with no extensions for degree reasons,
and it converges by \Cref{lem:two-cond-van-line} and \Cref{thm:en-vanishing-ex-1}. 

The homotopy classes detected by $1$, $\zeta_2$, $e$, $e\zeta_2$ combine to give a map
\[
S^0_2 \vee S^{-1}_2 \vee S^{-3}_2 \vee S^{-4}_2 \to L_{0}L_{K(2)}S^0
\]
which is a rational equivalence.
\end{proof}


\section{Lifting to $L_{K(2)}S^0$}\label{sec:K2locy}

In \Cref{sec:k1k2s} we constructed a map
\[
x \vee x\zeta_2:\Sigma^{-2}V(0) \vee \Sigma^{-3}V(0) \longr L_{K(1)}L_{K(2)}S^0
\]
which became part of our $K(1)$-equivalence. The map $x$, which appeared in \Cref{cor:mapx}, was not unique, but any choice
yielded the equivalence. In this section we show that there is a choice for this
map that factors through the localization map
\[
L_{K(2)}S^0 \to L_{K(1)}L_{K(2)}S^0.
\]
This is important for the chromatic assembly process.

The map $x$ is defined in \Cref{cor:mapx} to be SW adoint to a class
\[
y \in \pi_{-1}L_{K(1)}L_{K(2)}V(0)
\]
detected by $\myoneb = \myoneb_2 \in v_1^{-1}H_c^1(\GG_2,E_\ast V(0))$. 
The key result of this section is \Cref{thm:bzeroinvo}, which shows that there is a
choice of $y$ in the image of the map from $\pi_{-1}L_{K(2)}V(0)$.

There are two main steps: we first show
$\pi_{-1}(E^{h\GG_2^1}\wedge V(0)) \cong \ZZ/4$ generated by a class detected by $\myoneb \in
H_c^1(\GG_2^1,E_0/2)$; this uses the topological duality spectral sequence of \Cref{thm:top-dual-res}. 
We then show this generator is in the image of the map
\[
\pi_{-1}L_{K(2)}V(0) \longr \pi_{-1}(E^{h\GG_2^1}\wedge V(0))
\]
using entirely different information, ultimately going back to \Cref{prop:b0d3cycle}. 

\subsection{Computing $\pi_iE^{h\GG_2^1}$, $-3 \leq i \leq 0$} For this calculation
we use the topological duality spectral sequence of \Cref{rem:topdualresSS}:
\[
\pi_q\sE_p \Longrightarrow \pi_{q-p}E^{h\mathbb{S}_2^1}.
\]
We will use Lemmas \ref{coh-decomp-galois} and \ref{more-split} to deduce information
about $\pi_\ast E^{h\GG_2^1}$. The result is recorded in \Cref{thm:f4}. The homotopy
of $E^{h\GG_2^1}$ in the indicated range is given in \Cref{thm:piG21}.

The first step is to record what we need about the homotopy groups of the fibers $\sE_p$. Since 
we will be trying to compute $\pi_iE^{h\SS_2^1}$  for $i$ near $0$, we will only need information
in a small range of dimensions. 

\begin{rem}\label{rem:hur} If $p=0$, then $\sE_p = E^{hG_{24}}$; if $p=3$,  then
$\sE_p=\Sigma^{48}E^{hG_{24}}$. Write $G_{48} = G_{24} \rtimes \Gal$. Then the computation
of $\pi_\ast E^{hG_{48}}$ can be mined from \cite{tbauer} or \cite{tmfbook}; it is
very elaborate, but the part we need is fairly simple. We then have $\pi_\ast E^{hG_{24}} \cong
\WW \otimes \pi_\ast E^{hG_{48}}$. See \Cref{more-split}.

Recall from \Cref{input-for-the-DSS}  that there
are classes $c_4$, $c_6$, $\Delta$, and $j$ in $H^0(G_{48},E_\ast)$, of degrees
$8$, $12$, $24$, and $0$ respectively, and an isomorphism
\[
\ZZ_2[[j]][c_4,c_6,\Delta^{\pm 1}]/I \cong H^0(G_{48},E_\ast)
\]
where $I$ is the ideal generated by
\[
c_4^3 - c_6^2 = (12)^3 \Delta\qquad\mathrm{and}\qquad j\Delta = c_4^3.
\]
The class $j$ is a permanent cycle. The class $\Delta^{-2}$ is not a permanent cycle, 
but $4\Delta^{-2}$ and $j\Delta^{-2}$ are.
This will  explain the complicated form of $\pi_{-48}E^{hG_{24}}$ in \eqref{eq:homotopy-for-g24} and the
factor of $4$ in part (1) of \Cref{thm:piG21}. See also \Cref{rem:eis4dzero}. 
Let $\eta$ and $\nu$ be the Hopf classes in the homotopy groups of spheres. 

We pause to explain some notation. All homotopy groups $\pi_\ast E^{hG_{24}}$ are modules over
$\WW[[j]] \cong \pi_0E^{hG_{24}}$. For a quotient ring $R$
of $\WW[[j]]$ we write $R\{x\}$ for the free $R$-module generated by $x$. If
$J \subseteq R$ is an ideal, we write $J\{x\} \subseteq R\{x\}$ for the sub-$R$-module
generated by $J$. 
The name of a generator, if given, is the
label given to the class in the cohomology $H^\ast(G_{24},E_\ast) \cong \WW \otimes_{\ZZ_2}
H^\ast(G_{48},E_\ast)$ which detects the element in question. In all the degrees we list,
$\pi_iE^{hG_{24}}$ is a sub-$\WW[[j]]$-module of a single cohomology group $E_2^{s,t} \cong
H^s(G_{24},E_t)$, and we indicate that cohomology group.

With this in mind we now record the homotopy groups we need. 
\begin{equation}\label{eq:homotopy-for-g24}
\pi_{i}E^{hG_{24}}\cong 
\begin{cases}
\W[[j]] \cong E_2^{0,0}  & i=0 \\
\FF_4[[j]]\{\eta^i\} \cong E_2^{i,2i} & i=1,2 \\
\WW/8\{\nu\} \cong E_2^{1,4} & i=3\\
0         &   i=-3,-2,-1  \\ 
(4,j)\{\Delta^{-2}\}\subseteq \W[[j]]\{\Delta^{-2}\} \cong E_2^{0,\text{-}48} & i=-{48}  \\
(j)\{\eta\Delta^{-2}\} \subseteq \FF_4[[j]]\{\eta\Delta^{-2}\} \cong E_2^{1,\text{-}46} & i=-47 \\ 
(j)\{\eta^2\Delta^{-2}\} \subseteq \FF_4[[j]]\{\eta^2\Delta^{-2}\} \cong E_2^{2,\text{-}44} & i=-46  \\
\W/8\{\nu\Delta^{-2}\} \cong E_2^{1,\text{-}44}    &   i=-45  
\end{cases}
\end{equation}
The edge homomorphism $h\co \pi_iE^{hG_{24}} \to E_i^{G_{24}}$ is an 
isomorphism when $i=0$ and injective when $i=-48$. 

If $p=1$ or $2$, then $F_p = E^{hC_6}$. The following computation is relatively
simple and can be read off of Section 4 of \cite{MR} or, more explicitly, from Section
2 of \cite{BobkovaGoerss}. The notation is analogous to that
in \eqref{eq:homotopy-for-g24}, except that $\pi_3E^{hC_6} \cong E_\infty^{3,6}$ is only
a quotient of $E_2^{3,6}$.
\begin{equation}\label{eq:homotopy-for-c6}
\pi_{i}E^{hC_6}\cong 
\begin{cases}
0  & i=-1,-2 \\
\W[[u_1^3]] \cong E_2^{0,0} & i=0  \\
\FF_4[[u_1^3]]\{\eta^i\} \cong E_2^{i,2i}  & i=1,2 \\ 
\FF_4 \{ \nu \} & i=3 \\
\end{cases}
\end{equation}
The edge homomorphism $h\co \pi_0E^{hC_6} \to E_0^{C_6}$ of the homotopy fixed 
point spectral sequence is an isomorphism. 
\end{rem}

Finally, we will need the following result in our calculations below.

\begin{lem}\label{eq:the-extra-stuff-hw-wanted} Let $G=G_{24}$ and $G = C_6$.
Multiplication by $\eta$ induces isomorphisms
\begin{align*}
\eta\colon H^0(G,E_0/2) &\mathop{\longr}^{\cong} H^1(G,E_2/2)\\
\eta^2\colon H^0(G,E_0/2) &\mathop{\longr}^{\cong} H^2(G,E_4/2).
\end{align*}
Furthermore, for $i=1,2$ reduction modulo $2$ induces an isomorphism
\[
\xymatrix{
H^i(G,E_{2i}) \ar[r]^-{\cong} & H^i(G,E_{2i}/2).
}
\]
\end{lem}

\begin{proof} For these statements we will need some information about
$H^s(G,E_t)$ for low values of $s$ and $t$. There are many references
for this; see \cite{MR} for $C_6$, \cite{tmfbook} for $G_{24}$, or \S 2 of \cite{BobkovaGoerss}
for more references and a convenient chart for $G= G_{24}$.

From these references we see that there are isomorphisms
\[
\xymatrix{
H^0(G,E_0)/2 \ar[r]^{\eta}_\cong & H^1(G,E_2)  \ar[r]^{\eta}_\cong & H^2(G,E_4).
}
\]
We also have that $H^s(G,E_t) = 0$ for $(s,t) = (1,0)$, $(2,2)$, and $(3,4)$.
The result follows from the long exact sequence in group cohomology induced by the short exact sequence
\[
\xymatrix{
0 \ar[r] & E_\ast \ar[r]^-{\times 2} & E_\ast \ar[r] & E_\ast/2 \ar[r] & 0.
}\qedhere
\]
\end{proof}

We now begin our calculation of the topological duality spectral sequence
for $E^{h\SS_2^1}$. The result is illustrated in \Cref{fig:topduality}.
The main calculations needed to produce the $E_2$-term can be found in
\Cref{E2*0-comp} and \Cref{thm:f4}. 
There is an explanatory remark for these charts immediately following.

\begin{center}
\begin{minipage}{\textwidth}
\begin{figure}[H]
\begin{minipage}{.49\textwidth}
\centering
\includegraphics[width=0.6\textwidth]{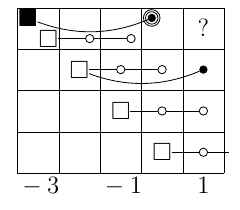}
\end{minipage}
\begin{minipage}{.49\textwidth}
\centering
\includegraphics[width=0.6\textwidth]{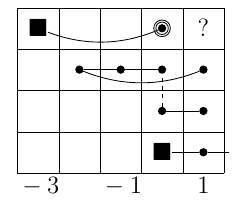}
\end{minipage}
\caption{{The $E_1$ (left) and $E_{2}$ (right) terms of the TDSS for $\pi_*(E^{h\mathbb{S}_2^1})$.
}}
\label{fig:topduality}
\end{figure}
\end{minipage}
\end{center}

\begin{rem}\label{rems-on-TDRSS4-alt} The two charts are respectively the $E_1$-term  and
$E_2$-term of the topological duality spectral sequence of \Cref{rem:topdualresSS} for
$\pi_*(E^{h\mathbb{S}_2^1})$. 
We use the Adams grading
with $E_1^{p,q}$ in \Cref{rem:topdualresSS} in the $(q-p,p)$ box. 

A $\Box$ denotes a copy of $\W[[j]]$ if $s=0,3$ and $\W[[u_1^3]]$ if $s=1,2$. Similarly,
a $\circ$ denotes a copy 
of $\F_4[[j]]$ or $\F_4[[u_1^3]]$. 
A $\bullet$ is a copy of $\mathbb{F}_4$. A $\blacksquare$ is a copy  
of $\W$. The symbol \circled{\circled{$\bullet$}} denotes a copy of $\W/8$. 

Horizontal and curved lines indicate multiplication by $\eta$ respectively $\nu$.
The dashed vertical line on the $E_2$-term is there to indicate the additive extension 
on the $E_{\infty}$-term.  See \Cref{thm:f4}.
In the top row of the left chart, we have depicted the ideal $(4,j)\{\Delta^{-2}\} 
\subseteq \pi_{-48}E^{hG_{24}}$ using the $\WW$-module structure as
$\WW\{4\}\oplus \WW[[j]]\{j\} \subseteq \WW[[j]]$. The $\blacksquare$ represents $\WW\{4\}$ and 
the $\Box$ represents $\WW[[j]]\{j\}$.

The upper right corners of both charts contribute to $\pi_1 E^{h\SS_2^1}$. 
The elements in these spots do not support 
differentials in the topological duality spectral sequence, so don't come into the calculations
of $\pi_0$. As they don't affect our answers, we don't include them in the discussion and leave
these boxes marked with a question mark. 
\end{rem}

We now come to the calculation of $\pi_\ast E^{h\SS_2^1}$ in the range we need. We compute using the topological duality 
spectral sequence of  \Cref{rem:topdualresSS}:
\begin{equation*}
\xymatrix{E_1^{p,q} = \pi_{q}\sE_p \ar@{=>}[r] & \pi_{q-p}(E_{2}^{h\bS_{2}^1})}.
\end{equation*}
We will use information from \Cref{rem:hur}. The $E_1$-term is displayed in \Cref{fig:topduality}. 
The notation $\la_0$, 
$\lb_0$, $\lc_0$, and $\ld_0$ was explained in \Cref{rem:figure1explained}. 

\begin{thm}\label{thm:f4bis}
There are isomorphisms
\begin{enumerate}

\item $\pi_{-3} (E^{h\bS_{2}^1}) \cong \WW$ with generator detected
by $4\ld_0 \in E_1^{3,0} = \pi_0\Sigma^{48}E^{hG_{24}}$.

\item $\pi_{-2}(E^{h\bS_{2}^1}) \cong \F_4$ detected by $\lc_0 \in 
E_1^{2,0} = \pi_0E^{hC_6}$.

\item $\pi_{-1}(E^{h\bS_{2}^1}) \cong \F_4$ with generator
detected by $\eta \lc_0$ in $E_1^{2,1} = \pi_1E^{hC_6}$.

\item $\pi_0(E^{h\mathbb{S}_2^1}) \cong \W \oplus \W/4 \oplus \W/8$.
\end{enumerate}

Furthermore, for $\pi_0E^{h\SS_2^1}$,

\begin{enumerate}

\item[(i)] a generator of the summand $\WW$ is detected by the unit in $E_1^{0,0} = \pi_0E^{hG_{24}}$;

\item[(ii)] a generator of the summand $\W/4$ is detected by $\eta\lb_0 \in E_1^{1,1} = \pi_1E^{hC_6}$;

\item[(iii)] a generator of the summand $\WW/8$ is detected by $\nu\ld_0 \in E_1^{3,3} = \pi_3\Sigma^{48}E^{hG_{24}}$.
\end{enumerate} 

If a class in $\pi_0E^{h\SS_2^1}$ is detected by $\eta\lb_0 \in E_1^{1,1} $ then twice that class is detected by
$\eta^2\lc_0 \in E_1^{2,2}\cong \pi_2E^{hC_6}$. 

\end{thm}

\begin{proof}
This requires a multi-part argument and we break it into steps.

{\bf Calculating $d_1$ with $q=0$.} The fundamental calculation is with the sequence
\[
\xymatrix{
E_1^{0,0} \ar[r]^{d_1} & E_1^{1,0} \ar[r]^{d_1}  & E_1^{2,0} \ar[r]^{d_1}  & E_1^{3,0}\ .
}
\]
Using the calculations of \Cref{rem:hur}, this can be fit into a diagram of chain complexes 
\begin{equation*}
\xymatrix{
\pi_0E^{hG_{24}} \ar[r]^{d_1} \ar[d]_\cong & \pi_0E^{hC_6} \ar[r]^{d_1} \ar[d]_\cong &
\pi_0E^{hC_6} \ar[r]^{d_1} \ar[d]_\cong & \pi_{-48}E^{hG_{24}} \ar[d]_\subseteq\\
H^0(G_{24},E_0) \ar[r]^{d_1} &H^0(C_6,E_0) \ar[r]^{d_1}&H^0(C_6,E_0) \ar[r]^{d_1}
&H^0(G_{24},E_0)\\
H^0(G_{24},\WW) \ar[r]^{d_1} \ar[u]&H^0(C_6,\WW) \ar[r]^{d_1}\ar[u]
&H^0(C_6,\WW) \ar[r]^{d_1}\ar[u]
&H^0(G_{24},\WW)\ar[u]\ .
}
\end{equation*} 
The  vertical maps down from  the top row are edge homomorphisms. 
All but the last of these maps are isomorphisms. The last map is injective with image 
the ideal $(4,j) \subseteq \WW[[j]]$. The map from the bottom
to the middle row is induced by the unit map $i \colon \WW \to E_0$; thus, by
\Cref{E2*0-comp}, the bottom two rows have the same cohomology. It then follows that in the 
topological duality spectral sequence we have
\[
E_2^{0,0} \cong \WW\qquad E_2^{1,0} = 0\qquad E_2^{2,0} \cong \FF_4\qquad
E_2^{3,0} \cong \WW.
\]
The generator for $E_2^{2,0}$ is $\lc_0$ and the generator for $E_2^{3,0}$ is $4\ld_0$.

{\bf Calculating $d_1$ with $q=1$.} The next calculation is with the sequence
\[
\xymatrix{
E_1^{0,1} \ar[r]^{d_1} & E_1^{1,1} \ar[r]^{d_1}  & E_1^{2,1} \ar[r]^{d_1}  & E_1^{3,1}\ .
}
\]
This can be fit into a diagram 
\[
\xymatrix{
\pi_1E^{hG_{24}} \ar[r]^{d_1} \ar[d]_{\cong} &\pi_1E^{hG_{6}} \ar[r]^{d_1}\ar[d]_{\cong}&
\pi_1E^{hG_{6}} \ar[r]^{d_1} \ar[d]_{\cong} &\pi_{-47}E^{hG_{24}}\ar[d]^{\subseteq}\\
H^1(G_{24},E_2) \ar[r]^{d_1} \ar[d]_{\cong}&H^1(C_6,E_2) \ar[r]^{d_1} \ar[d]_{\cong}
&H^1(C_6,E_2) \ar[r]^{d_1} \ar[d]_{\cong} &H^1(G_{24},E_2)\ar[d]_{\cong}\\
H^1(G_{24},E_2/2) \ar[r]^{d_1} &H^1(C_6,E_2/2) \ar[r]^{d_1}&H^1(C_6,E_2/2) \ar[r]^{d_1} 
&H^1(G_{24},E_2/2)\\
H^0(G_{24},E_2/2) \ar[r]^{d_1} \ar[u]_{\cong}^\eta&H^0(C_6,E_2/2) \ar[r]^{d_1} \ar[u]_{\cong}^\eta
&H^0(C_6,E_2/2) \ar[r]^{d_1} \ar[u]_{\cong}^\eta &H^0(G_{24},E_2/2) \ar[u]_{\cong}^\eta\ .
}
\]
Using the calculations of \Cref{rem:hur}, we have all vertical maps from the top row are isomorphisms
except the edge homomorphism on the right hand side. This is an inclusion onto the submodule
$\FF_4[[j]]\{j\eta\} \subseteq \FF_4[[j]]\{\eta\}$.
The maps from the second row to the third row are reduction modulo $2$ and the maps from 
the bottom row to the third row are multiplication by $\eta$. All are isomorphisms
by \Cref{eq:the-extra-stuff-hw-wanted}. It follows that in the topological duality spectral sequence we have
\[
E_2^{0,1} \cong \FF_4\qquad E_2^{1,1} = \FF_4\qquad E_2^{2,1} \cong \FF_4\qquad
E_2^{3,1} \cong 0.
\]
The generator for $E_2^{0,1}$ is $\eta\la_0$, that for $E_2^{1,1}$ is $\eta b_0 $ 
and the generator for $E_2^{2,1}$ is $\eta\lc_0$.

{\bf Calculating $d_1$ with $q=2$.} The exact same argument, with $\eta$ replaced by
$\eta^2$, shows that
\[
E_2^{0,2} \cong \FF_4\qquad E_2^{1,2} = \FF_4\qquad E_2^{2,2} \cong \FF_4\qquad
E_2^{3,2} \cong 0
\]
and that the generator for $E_2^{0,2}$ is $\eta^2\la_0$, that for
$E_2^{1,2}$ is $\eta^2\lb_0$ and the generator for $E_2^{2,2}$ is $\eta^2\lc_0$.

{\bf Calculating $d_1$ with $q=3$.} Finally, we calculate with the sequence
\[
\xymatrix{
E_1^{0,3} \ar[r]^{d_1} & E_1^{1,3} \ar[r]^{d_1}  & E_1^{2,3} \ar[r]^{d_1}  & E_1^{3,3}\ .
}
\]
In this case we need a different style of argument. By \Cref{rem:hur} we have
\[
E_1^{0,3} \cong \WW/8\qquad E_1^{1,3} = \FF_4\qquad E_1^{2,3} \cong \FF_4\qquad
E_1^{3,3} \cong \WW/8
\]
generated by $\nu\la_0$, $\nu\lb_0$, $\nu\lc_0$, and $\nu\ld_0$ respectively. Then
$d_1=0$ by the calculation for $q=0$ and the $\nu$-linearity of $d_1$. 

{\bf Calculating $d_2$ and $d_3$.} 
We first turn to $d_2$. By looking at \Cref{fig:topduality}, we see that in the range $0 \leq q-p \leq 1$, 
the only classes which could support a non-zero $d_2$  
are the unit, detected by $\la_0$ and the classes detected by $\eta\Delta_0$ and
$\eta^2\lb_0$ in $E_2^{0,1}$ and $E_2^{1,2}$.
The unit and $\eta\Delta_0$ are evidently permanent cycles and $\eta^2\lb_0$ is an $\eta$-multiple
of $\eta\lb_0$, which is permanent cycle. Thus $d_2=0$.

This leaves only $d_3$. But in the range $0 \leq q-p \leq 1$,
the only class which could support a  non-zero $d_3$ is $\eta$ times the unit,
so $d_3 = 0$ as well and we have shown
\[
E_2^{p,q} \cong E_\infty^{p,q},\qquad -3 \leq q-p \leq 0.
\]

{\bf Settling extensions.} At this point we have a filtration of the $\WW$-module $\pi_0E^{hG_{24}}$ 
\[
\xymatrix@C=15pt{
0 \ar[r]^\subseteq & F_3 \ar[d]_{=} \ar[r]^\subseteq & F_2 \ar[d] \ar[r]^\subseteq & F_1 \ar[d] \ar[r]^\subseteq &
\pi_0E^{hG_{24}}\ar[d]\\
&E_\infty^{3,3} \cong \W/8&E_\infty^{2,2} \cong \W/2&E_\infty^{1,1} \cong \W/2&E_\infty^{0,0} \cong \W
}
\]
The unit $S^0 \to E^{h\SS_2^1} $ induces a map $\Z_2 \to \pi_0 E^{h\SS_2^1}$ which extends to a Galois equivariant 
homomorphism $\W \to \pi_0 E^{h\SS_2^1}$. This splits off the torsion-free summand.
Thus we need only show that a generator of $E_\infty^{1,1}$ detects an element of order $4$.
This follows from \Cref{rem:basic-toda} and, in particular, from \eqref{order4one}. Specifically,  $\lc_0 \in E_\infty^{2,0}$
detects a unique Galois invariant homotopy class $z$ of order $2$ in $\pi_{-2}E^{h\SS_2^1}$. Since a generator
$E_\infty^{2,2}$ is detected by $\eta^2\lc_0$, we have $2\langle z,2,\eta \rangle = \eta^2 z \ne 0$ and, in particular
that $\langle z,2,\eta \rangle$ does not contain zero. Since $\langle z,2,\eta \rangle$ is Galois invariant, it can only be detected by $\eta\lb_0$.
\end{proof}

The next result describes how the generators 
of $\pi_\ast E^{h\SS_2^1}$ in our range are detected in the homotopy fixed point spectral sequence.
In reading the statement, a bit of care is needed. The class $e$ is not a
permanent cycle, so the class  $\nu e$ of part (iii) is not a product; it is 
named by the class that detects it in the spectral sequence.

\begin{thm}\label{thm:f4}
In the homotopy fixed point spectral sequence
\[
H_c^s(\SS_2^1,E_t) \Longrightarrow \pi_{t-s}E^{h\SS_2^1}
\]
\begin{enumerate}

\item a generator of $\pi_{-3} (E^{h\bS_{2}^1}) \cong \WW$ is detected
by $4e$ in $H_c^3(\SS_2^1,E_0)$;

\item a generator of $\pi_{-2}(E^{h\bS_{2}^1}) \cong \F_4$ is detected by $\mytwob$
in $H_c^2(\SS_2^1,E_0)$;

\item  a generator of $\pi_{-1}(E^{h\bS_{2}^1}) \cong \F_4$ is
detected by $\eta \mytwob$ in $H_c^3(\SS_2^1,E_2)$;

\item $\pi_0(E^{h\mathbb{S}_2^1}) \cong \W \oplus \W/4 \oplus \W/8$. 
\end{enumerate}
Furthermore, for $\pi_0E^{h\SS_2^1}$,

\begin{enumerate}

\item[(i)] a generator of the summand $\WW$ is detected by the unit in $H_c^0(\SS_2^1,E_0)$;

\item[(ii)] a generator of the summand $\W/4$ is detected by an element in the Massey product
\[
\masseyp \in H_c^2(\SS_2^1,E_1);
\]

\item[(iii)] a generator of the summand $\WW/8$ is detected by $\nu e \in H_c^4(\SS_2^1,E_4)$.
\end{enumerate} 
If a class in $\pi_0E^{h\SS_2^1}$ is detected by $\masseyp$ then twice that class is detected by
\[
\eta^2\mytwob \in H_c^4(\SS_2^1,E_4).
\]
\end{thm}

\begin{proof} This follows from \Cref{thm:f4bis} and \Cref{whatisd0anyway}, with judicious uses of
\Cref{lem:chasing-detectors}.

For part (1) we know from \Cref{whatisd0anyway} that $4\ld_0$ detects $4e$ in the algebraic
duality spectral sequence, so \Cref{lem:chasing-detectors} applies directly.

For part (2), the same result shows that $\lc_0$ detects $\mytwob$, and again we can use \Cref{lem:chasing-detectors}. 
Part (3) then follows from Part (2).

It remains to discuss $\pi_0E^{h\SS_2^1}$.

At the very end of the proof of \Cref{thm:f4bis} we argued that if $z \in \pi_{-2}E^{h\SS_2^1}$ and is the 
unique non-zero Galois invariant class, then the Toda bracket $\langle z,2,\eta\rangle$ does not contain
zero and is detected in the topological duality spectral sequence by $\eta\lb_0$. This forces the Massey 
product $\masseyp$ to be non-zero. Since $z$ is detected by $\mytwob$ and $\eta^2z \ne 0$, the exotic
extension follows from
the standard juggling formula $2\langle z,2,\eta\rangle = \eta^2z$  of \eqref{order4one}.
\end{proof}

We next pass to $\GG^1_2$. The result is a corollary of \Cref{thm:f4} obtained by
taking Galois fixed points. 

\begin{cor}\label{thm:piG21} There are isomorphisms
\begin{enumerate}

\item $\pi_{-3}E^{h\GG_2^1} \cong \ZZ_2$ generated by a class detected by $4e
\in H_c^3(\GG_2^1,E_0)$;

\item $\pi_{-2}E^{h\G_2^1}\cong \Z/2$ generated by a class detected by $\mytwob \in H_c^2(\GG_2^1,E_0)$;

\item $\pi_{-1}E^{h\G_2^1}\cong \Z/2$ generated by a class detected by 
$\eta\mytwob\in H_c^3(\GG_2^1,E_2)$;

\item $\pi_0E^{h\G_2^1}\cong \Z_2 \oplus \Z/4 \oplus \Z/8$. 
\end{enumerate}

Furthermore, the summands of $\pi_0E^{h\G_2^1}$ are generated by the unit and classes detected by 
$\masseyp \in H_c^2(\GG_2^1,E_2)$ and $\nu e \in H_c^4(\GG_2^1,E_4)$. If a class
is detected by $\masseyp$, then twice that class is detected by $\eta^2\mytwob$.
\end{cor}

For future reference we will also describe the multiplicative structure of $\pi_0E_2^{h\GG_2^1}$. 
\begin{cor}\label{cor:hw-new-mult} In the kernel of the edge homomorphism of the Adams-Novikov Spectral Sequence
\[
\pi_0E^{h\GG_2^1}\to H^0(\GG_2^1,E_0)
\]
we have the following multiplicative relations.
\begin{enumerate}

\item The product of any two classes detected by $\nu e$ is $0$.

\item The product of any class detected by $\langle \widetilde{\chi},2,\eta \rangle$ 
and any class detected in filtration $4$ is trivial.

\item The product of two classes detected by $\langle \widetilde{\chi},2,\eta\rangle$  is detected by $\eta^2\widetilde{\chi}$
modulo $2\nu e$.
\end{enumerate}
\end{cor}

\begin{proof} Parts (1) and (2) are clear because in the spectral sequence
\[
H^s_c(\GG_2^1,E_t) \Longrightarrow \pi_{t-s}E^{h\GG_2^1}
\]
we have $E_\infty^{s,s} = 0$ for $s > 4$. This follows from \Cref{thm:piG21}. 

For Part (3) we examine the reduction map map
\[
H^s_c(\GG_2^1,E_t) \longr H^s_c(\GG_2^1,E_t/2)
\]
This map kills $2\nu e$.  By definition is sends $\widetilde{\chi}$ to $\chi^2$; hence it sends the
Massey product to $\eta\chi $. The result follows. 
\end{proof}

\begin{rem}\label{rem:backfill-on-pi-3} Using the fiber sequence of \eqref{eq:an-easy-reduction}
\[
\xymatrix{
L_{K(2)}S^0 \ar[r] & E^{h\GG_2^1} \ar[r]^{\pi-1} & E^{h\GG_2^1}
}
\]
we can make analogous calculations for the sphere itself. For example
\[
\pi_{-3}L_{K(2)}S^0 \cong \ZZ_2 \oplus \ZZ/2 \subseteq H_c^3(\GG_2,E_0)
\]
generated by the classes $4e$ and $\zeta_2\mytwob$.
\end{rem}

We now pass to an analysis of $\pi_\ast (E^{h\GG_{2}^1}\wedge V(0))$. Recall from \Cref{sec:remonV0} that $q$
and $j$ are defined by the cofiber sequence
\[
\xymatrix{
E^{h\GG_{2}^1} \ar[r]^-{\times 2}& E^{h\GG_{2}^1} \ar[r]^-j & E^{h\GG_{2}^1} 
\wedge V(0) \ar[r]^-q & \Sigma E^{h\GG_{2}^1}.
}
\]
From \Cref{rem:betabock} we have $\beta = jq$.
Let $z \colon S^{-2} \to E^{h\GG_{2}^1}$ be the class detected by $\mytwob$. Since this
class has order two, it factors through 
\[
q\colon \Sigma^{-1} E^{h\GG_{2}^1}\wedge V(0) \longr E^{h\GG_{2}^1}
\]
and we get a map
\begin{equation}\label{definey}
y_1\colon S^{-1} \longr E^{h\GG_2^1} \wedge V(0).
\end{equation}
such that $q_*(y_1)=z$. We now have the following result.

\begin{prop}\label{prop:V0m2}
There are isomorphisms
\begin{enumerate}[(a)]

\item $\pi_{-1} (E^{h\GG_2^1} \wedge V(0)) \cong \ZZ/4$ generated by the class $y_1$. 
This class is detected by $\myoneb \in H_c^1(\GG_2^1,E_0/2)$.

\item $\pi_{-2}(E^{h\mathbb{G}_2^1}\smsh V(0)) \cong\ZZ/2$. The generator is detected by 
$\myoneb^2 \in H_c^2(\GG_2^1,E_0/2)$.

\end{enumerate}
Furthermore, the class $2y_1$ is detected by $\eta\myoneb^2 \in H_c^3(\GG_2^1,E_2/2)$.
\end{prop}
\begin{proof} From \Cref{thm:piG21} it follows that
\[
\pi_{-2} (E^{h\GG_2^1} \wedge V(0)) \cong \ZZ/2
\]
generated by $j_*(z)$. This class is detected by the image of $\mytwob$ in $H_c^2(\GG_2^1,E_0/2)$, which is exactly $\myoneb^2$. From the same result we have a short exact sequence
\[
0 \to \ZZ/2 \to \pi_{-1}(E^{h\GG_2^1} \wedge V(0)) \to \ZZ/2 \to 0.
\]
Since $j_*(z)$ is detected by the Bockstein on $\myoneb$ and $j_*(y_1)=j_*q_*(z) = \beta (z)$,
the Geometric Boundary Theorem (see \Cref{rem:whats-up-GBT})
implies that the class $y_1$ is detected by $\myoneb$. 
The generator of the kernel is $\eta j_*(z)$, detected by 
$\eta\chi^2$. It remains to show $y_1$ has order $4$; this follows from \Cref{lem:twoxtildes10}.
\end{proof}

\subsection{Producing the lifting}
We now show that the class $y_1 \in \pi_{-1}(E^{h\GG_2^1}\wedge V(0))$ of order $4$ 
constructed above in
\Cref{prop:V0m2} is in the image of the unit map $\pi_\ast L_{K(2)}V(0) \to 
\pi_\ast (E^{h\GG_2^1}\wedge V(0))$. Recall from \Cref{prop:chiperm} that the
class $\myoneb \in v_1^{-1}H_c^1(\GG_2,E_\ast V(0))$ is a non-zero permanent cycle in the
localized Adams-Novikov Spectral Sequence computing $\pi_\ast L_{K(1)}L_{K(2)}V(0)$.

\begin{thm}\label{thm:bzeroinvo} There is a class $y_0 \in \pi_{-1}L_{K(2)}V(0)$ detected
by $\myoneb \in H_c^1(\GG_2,E_0/2)$. Under the map 
\[
 \pi_{-1}L_{K(2)}V(0)\longr \pi_{-1}(E^{h\GG_2^1}\wedge V(0))
\]
the class $y_0$ maps to $y_1$ and under the map
\[
 \pi_{-1}L_{K(2)}V(0)\longr \pi_{-1}L_{K(1)}L_{K(2)}V(0)
\]
the class $y_0$ is non-zero and detected by the class $\myoneb \in v_1^{-1}H_c^\ast(\GG_2,E_\ast V(0))$. 
\end{thm}

\begin{proof} There is a fiber sequence
\begin{equation}\label{eqn:piseq}
\xymatrix{
L_{K(2)}V(0) \ar[r]& E^{h\GG_2^1}\wedge V(0) \ar[r]^{\pi-1} & E^{h\GG_2^1}\wedge V(0)
}
\end{equation}
where $\pi \in \GG_2$ is an element that  generates $\GG_2/\GG_2^1
\cong \ZZ_2$. Since $\pi_{-1}(E^{h\GG_2^1}\wedge V(0))$ is isomorphic to $\ZZ/4$,
we have\footnote{We use $\pi_\ast$ for the action of a group
element $\pi \in \GG_n$ on the homotopy groups of $E^{h\\G_2^1} \wedge V(0)$. Our apologies.
Here at the prime $2$ there is a classical choice of $\pi \in \GG_2$. See \Cref{rem:g24prime}.}
that $\pi_\ast y_1 = \pm y_1$ so either $(\pi-1)_\ast y_1 = 0$  or $(\pi-1)_\ast y_1 =2y_1$.

If the first case applies, then we can choose a class $y_2 \in \pi_{-1}L_{K(2)}V(0)$ which maps
to $y_1$. By \Cref{we-need-this-in-sec8} we have an isomorphism
\[
H_c^1(\GG_2,E_0V(0)) \cong \ZZ/2 \times \ZZ/2
\]
with generators $\myoneb$ and $\zeta_2$. The class $\zeta_2$ maps to zero in
$H_c^1(\GG_2^1,E_0V(0))$; hence,
we have that $y_2$ is detected by $\myoneb + \epsilon \zeta_2$ where $\epsilon = 0$ or $1$.
Since $\zeta_2$ is a permanent cycle by \Cref{prop:zeta-permanent} 
detecting a homotopy class also called $\zeta_2$, we set $y_0 = y_2+ \epsilon \zeta_2$. 
Then $y_0$ is detected by $\myoneb$. Finally, by \Cref{thm:cohG2V0}, the map
\[
H_c^1(\GG_2,E_0V(0)) \to v_1^{-1}H_c^\ast(\GG_2,E_\ast V(0))
\]
is an injection in degree $1$, so the final statement follows.

If the second case applies we would have that $y_1$ maps to a class of order
$2$ in $\pi_{-2}L_{K(2)}V(0)$ under the boundary map in 
the long exact sequence in homotopy
obtained from the cofiber sequence of (\ref{eqn:piseq}). To rule this out, we use the following result, \Cref{thm:boundbzero}.
\end{proof}

\begin{thm}\label{thm:boundbzero} Under the boundary homomorphism
\[
\partial \colon \pi_{-1}(E^{h\GG_2^1}\wedge V(0)) \longr \pi_{-2}L_{K(2)}V(0)
\]
the class $x_1=\partial(y_1)$ is a class of exact order $4$ detected by
$\myzeta\myoneb \in H_c^2(\GG_2,E_0/2)$.
The class $2x_1$ is detected by $\eta\myzeta\myoneb^2 \in H_c^4(\GG_2,E_2/2)$.
\end{thm}

This is an application of the Geometric Boundary Theorem; see \Cref{rem:whats-up-GBT}. 
The proof will be below, after we have given some background.

Since $E_\ast E^{\GG_2^1} \cong \map_{cts}(\GG_2/\GG_2^1,E_\ast)$
we can apply $E_\ast$ to the cofiber sequence (\ref{eqn:piseq}) and obtain
a short exact sequence of Morava modules
\begin{equation}\label{eqn:piseqalg}
\xymatrix{
0 \to E_\ast V(0) \ar[r]& E_\ast (E^{h\GG_2^1}\wedge V(0)) \ar[r]^-{(\pi-1)_\ast} &
E_\ast (E^{h\GG_2^1}\wedge V(0)) \to 0
}
\end{equation}
and hence a diagram of spectral sequences
\begin{equation}\label{eqn:geomboundary}
\xymatrix{
H_c^s(\GG_2^1,E_t/2) \ar@{=>}[r] \ar[d]_\delta &
\pi_{t-s}(E^{h\GG_2^1} \wedge V(0)) \ar[d]^\partial\\
H_c^{s+1}(\GG_2,E_t/2) \ar@{=>}[r]  &
\pi_{t-s-1}L_{K(2)} V(0)
}
\end{equation}
where $\delta$ is the algebraic connecting map
\[
\xymatrix{
H_c^s(\GG_2^1,E_t/2) \cong H_c^s(\GG_2,E_t(E^{h\GG_2^1}\wedge V(0))) 
\ar[r]^-\delta&  H_c^{s+1}(\GG_2,E_t/2)
}
\]
in the long exact sequence induced by the short exact sequence (\ref{eqn:piseqalg}).

\begin{lem}\label{lem:zetaboundary} Suppose the class $a \in H_c^s(\GG_2^1,E_*/2)$ is the image
of an element $b \in H_c^s(\GG_2,E_*/2)$ under the restriction
\[
H_c^s(\GG_2,E_*/2) \to  H_c^s(\GG^1_2,E_*/2)\ .
\]
Then $\delta(a) = \myzeta_2 b$.
\end{lem}

\begin{proof}  The connecting homomorphism $\delta$ is a homomorphism of $H_c^*(\GG_2,E_*/2)$-modules.
By the definition of $\zeta_2$ (see \Cref{rem:hone}), we have $\delta(1)=\zeta_2$,
and the result follows.
\end{proof}

\begin{proof}[Proof of \Cref{thm:boundbzero}.]
Using 
\Cref{lem:zetaboundary} and the diagram of spectral sequences \eqref{eqn:geomboundary}
we have that $\partial (y_1) \in \pi_{-2}L_{K(2)}V(0)$ is detected by $\myzeta_2\myoneb \in
H_c^2(\GG_2,E_0/2)$.  Since $2y_1$ is detected by $\eta\myoneb^2$, we can again use
\Cref{lem:zetaboundary} and the diagram of spectral sequences
\eqref{eqn:geomboundary} to conclude that $\partial(2y_1) \in \pi_{-2}L_{K(2)}V(0)$ is
detected by
\[
\myzeta_2\eta\myoneb^2 \in H_c^4(\GG_2,E_2/2).
\]
This class is non-zero by \Cref{thm:cohG2V0}.
Thus we need to check that this class cannot be a boundary in the spectral
sequence. Since $E_\ast/2=0$
in odd degrees, the only possible differential is
\[
d_3\colon H_c^1(\GG_2,E_0/2) \to H_c^4(\GG_2,E_2/2).
\]
By \Cref{we-need-this-in-sec8} we have an isomorphism $H_c^1(\GG_2,E_0/2) \cong \ZZ/2 \times \ZZ/2$
with generators $\myzeta_2$ and
$\myoneb$. We know by \Cref{prop:zeta-permanent} that
$\myzeta_2$ is a permanent cycle and we know by  \Cref{prop:b0d3cycle} that $\myoneb$ is a $d_3$-cycle.
Thus $d_3 = 0$ on $H_c^1(\GG_2,E_0/2)$ and the result follows.
\end{proof}


\end{document}